\documentclass[12pt]{amsart}
\usepackage{amsmath}
\usepackage{amssymb}
\usepackage{latexsym}
\usepackage{enumitem}

\newcommand{\Zint}{\mathbb {Z}}    
     
\newcommand{\C}{\mathbb {C}}     
\newcommand{\halmos}{\rule{5pt}{5pt}}

\numberwithin{equation}{section}
\newtheorem{thm}{Theorem}[section]
\newtheorem{prop}[thm]{Proposition}

\newtheorem{cor}[thm]{\bf Corollary}
\theoremstyle{definition} 
\newtheorem{df}[thm]{\bf Definition}
\theoremstyle{remark}

\begin{document}
\title[]{On $q$-middle convolution and generalized $q$-hypergeometric equation}
\author{Yumi Arai}
\address{Department of Mathematics, Ochanomizu University, 2-1-1 Otsuka, Bunkyo-ku, Tokyo 112-8610, Japan}
\email{araiyumi.math@gmail.com}
\subjclass[2020]{33D15,39A13}
\keywords{hypergeometric function, $q$-hypergeometric equation, middle convolution, $q$-integral}
\begin{abstract}
We investigate higher-order $q$-difference equations obtained by successive applications of a $q$-analogue of addition and the $q$-middle convolution.
We construct an $n$-dimensional system of linear $q$-difference equations and derive a single $q$-difference equation 
whose coefficients are polynomials of degree at most $1$ in the independent variable.
Under certain conditions on the parameters, the resulting equation is a generalized $q$-hypergeometric equation of order $n$.
We also obtain a multiple Jackson integral representation of solutions and give sufficient conditions for its convergence.
By suitably choosing the functions appearing in the integrand and specializing the parameters, we obtain $q$-hypergeometric series solutions of type ${}_n \phi_{n-1}$.
\end{abstract}
\maketitle

\section{Introduction}
Middle convolution was originally introduced by Katz \cite{Katz} in the context of rigid local systems.
A linear algebraic formulation applicable to Fuchsian systems of differential equations was subsequently developed by Dettweiler and Reiter \cite{DR1, DR2}.
Consider the Fuchsian system
\begin{equation}
\frac{dY}{dx}=\left( \frac{A_0}{x-t_0}+\frac{A_1}{x-t_1} + \dots + \frac{A_N}{x-t_N} \right) Y, \label{eq:original}
\end{equation}
where $Y$ is an $m$-dimensional column vector and $A_0, A_1, \dots ,A_N$ are $m \times m$ constant matrices.
For a parameter $\lambda \in \mathbb{C}$, the convolution associates with the tuple $(A_0, A_1, \dots , A_N)$ a tuple of $(N+1)m \times (N+1)m$ matrices $(F_0, F_1, \dots , F_N)$ as follows:
\begin{equation}
F_{0} = \left(
\begin{array}{@{\,}c@{\,}ccc@{\,}}
A_{0} + \lambda I_{m} & A_1 &  \cdots & A_N \\
0 & 0 &\cdots & 0 \\
\vdots & \vdots & & \vdots \\
0 & 0 & \cdots & 0 
\end{array}
\right) , \;
F_{1} = \left(
\begin{array}{@{\,}c@{\,}c@{\,}cc@{\,}}
0 & 0 & \cdots & 0 \\
A_{0} & A_1 + \lambda I_{m} & \cdots & A_N \\
\vdots & \vdots & & \vdots \\
0 & 0 & \cdots & 0 
\end{array}
\right) , \dots ,
\label{eq:Fi}
\end{equation}
where $I_m $ is the $m \times m$ identity matrix.
The matrices $F_i$ $(i = 0, 1, \ldots, N)$ leave invariant the subspaces
\begin{equation*}
\mathcal{K} = \left(
\begin{array}{@{\,}c@{\,}}
\mathrm{ker}A_{0} \\
\vdots \\
\mathrm{ker}A_{N}
\end{array}
\right), \quad 
\mathcal{L} = \mathrm{ker}(F_{0} + F_{1} + \cdots + F_{N}) 
\end{equation*}
of $\mathbb{C}^{(N+1) m}$.
Let $\overline{F}_i$ denote the linear transformation induced by $F_i$ on the quotient space $\mathbb{C}^{(N+1) m}/(\mathcal{K} + \mathcal{L})$.
The transformation $(A_0, \dots ,A_N ) \mapsto (\overline{F}_0 , \dots ,\overline{F}_N )$ is called the middle convolution with parameter $\lambda$ and is denoted by $mc_{\lambda }$.
Moreover, the convolution is closely related to Euler's integral transformation of solutions of Fuchsian systems.

A $q$-analogue of the middle convolution was introduced by Sakai and Yamaguchi \cite{SY} for systems of linear $q$-difference equations.
In our previous paper \cite{AT2}, we reformulated their $q$-convolution and $q$-middle convolution and introduced $q$-analogues of addition, which are related to gauge-transformations.
The $q$-middle convolution is also associated with a $q$-integral transformation and can therefore be used to construct $q$-integral representations of solutions to the transformed $q$-difference equations.
In \cite{AT2}, we also investigated and established sufficient conditions for the convergence of the $q$-integral.

The $q$-hypergeometric equation is a second-order linear $q$-difference equation satisfied by the $q$-hypergeometric function ${}_2 \phi_1$.
A standard form of the $q$-hypergeometric equation, also known as Heine's $q$-hypergeometric equation, is given as
\begin{equation}
(x - q)g(x/q) + (abx - c)g(qx) - \{(a + b)x - q - c \}g(x) = 0 . \label{eq:Heine}
\end{equation}
Roques \cite{Ro} introduced the generalized $q$-hypergeometric operator
\begin{equation*}
\mathcal{L}_q (\underline{a};\underline{b};\lambda)
= \prod_{j = 1}^{s} \left(\frac{b_j}{q}\sigma_q - 1 \right) - x\lambda \prod_{i = 1}^{r}(a_i\sigma_q - 1),
\end{equation*}
where
$\underline{a}=(a_1,\ldots,a_r)\in(\mathbb{C}\setminus{\{ 0\}})^r$, $\underline{b}=(b_1,\ldots,b_s)\in(\mathbb{C}\setminus{\{ 0\}})^s$, $\lambda\in\mathbb{C}\setminus{\{ 0\}}$, and $\sigma_qg(x)=g(qx)$.
In the case $r = s = n$, by expanding this operator, the equation
\begin{equation*}
\mathcal{L}_q (\underline{a};\underline{b};\lambda) g(x) = 0 
\end{equation*}
is expressed as an $n$-th order linear $q$-difference equation whose coefficients are polynomials of degree at most $1$ in $x$. Moreover, the coefficients of $g(q^n x)$ and $g(x)$ have nonzero leading and constant terms.
In our previous papers \cite{AT, AT2}, we applied $q$-convolution to a certain specific $q$-difference equation and obtained Heine's $q$-hypergeometric equation and its $q$-integral representations of solutions.
In \cite{AT2}, we applied the $q$-analogue of the addition and the reformulated $q$-middle convolution to Heine's $q$-hypergeometric equation.
As a result, we obtained a generalized $q$-hypergeometric equation of order $3$ of the form
\begin{equation}
(a_3 x - b_3)g(q^3 x) + (a_2 x - b_2)g(q^2 x) + (a_1 x - b_1)g(qx) + (a_0 x - b_0)g(x) = 0 , \label{eq:genqhge3-intro}
\end{equation}
where $a_3b_3a_0b_0 \neq 0$, together with a double $q$-integral representation of its solutions.
We also obtained a $q$-hypergeometric series solution from this integral representation.
It is then natural to ask what happens if we further apply the addition and the $q$-middle convolution repeatedly to this third-order equation, and whether an equation of arbitrary order $n$ and its $q$-integral representations of solutions can be obtained in this way. 

The purpose of the present paper is to investigate this problem.
Starting from the third-order equation obtained in our previous paper, we repeatedly apply compositions of the addition and the $q$-middle convolution.
We show inductively that this procedure gives an $n$-dimensional system of linear $q$-difference equations and explicitly determine its coefficient matrices.
We then eliminate the other unknown functions from the system to obtain a single $q$-difference equation for one of its components.
We show that this single equation has the form
\begin{align}
&(\alpha x - 1) g(q^n x) + (U_{n-1} x + W_{n-1}) g(q^{n-1} x) +  (U_{n-2} x + W_{n-2}) g(q^{n-2} x) \label{eq:gn-Intro}\\
&\quad + \cdots + (U_{2} x + W_{2}) g(q^{2} x) + (U_{1} x + W_{1}) g(q x) + A(U_{0} x - 1) g(x) = 0, \nonumber
\end{align}
where $\alpha \neq 0$ and $A \neq 0$.
Although explicit formulas for the intermediate coefficients $U_j$, $W_j$ (j=1, \ldots, n-1) are not obtained in general, we prove that all coefficients of equation \eqref{eq:gn-Intro} are polynomials of degree at most $1$ in $x$.
Calculations for low-order cases suggest that $U_0$ has a simple explicit form and is nonzero.
Thus, under the assumption $U_0 \neq 0$, equation \eqref{eq:gn-Intro} is a generalized $q$-hypergeometric equation of order $n$.
Explicit single equations for the low-order cases up to order $6$ are given in Appendix \ref{sec:geneq4-6}.

We also investigate a $q$-integral representation of solutions to the generalized $q$-hypergeometric equation of order $n$.
Since the $q$-middle convolution is associated with a $q$-integral transformation, the successive applications of the $q$-middle convolution lead to a multiple Jackson integral representation of solutions.
We show that, by appropriately choosing the functions appearing in the integrand and specializing the parameters, this multiple $q$-integral representation yields $q$-hypergeometric series solutions.
We obtain five new $q$-hypergeometric series solutions to the third-order equation, in addition to the one obtained in our previous paper. 
We further obtain $q$-hypergeometric series solutions for its generalization to arbitrary order $n$.

This paper is organized as follows.
In Section 2, we review the $q$-analogue of addition, the $q$-middle convolution, and the convergence of the associated $q$-integral transformation.
In Section 3, we review the constructions of Heine's $q$-hypergeometric equation and the generalized $q$-hypergeometric equation of order $3$ obtained in our previous paper.
In Section 4, we construct an $n$-dimensional system of $q$-difference equations by repeatedly applying the addition and the $q$-middle convolution, and establish its matrix form by induction.
We also derive the corresponding single $q$-difference equation and show that it is a generalized $q$-hypergeometric equation of order $n$.
In Section 5, we investigate a $q$-integral representation of its solutions and show how it yields $q$-hypergeometric series solutions.

Throughout this paper, we assume that the complex number $q$ satisfies $0 < |q| < 1$. 
For $k\in\mathbb{Z}_{\geq 0}$, the $q$-Pochhammer symbols are defined by
\[
(a;q)_k=\prod_{j=0}^{k-1}(1-aq^j),
\qquad
(a;q)_{\infty} = \prod_{j=0}^{\infty}(1-aq^j),
\]
where $(a;q)_0=1$. We also use the notation
\[
(a_1,\ldots,a_m;q)_k = \prod_{i=1}^{m}(a_i;q)_k,
\qquad k\in\mathbb{Z}_{\geq 0}\cup\{\infty\} .
\]
The $q$-hypergeometric series ${}_n\phi_{n-1}$ is defined by
\[
{}_n\phi_{n-1}
\left(
{a_1,\ldots,a_n\atop b_1,\ldots,b_{n-1}};q,z
\right)
=
\sum_{k=0}^{\infty}
\frac{(a_1,\ldots,a_n;q)_k}
{(q,b_1,\ldots,b_{n-1};q)_k}z^k.
\]

\section{$q$-analogue of addition, $q$-middle convolution, and convergence }
In \cite{AT2}, we reformulated the $q$-middle convolution.
In this section, we review the definitions of the reformulated $q$-convolution and $q$-middle convolution, and the theorems concerning a $q$-integral representation and its convergence.
For the $q$-middle convolution introduced by Sakai and Yamaguchi, we refer the reader to their paper \cite{SY}.

To reformulate the $q$-convolution in a simpler form, we rewrote the associated linear $q$-difference equation as follows:
\begin{equation}
\frac{Y(q x)  - Y(x )}{-x} = \Big[ \frac{B_{0}}{x} + \sum^{N}_{i = 1}\frac{B_{i}}{x -b_{i}} \Big] Y(x) = \Big[ \sum^{N}_{i = 0}\frac{B_{i}}{x -b_{i}} \Big] Y(x)  .
\label{eq:YqxBYx2}
\end{equation}
This equation is equivalent to the linear $q$-difference equation 
\begin{equation}
 Y(qx)= \Big( B_{\infty}+\sum_{i=1}^N \frac{B_i}{1-x/b_i} \Big) Y(x) .
\label{eq:YqxSY}
\end{equation}
used by Sakai and Yamaguchi under the conditions $B_0 = I_m - B_\infty - B_{1} - \dots -B_N$ and $b_0=0$.

We introduce a $q$-analogue of addition.
If the function $Y(x) \in \mathbb{C}^m$ satisfies $Y(qx) =B(x)Y(x) $, then the function 
\begin{equation}
Z(x)= x^{\mu } Y(x)
\label{eq:ZY}
\end{equation}
satisfies $Z(qx) = q^{\mu } B(x)Z(x) $.
In particular, if $Y(x)$ is a solution to equation (\ref{eq:YqxBYx2}), then the function $Z(x)$ in equation (\ref{eq:ZY}) satisfies
\begin{equation*}
\frac{Z(q x)  - Z(x )}{-x} = \Big[ \frac{(1-q^{\mu} )I_m + q^{\mu }B_0}{x} + \sum^{N}_{i = 1}\frac{q^{\mu } B_{i}}{x -b_{i}} \Big] Z(x) .
\end{equation*}
Thus, we define the addition $add _{\mu } $ by
\begin{equation*}
add _{\mu } : (B_0, B_1, \ldots, B_N) \mapsto ((1-q^{\mu} )I_m + q^{\mu }B_0 , q^{\mu }B_1, \ldots, q^{\mu }B_N) .
\end{equation*}
Let $c^{SY}_{\lambda}$ denote the $q$-convolution introduced by Sakai and Yamaguchi, and let $c^{q}_{\lambda}$ denote the reformulated $q$-convolution.
We reformulated the $q$-convolution by
\begin{equation}
c^{q}_{\lambda} = add_{-\lambda} \circ c^{SY}_{\lambda} .  
\end{equation}
\begin{df}[Reformulation of $q$-convolution \cite{AT2}] \label{def:tqc}
Let $ (  B_{0},  B_{1} ,\dots ,B_N ) $ be the tuple of $m\times m $ matrices and $\lambda \in \mathbb{C} $.
We define the tuple $( G_0,  G_1, \dots , G_N )$ of $(N+1)m \times (N+1)m$ matrices as follows:
\begin{align} \label{eq:bG} 
    G_0 &= \left( 
    \begin{array}{ccccc}
       q^{-\lambda } B_0 + (1-q^{-\lambda})I_m & q^{-\lambda }B_1 & \cdots &  \cdots & q^{-\lambda } B_N \\[5pt]
        0 & 0  & \cdots &  \cdots & 0 \\
        \vdots &  \vdots  & {} & {} & \vdots \\
    \end{array}
\right)  \\
& \qquad \qquad \qquad \vdots  \nonumber \\
    G_i &= \left( 
    \begin{array}{ccccc}
        {} & {} & O & {} & {} \\[5pt]
       q^{-\lambda } B_0 & \cdots & q^{-\lambda }B_i + (1-q^{-\lambda})I_m & \cdots & q^{-\lambda } B_N \\[5pt]
        {} & {} & O & {} & {}
    \end{array}
\right)
    {\scriptstyle (i+1)}  \nonumber  \\
& \qquad \qquad \qquad \vdots  \nonumber \\
    G_N &= \left( 
    \begin{array}{ccccc}
        \vdots &   {} & {} & \vdots  & \vdots \\[5pt]
        0 & \cdots  & \cdots & 0 & 0 \\[5pt]
       q^{-\lambda } B_0 & \cdots &  \cdots & q^{-\lambda }B_{N-1} & q^{-\lambda } B_N + (1-q^{-\lambda})I_m 
    \end{array}
\right) .\nonumber
\end{align}
The correspondence $c ^q _\lambda : ( B_{0},  B_{1} ,\dots ,B_N ) \mapsto ( G_0,  G_1, \dots , G_N )$ is called the $q$-convolution.
\end{df}
The $q$-convolution in Definition \ref{def:tqc} is related to a $q$-integral transformation on the linear $q$-difference equations in the form of equation (\ref{eq:YqxBYx2}).
Let $\xi \in \mathbb{C} \setminus \{ 0 \}$.
The Jackson integral is defined as the infinite sum
\begin{equation*}
 \int^{\xi \infty }_{0}f(s) \, d_{q}s = (1-q)\sum^{\infty}_{n=-\infty}q^{n} \xi f(q^n \xi ).
\end{equation*}
Let $K_{\lambda}(x, s) $ be a function which satisfies
\begin{equation}
q^{\lambda } K_{\lambda}(qx, s) = K_{\lambda}(x, s/q)  = \frac{ x - q^{\lambda } s}{x - s} K_{\lambda}(x, s) . \label{eq:tPlambda}
\end{equation}
Then, the functions
\begin{equation}
K^{(1)}_{\lambda}(x, s) =  x^{-\lambda } \frac{(q^{\lambda+1}s/x;q)_{\infty}}{(qs/x;q)_{\infty}} , \quad K^{(2)}_{\lambda}(x, s) =  s^{-\lambda } \frac{(x/s;q)_{\infty}}{(q^{-\lambda }x/s;q)_{\infty}}
 \label{eq:tP1tP2}
\end{equation}
satisfy equation (\ref{eq:tPlambda}).

The following theorem indicates a correspondence of the solutions to $q$-difference equations associated with the $q$-convolution.
\begin{thm}$($c.f. \cite[Theorem 2.1]{SY}, \cite[Theorem 3.2]{AT2}$)$ \label{thm:qcint} 
Set $b_0=0$.
Let $Y(x)$ be a solution to
\begin{equation}
\frac{Y(q x)  - Y(x )}{-x} = \Big[ \sum^{N}_{i = 0}\frac{B_{i}}{x -b_{i}} \Big] Y(x) ,
\label{eq:YqxBYx}
\end{equation}
and set 
\begin{equation}
 \widetilde{Y}_{i}(x) = \int^{\xi \infty}_{0}\frac{K_{\lambda}(x, s)}{s-b_{i}}Y(s) \, d_{q}s, \; (i=0,\dots ,N), \; \widetilde{Y}(x) = \begin{pmatrix} \widetilde{Y}_{0}(x) \\ \vdots \\ \widetilde{Y}_{N}(x) \end{pmatrix} . \label{eq:qcintadd}
\end{equation}
Assume that the variable $\xi $ is independent of the variable $x$ or it is proportional to $x$ $($i.e.~$\xi =Ax$ where $A $ is independent of $x$$)$.
If every element of $\widetilde{Y}_{i} (x) $ converges for $i=0,1 , \dots ,N$ and 
\begin{equation}
\lim _{K \to -\infty} K_{\lambda}(x, q^{K-1} \xi)Y(q^{K} \xi) =0, \; \lim _{L \to +\infty } K_{\lambda}(x, q^{L} \xi)Y(q^{L+1} \xi) =0,
\label{eq:tPKLlim0}
\end{equation}
then the function $\widetilde{Y}(x)$ satisfies
\begin{equation*}
\frac{\widetilde{Y}(q x)  - \widetilde{Y}(x )}{-x} = \Big[ \sum^{N}_{i = 0}\frac{G_{i}}{x -b_{i}} \Big] \widetilde{Y}(x) .
\end{equation*}
\end{thm}
For the proof of this theorem, see \cite{AT2}. 
The following theorem gives sufficient conditions for the convergence required in Theorem \ref{thm:qcint}.
\begin{thm}$($\cite[Theorem 3.8]{AT2}$)$ \label{thm:suffcondconv}
Set $b_0=0$.
Let $Y(x)$ be a solution to
\begin{equation*}
\frac{Y(q x)  - Y(x )}{-x} = \Big[ \sum^{N}_{i = 0}\frac{B_{i}}{x -b_{i}} \Big] Y(x) .
\end{equation*}
Assume that the integrand of $\widetilde{Y}_{i} (x) $ in equation (\ref{eq:qcintadd}) is well-defined $($i.e.~does not diverge$)$ about $s= q^n \xi $ for any $i \in \{ 0 , 1, \dots , N \}$ and $n \in \mathbb{Z} $.
If the absolute value of any eigenvalue of $I_m - B_0$ is strictly less than $1$ and the absolute value of any eigenvalue of $I_m - B_0 -B_1 - \dots - B_N$ is strictly more than $|q^{\lambda }|$, then the function $\widetilde{Y}(x)$ in equation (\ref{eq:qcintadd}) converges and it satisfies
\begin{equation*}
\frac{\widetilde{Y}(q x)  - \widetilde{Y}(x )}{-x} = \Big[ \sum^{N}_{i = 0}\frac{G_{i}}{x -b_{i}} \Big] \widetilde{Y}(x) .
\end{equation*}
\end{thm}
See \cite{AT2} for the proof of this theorem. 

The definition of the $q$-middle convolution based on the $q$-convolution is defined in essentially the same way as that of Sakai and Yamaguchi \cite{SY}, which is a $q$-analogue of the one by Dettweiler and Reiter \cite{DR1}.
\begin{df}[$q$-middle convolution]\label{def:qmc}
Define the subspaces $\mathcal{K}$ and $\mathcal{L} $ of $(\mathbb{C}^m )^{N+1}$ as follows:
\begin{equation}
    \mathcal{K} = 
    \left( \hspace{-.5em}
        \begin{array}{c}
            \mathrm{ker}B_0 \\
            \vdots \\
            \mathrm{ker}B_N
        \end{array}
    \hspace{-.5em} \right),  \quad
    \mathcal{L} = \mathrm{ker}(G_0 +G_1 + \dots + G_N )  .
\label{eq:qKL}
\end{equation}
Then, the spaces $\mathcal{K}$ and $\mathcal{L} $ are invariant under the action of $G_k$ for $k=0 ,1, \dots ,N $.
Let $\overline{G}_k$ $(k=0 ,1, \dots ,N )$ denote the matrix induced from the action of $G_k$ on the quotient space $(\mathbb{C}^m )^{N+1}/(\mathcal{K} + \mathcal{L})$ .
The $q$-middle convolution $mc ^q _\lambda$ is defined by the correspondence $ ( B_{0}, B_{1} ,\dots , B_N ) \mapsto ( \overline{G}_{0}, \overline{G}_{1} ,\dots ,\overline{G}_N ) $.
\end{df}

\section{$q$-hypergeometric equation and generalization} \label{sec:qhypgeomgen}
In this section, we review the constructions of Heine's $q$-hypergeometric equation and the generalized $q$-hypergeometric equation of order $3$ through the reformulated $q$-middle convolution presented in \cite{AT2}.
This review serves as a starting point for the construction of the generalized $q$-hypergeometric equation of order $n$ in the next section.

\subsection{Heine's $q$-hypergeometric equation} \label{sec:qconvex1}$ $

Let $\mu_0 \in \mathbb{C}$, $\alpha, \beta  \in \mathbb{C} \setminus \{0\}$ with $\alpha \neq \beta$, and set $y(x)=x^{\mu_0}(\alpha x;q)_{\infty}/(\beta x;q)_{\infty}$.
Then, $y(x)$ satisfies the first-order linear $q$-difference equation
\begin{equation} \label{eq:qeab}
\frac{y(qx) - y(x)}{-x} = \bigg[ \frac{B_0^{[1]}}{x} + \frac{B_1^{[1]}}{x-1/\alpha} \bigg] y(x) ,\ B_0^{[1]} = 1 - q^{\mu_0}, \ B_1^{[1]} = q^{\mu_0} \left(1 - \frac{\beta}{\alpha} \right) .
\end{equation}
The function $y(x)= x^{\widetilde{\mu }} (q/( \beta x) ;q)_{\infty}/(q/(\alpha x) ;q)_{\infty} $ also satisfies equation (\ref{eq:qeab}) with the condition $q^{\widetilde{\mu }} \alpha /\beta = q^{\mu_0} $ .
We apply $q$-convolution $c^q_{\lambda_0}$ to the pair of $1 \times 1$ matrices $(B_0^{[1]}, B_1^{[1]})$, and set $c^q_{\lambda_0}(B_0^{[1]}, B_1^{[1]}) = (G_0^{[2]}, G_1^{[2]})$.
The matrices $G_0^{[2]}$ and $G_1^{[2]}$ are given by 
\begin{align}
& G_0^{[2]}  = \begin{pmatrix}
                          q^{-\lambda_0}B_0^{[1]} + 1 -q^{-\lambda_0} & q^{-\lambda_0}B_1^{[1]} \\
                          0                                           & 0
                         \end{pmatrix}
                       = \begin{pmatrix}
                          1 - q^{\mu_0 - \lambda_0} & q^{\mu_0 - \lambda_0}(1 - \beta/\alpha) \\
                          0                                 & 0
	                         \end{pmatrix},  \label{eq:G0G1qhyp0} \\
& G_1^{[2]} = \begin{pmatrix}
                          0                        & 0 \\
                          q^{-\lambda_0}B_0^{[1]}  & q^{-\lambda_0}B_1^{[1]} + 1 -q^{-\lambda_0}
                         \end{pmatrix}
                       = \begin{pmatrix}
                          0                & 0 \\
                          q^{-\lambda_0} - q^{\mu_0 - \lambda_0}  & q^{\mu_0 - \lambda_0}(1 - \beta/\alpha) + 1 -q^{-\lambda_0}
                         \end{pmatrix}. \nonumber
\end{align}
Note that $\det ( G_0^{[2]} + G_1^{[2]} )= ( 1 -q^{-\lambda_0})(1-  q^{\mu_0 - \lambda_0} \beta/\alpha )$, and the $q$-middle convolution coincides with the $q$-convolution (i.e.~$ \mathcal{K} + \mathcal{L} =\{ 0 \}$), if $\lambda_0 \neq 0 \neq \mu_0 $, $\alpha \neq \beta $ and $q^{\mu_0 - \lambda_0} \beta/\alpha  \neq 1$.
The corresponding $q$-difference equation can be written as 
\begin{align}  \label{eq:qhg-sta} 
\frac{\widetilde{Y}(qx)-\widetilde{Y}(x)}{-x} = \Big[ \frac{G_0^{[2]}}{x} +\frac{G_1^{[2]}}{x-1/\alpha } \Big] \widetilde{Y}(x) , 
\quad \widetilde{Y}(x)=\begin{pmatrix}
                 \widetilde{y}_0(x)\\
                 \widetilde{y}_1(x)
                \end{pmatrix} .
\end{align}
Then, the function $\widetilde{y}_0(x)$ satisfies the single second-order equation
\begin{equation} \label{eq:ytil}
(q^{-\lambda_0}\beta x-q)\widetilde{y}_0(x/q) + q^{\lambda_0 - \mu_0}(\alpha x-q )\widetilde{y}_0(qx) -\{ (q^{-\mu_0}\alpha+\beta)x- q - q^{\lambda_0 - \mu_0 + 1}\} \widetilde{y}_0(x)=0,
\end{equation}
and equation (\ref{eq:ytil}) corresponds to Heine's $q$-hypergeometric equation (\ref{eq:Heine}) by setting the parameters appropriately.

By the $q$-integral transformation associated with the $q$-convolution, the $q$-integral representation of solutions to equation (\ref{eq:qhg-sta}) is obtained as
\begin{equation} \label{sol:Ytil-sta}
\widetilde{Y}(x)
= \begin{pmatrix}
   \widetilde{y}_0(x) \\
   \widetilde{y}_1(x)
  \end{pmatrix}
= 
\begin{pmatrix}
  \displaystyle \int_0^{\xi_0 \infty} \frac{K_{\lambda_0}(x,s_0)}{s_0} y(s_0) d_q s_0 \\
  \displaystyle \int_0^{\xi_0 \infty} \frac{K_{\lambda_0}(x,s_0)}{s_0 - 1/\alpha} y(s_0) d_q s_0
\end{pmatrix} ,
\end{equation}
if the assumption of Theorem \ref{thm:qcint} holds.
We investigate a sufficient condition that the $q$-integral representation converges and satisfies equation (\ref{eq:qhg-sta}) by applying Theorem \ref{thm:suffcondconv}.
In this case, the matrices $B_0^{[1]}$ and $B_1^{[1]}$ are scalars in equation (\ref{eq:qeab}), and we have 
\begin{align}
& I_1 - B_0^{[1]} =  q^{\mu_0}, \label{eq:I1B0}\\
& I_1 - B_0^{[1]} - B_1^{[1]} =  q^{\mu_0} \beta /\alpha . \label{eq:I1B0B1}
\end{align}
Therefore, under the condition that the functions $K_{\lambda_0}(x,s_0)y(s_0)/ s_0  $ and $K_{\lambda_0}(x, s_0) \\ y(s_0) /(s_0 - 1/\alpha ) $ are well-defined about $s_0 = q^{m_0} \xi_0 $ for any $m_0 \in \Zint $, a sufficient condition that the function in (\ref{sol:Ytil-sta}) converges and satisfies equation (\ref{eq:qhg-sta}) is written as $|q^{\mu_0 } |<1$ and $ |q^{ \mu_0} \beta / \alpha | >  |q^{\lambda_0 } | $.

In the case $K_{\lambda_0}(x, s_0) = K^{(1)}_{\lambda_0}(x, s_0) $ (see equation (\ref{eq:tP1tP2})) and $y(s_0)=s^{\mu_0}(\alpha s_0;q)_{\infty}/ \\ (\beta s_0;q)_{\infty} $, 
the function $\widetilde{y}_0(x)$ in equation (\ref{sol:Ytil-sta}) is written as   
\begin{equation} \label{eq:y0P1}
\widetilde{y}_0(x) = (1 - q) x^{-\lambda_0} \sum_{m_0 = -\infty}^{\infty} (q^{m_0} \xi_0)^{\mu_0} \frac{(q^{\lambda_0 + m_0 + 1} \xi_0/x, q^{m_0} \xi_0 \alpha ;q)_{\infty}}{(q^{m_0 + 1} \xi_0/x, q^{m_0} \xi_0 \beta ;q)_{\infty}} .
\end{equation}
By specializing the parameter $\xi_0$, it can further be expressed as a unilateral $q$-hypergeometric series ${_2 \phi_1}$. 
We set $\xi_0 = 1/\alpha $ in equation (\ref{eq:y0P1}).
As shown in \cite{AT}, it follows from $(q^{\ell};q)_{\infty} = 0$ for $\ell \in \mathbb{Z}_{< 0}$ that
\begin{align}
\widetilde{y}_0(x) |_{\xi_0 = 1/\alpha } 
= (1-q)\frac{q^{\mu_0}}{\alpha^{\mu_0}}\frac{(q^{\lambda_0 + 2}/(\alpha x),q;q)_{\infty}}{(q^2/(\alpha x),q\beta/\alpha;q)_{\infty}} x^{-\lambda_0 } {}_2\phi_1 \left( \begin{array}{c} q^2/(\alpha x), q\beta/\alpha\\ q^{\lambda_0 + 2}/(\alpha x) \end{array} ;q,q^{\mu_0} \right) .
\end{align}
By applying Heine's transformation formula in \cite[(1.4.1)]{GR}
\[
{}_2\phi_1 (a,b;c;q,z)= \frac{(b,az;q)_{\infty}}{(c,z;q)_{\infty}} {}_2\phi_1 (c/b,z;az; q,b), \quad |z|<1, |b|<1,
\]
we can rewrite it as
\begin{align} \label{eq:y0al} 
\widetilde{y}_0(x) |_{\xi_0 = 1/\alpha } 
=  (1-q) q^{\mu_0} \alpha^{-\mu_0} x^{-\lambda_0} \frac{(q, q^{\mu_0 + 1}\beta/\alpha ;q)_{\infty}}{(q\beta/\alpha,q^{\mu_0} ;q)_{\infty}}
{}_2\phi_1 \left( \begin{array}{c} q^{\lambda_0},q^{\mu_0}\\
q^{\mu_0 + 1}\beta/\alpha \end{array} ;q, q^2/(\alpha x) \right),
\end{align}
which was obtained in version 2 of \cite{Ar}.
We also have
\begin{align}
 \widetilde{y}_0(x) |_{\xi_0 = q^{-\lambda_0}x } & = (1-q)q^{-\lambda_0\mu_0}x^{\mu_0 -\lambda_0 } \frac{(q^{-\lambda_0}\alpha x,q;q)_{\infty}}{(q^{-\lambda_0}\beta x,q^{1-\lambda_0};q)_{\infty}} {}_2\phi_1 \left( \begin{array}{c} q^{-\lambda_0}\beta x,q^{1-\lambda_0}\\ q^{-\lambda_0}\alpha x \end{array} ;q,q^{\mu_0} \right) \label{eq:y0la} \\
& = (1-q) q^{-\lambda_0\mu_0} x^{\mu_0 - \lambda_0} \frac{(q,q^{-\lambda_0 + \mu_0 + 1};q)_{\infty}}{(q^{-\lambda_0 + 1},q^{\mu_0};q)_{\infty}}
{}_2\phi_1 \left( \begin{array}{c} \alpha/\beta, q^{\mu_0}\\
q^{-\lambda_0 + \mu_0 + 1} \end{array} ;q,q^{-\lambda_0}\beta x \right) . \nonumber
\end{align}
The functions in equations (\ref{eq:y0al}) and (\ref{eq:y0la}) satisfy the $q$-hypergeometric equation in (\ref{eq:ytil}).

In the case $K_{\lambda_0}(x, s_0) = K^{(2)}_{\lambda_0}(x, s_0) $ and $y(s_0)= s_0^{\widetilde{\mu}} (q/( \beta s_0) ;q)_{\infty}/(q/(\alpha s_0) ;q)_{\infty}  $, 
the function $\widetilde{y}_0(x)$ is written as   
\begin{equation}  \label{eq:y0P2}
\widetilde{y}_0(x) = (1-q) \sum_{m_0 = -\infty}^{\infty}(q^{m_0}\xi_0)^{\widetilde{\mu } -\lambda_0 }  \frac{(x q^{-m_0} / \xi_0 , q^{1 - m_0} / (\beta \xi_0 ) ; q)_{\infty}}{(x q^{-\lambda_0 - m_0} / \xi_0 , q^{1 - m_0} / (\alpha \xi_0 ) ; q)_{\infty}} .
\end{equation}
We set $\xi_0 = 1/\beta $ or $\xi_0 = x $ in equation (\ref{eq:y0P2}).
As established in version 2 of \cite{Ar} and \cite{AT}, 
\begin{align}
& \widetilde{y}_0(x) |_{\xi_0 = 1/\beta } 
 = (1-q) \beta^{\lambda_0 - \widetilde{\mu }} \frac{(\beta x, q ;q)_{\infty}}{(q^{-\lambda_0}\beta x, q\beta/\alpha ;q)_{\infty}}
{}_2\phi_1 \left( \begin{array}{c} q^{-\lambda_0}\beta x,q\beta/\alpha\\
\beta x \end{array} ;q,q^{\lambda_0 - \mu_0}\frac{\alpha}{\beta} \right) \label{eq:y0be} \\
& \qquad \qquad  = (1-q) \beta^{\lambda_0 - \widetilde{\mu }} \frac{(q, q^{\lambda_0 - \mu_0 + 1} ;q)_{\infty}}{(q\beta/\alpha, q^{\lambda_0 -\mu_0}\alpha/\beta ;q)_{\infty}}
{}_2\phi_1 \left( \begin{array}{c} q^{\lambda_0},q^{\lambda_0 - \mu_0}\alpha/\beta\\
 q^{\lambda_0 - \mu_0 + 1} \end{array} ;q,q^{-\lambda_0}\beta x \right), \nonumber \\
& \widetilde{y}_0(x) |_{\xi_0 = x } 
 = (1-q) q^{\lambda_0 - \mu_0} \frac{\alpha}{\beta} x^{\widetilde{\mu } - \lambda_0}\frac{(q^2/(\beta x), q ;q)_{\infty}}{(q^2/(\alpha x),q^{-\lambda_0 + 1} ;q)_{\infty}} \nonumber \\
&\hspace{200pt} \cdot {}_2\phi_1 \left( \begin{array}{c} q^{-\lambda_0 + 1},q^2/(\alpha x)\\
q^2/(\beta x) \end{array} ;q,q^{\lambda_0 - \mu_0}\frac{\alpha}{\beta} \right) \nonumber \\
& \qquad \quad = (1-q) q^{\lambda_0 - \mu_0} \frac{\alpha}{\beta} x^{\widetilde{\mu } - \lambda_0} \frac{(q, q^{-\mu_0 + 1}\alpha/\beta ;q)_{\infty}}{(q^{-\lambda_0 + 1}, q^{\lambda_0 - \mu_0}\alpha/\beta ;q)_{\infty}} \nonumber \\
&\hspace{200pt} \cdot {}_2\phi_1 \left( \begin{array}{c} \alpha/\beta, q^{\lambda_0 - \mu_0}\alpha/\beta\\
q^{-\mu_0 + 1}\alpha/\beta \end{array} ;q, q^2/(\alpha x) \right) \nonumber .
\end{align}
The functions in equation (\ref{eq:y0be}) satisfy the $q$-hypergeometric equation in (\ref{eq:ytil}).

\subsection{Generalized $q$-hypergeometric equation of order $3$} \label{sec:qmcex} $ $

To obtain the generalized $q$-hypergeometric equation of order $3$, we use an addition and a $q$-middle convolution.
We apply the addition $add _{\mu_1 } $ to the pair $(G_0^{[2]}, G_1^{[2]}) $ in equation (\ref{eq:G0G1qhyp0}).
Then, we have the pair of $2 \times 2$ matrices $(B^{[2]}_0 , B^{[2]}_1  )$, where
\begin{align}
& B^{[2]}_0   = \begin{pmatrix}
   1 -q^{\mu_0 + \mu_1 - \lambda_0 }  &  q^{\mu_0 + \mu_1 - \lambda_0 }(1 - \beta/\alpha) \\
       0          & 1 - q^{\mu_1 }   \\
      \end{pmatrix}, \label{eq:B'0B'1qhyp0}\\
& B^{[2]}_1  = \begin{pmatrix}
       0          &  0\\
      q^{\mu_1  - \lambda_0 } ( 1 -q^{\mu_0 } )  &  q^{\mu_1 } ( q^{\mu_0 -\lambda_0 }(1 - \beta/\alpha) + 1 - q^{-\lambda_0} ) 
\end{pmatrix} . \nonumber
\end{align}
The mapping $(G_0^{[2]}, G_1^{[2]}) \mapsto (B^{[2]}_0 , B^{[2]}_1 ) $ corresponds to a gauge-transformation of the $q$-difference equations.
Let $\widetilde{Y}(x)$ be a solution to equation (\ref{eq:qhg-sta}), and set
\begin{equation} \label{eq:Ygxxmu'}
  V^{[2]}(x) = x^{\mu_1}\widetilde{Y}(x).
\end{equation}
The function $ V^{[2]}(x) $ satisfies
\begin{align} \label{eq:B'0B'1m}
&\frac{V^{[2]}(qx) - V^{[2]}(x)}{-x} = \bigg[ \frac{B^{[2]}_0}{x} + \frac{B^{[2]}_1}{x-1/\alpha} \bigg] V^{[2]}(x).
\end{align}
It follows from equations (\ref{eq:Ygxxmu'}) and (\ref{sol:Ytil-sta}) that the function
 \begin{align}
V^{[2]}(x) &= 
\begin{pmatrix}
  \displaystyle  x^{\mu_1} \int_0^{\xi_0 \infty} \frac{K_{\lambda_0}(x,s_0)}{s_0} y(s_0) d_q s_0 \\
  \displaystyle  x^{\mu_1} \int_0^{\xi_0 \infty} \frac{K_{\lambda_0}(x,s_0)}{s_0 - 1/\alpha} y(s_0) d_q s_0
\end{pmatrix}  
\end{align}
satisfies equation (\ref{eq:B'0B'1m}), if the function $y(x)$ is a solution to equation (\ref{eq:qeab}), $|q^{\mu_0}| < 1$, $ |q^{\mu_0} \beta / \alpha | > |q^{\lambda_0}| $ and the functions $K_{\lambda_0}(x,s_0)y(s_0) / s_0  $ and $K_{\lambda_0}(x,s_0)y(s_0) /(s_0 - 1/\alpha ) $ are well-defined about $s_0 = q^{m_0} \xi_0 $ for any $m_0 \in \Zint $.

We apply the $q$-convolution $c^q_{\lambda_1}$ to $(B^{[2]}_0, B^{[2]}_1)$.
We obtain $c^q_{\lambda_1} (B^{[2]}_0, B^{[2]}_1)= (G^{[4]}_0, G^{[4]}_1)$, where $G^{[4]}_0$ and $G^{[4]}_1$ are the $4 \times 4$ matrices described as
\begin{align}
& G^{[4]}_0
= \begin{pmatrix}
     q^{-\lambda_1}B^{[2]}_0 + (1 - q^{-\lambda_1})I_2 & q^{-\lambda_1}B^{[2]}_1 \\
     O                                          & O
\end{pmatrix}, \\ 
& G^{[4]}_1
= \begin{pmatrix}
     O      & O    \\
     q^{-\lambda_1}B^{[2]}_0   & q^{-\lambda_1}B^{[2]}_1 + (1-q^{-\lambda_1}) I_2
\end{pmatrix}. \nonumber
\end{align}
We investigate a sufficient condition for the parameters that the $q$-difference equation 
\begin{align} \label{eq:G'0G'1m}
&\frac{\widetilde{V}^{[4]}(qx) - \widetilde{V}^{[4]}(x)}{-x} = \bigg[ \frac{G^{[4]}_0}{x} + \frac{G^{[4]}_1}{x-1/\alpha} \bigg] \widetilde{V}^{[4]}(x)
\end{align}
admits a $q$-integral representation of solutions by applying Theorem \ref{thm:suffcondconv} for $c^q_{\lambda_1} (B^{[2]}_0, B^{[2]}_1)\\= (G^{[4]}_0, G^{[4]}_1)$.
The eigenvalues of $I_2 - B^{[2]}_0 $ are $ q^{\mu_1}$, $ q^{\mu_0 + \mu_1 - \lambda_0} $, and the eigenvalues of $I_2 - B^{[2]}_0 - B^{[2]}_1 $ are $ q^{\mu_1 - \lambda_0 }$, $ q^{\mu_0 + \mu_1 - \lambda_0} \beta/ \alpha $.
The formal $q$-integral representation of solutions to the equation obtained after applying $q$-convolution $c_{\lambda_1}^q$ is given by
\begin{align} \label{eq:tYg}
\widetilde{V}^{[4]}(x) 
& = \begin{pmatrix}
  \displaystyle \int_0^{\xi_1 \infty} \frac{K_{\lambda_1}(x,s_1)}{s_1} V^{[2]} (s_1) d_q s_1 \\
  \displaystyle \int_0^{\xi_1 \infty} \frac{K_{\lambda_1}(x,s_1)}{s_1 - 1/\alpha} V^{[2]} (s_1) d_q s_1 
  \end{pmatrix} \\
& = \begin{pmatrix}
  \displaystyle \int_0^{\xi_1 \infty} \frac{K_{\lambda_1}(x,s_1)}{s_1} s_1^{\mu_1 } \int_0^{\xi_0 \infty} \frac{K_{\lambda_0}(s_1,s_0)}{s_0} y(s_0) d_q s_0 d_q s_1 \\
  \displaystyle \int_0^{\xi_1 \infty} \frac{K_{\lambda_1}(x,s_1)}{s_1} s_1^{\mu_1 } \int_0^{\xi_0 \infty} \frac{K_{\lambda_0}(s_1,s_0)}{s_0 - 1/\alpha} y(s_0) d_q s_0 d_q s_1 \\
   \displaystyle \int_0^{\xi_1 \infty} \frac{K_{\lambda_1}(x,s_1)}{s_1 - 1/\alpha} s_1^{\mu_1 } \int_0^{\xi_0 \infty} \frac{K_{\lambda_0}(s_1,s_0)}{s_0} y(s_0) d_q s_0 d_q s_1 \\
  \displaystyle \int_0^{\xi_1 \infty} \frac{K_{\lambda_1}(x,s_1)}{s_1 - 1/\alpha} s_1^{\mu_1 } \int_0^{\xi_0 \infty} \frac{K_{\lambda_0}(s_1,s_0)}{s_0 - 1/\alpha} y(s_0) d_q s_0 d_q s_1
  \end{pmatrix} , \nonumber
\end{align}
where $K_{\lambda_1 }(x,s_1) $ and $K_{\lambda_0 }(s_1,s_0)$ are functions which satisfy equation (\ref{eq:tPlambda}), and $y(s_0)$ is a function which satisfies equation (\ref{eq:qeab}).
Then, we obtain the following proposition by Theorem \ref{thm:suffcondconv}.
\begin{prop} \label{prop:ghge3conv}
If the parameters satisfy
\begin{align} \label{eq:prop41}
& |q^{\mu_0 } |<1,\ |q^{\mu_0} \beta/\alpha| > |q^{\lambda_0 } |, \\
& |q^{\mu_1 } |<1, \  |q^{\mu_0 + \mu_1 - \lambda_0 } |<1, \   |q^{ \mu_1 - \lambda_0 } | >  |q^{\lambda_1 } | ,  \ |q^{\mu_0 + \mu_1 - \lambda_0} \beta/\alpha| > |q^{\lambda_1 } | \nonumber
\end{align}
and the integrands in equation (\ref{eq:tYg}) are well-defined about $(s_1,s_0) = (q^{m_1} \xi_1, q^{m_0} \xi_0  )$ for any $m_1, m _0 \in \Zint $, then the function $\widetilde{V}^{[4]}(x)$ in equation (\ref{eq:tYg}) converges and satisfies equation (\ref{eq:G'0G'1m}).
\end{prop}

We consider the $q$-middle convolution $mc^q_{\lambda_1}$ to $(B^{[2]}_0, B^{[2]}_1)$.
Recall that the $q$-middle convolution is obtained as the pair of the operators on the quotient space $\C^4/(\mathcal{K} + \mathcal{L}) $.
We assume $q^{\mu_1  - \lambda_0 } ( 1 -q^{\mu_0 }) \neq 0$, or equivalently, $\mu_0 \neq 0$.
The vector ${}^t(q^{\mu_0}(1 - \beta/\alpha) + q^{\lambda_0} - 1, q^{\mu_0} - 1)$ belongs to the space $\ker (B^{[2]}_1)$.
Moreover, if $\det (B^{[2]}_0) = (1 - q^{\mu_0 + \mu_1 - \lambda_0})(1 - q^{\mu_1}) \neq 0$, or equivalently, $\mu_1 \neq 0$ and $\mu_0 + \mu_1 - \lambda_0 \neq 0 $, we obtain $\dim \mathcal{K}=1$.
We have $\dim \mathcal{L} = 0$, if $\det ( G^{[4]}_0 + G^{[4]}_1 ) = (1- q^{-\lambda_1} )^2 (1-q^{\mu_1 - (\lambda_0 + \lambda_1)}) (1-q^{\mu_0 + \mu_1 - (\lambda_0 + \lambda_1)} \beta /\alpha ) \neq 0$, namely $\lambda_1 \neq 0$, $\mu_1 - (\lambda_0 + \lambda_1) \neq 0$ and $q^{\mu_0 + \mu_1 - (\lambda_0 + \lambda_1)} \beta /\alpha \neq 1 $.
Continuing the argument under these conditions, we have $\dim (\C^4/(\mathcal{K} + \mathcal{L}) ) =3 $.
To obtain matrix representations of $mc^q_{\lambda_1} (B^{[2]}_0, B^{[2]}_1) $, we consider a simultaneous similarity transformation for the matrices $G^{[4]}_0 $ and $G^{[4]}_1$.
Set
\begin{equation} \label{mtP-1}
P^{[4]} = \begin{pmatrix}
      1 & 0 & 0 & 0 \\
      0 & 1 & 0 & 0 \\
      0 & 0 & 1 & f_0  \\
      0 & 0 & 0 & q^{\mu_0} - 1
\end{pmatrix}, \; f_0 = q^{\mu_0}(1 - \beta/\alpha) + q^{\lambda_0} - 1 .
\end{equation}
Then, $P^{[4]}$ is invertible and we have
\begin{align} \label{eq:P-1G3P}
& (P^{[4]})^{-1}G^{[4]}_0 P^{[4]} \\
&\quad = \begin{pmatrix}
    1-q^{\mu_0 + \mu_1 - (\lambda_0 + \lambda_1)} & q^{\mu_0 + \mu_1 - (\lambda_0 + \lambda_1)}(1 - \beta/\alpha) & 0        & 0 \\
    0                               & 1-q^{\mu_1 - \lambda_1}          & q^{\mu_1 - (\lambda_0 + \lambda_1)} (1 - q^{\mu_0 })& 0 \\
    0                                & 0                                    & 0                & 0 \\
    0                                & 0                                    & 0                & 0
  \end{pmatrix}, \nonumber \\ 
& (P^{[4]})^{-1}G^{[4]}_1 P^{[4]}
= \begin{pmatrix}
    0   & 0   & 0   & 0 \\
    0   & 0   & 0   & 0 \\
    f_1 & f_2 & f_3 & 0 \\
    0   & \displaystyle -q^{-\lambda_1 }\frac{q^{\mu_1 } - 1}{q^{\mu_0 } -1} & -q^{\mu_1 - (\lambda_0 + \lambda_1) }   & 1-q^{-\lambda_1} 
  \end{pmatrix}, 
\end{align}
where 
\begin{align} \label{eq:f1-3}
& f_1 = q^{-\lambda_1} - q^{\mu_1 + \mu_0 - (\lambda_0 + \lambda_1)}, \\
& f_2 = q^{\mu_0 + \mu_1 - (\lambda_0 + \lambda_1)} (1 - \beta/\alpha) + q^{-\lambda_1} \frac{q^{\mu_1 } - 1}{q^{\mu_0 } - 1 } f_0 , \nonumber \\
& f_3 = 1 - q^{-\lambda_1} + q^{\mu_1 - (\lambda_0 + \lambda_1)} f_0 . \nonumber
\end{align}
The upper-left $3\times 3$ submatrices of $(P^{[4]})^{-1}G^{[4]}_0 P^{[4]}$ and $(P^{[4]})^{-1}G^{[4]}_1 P^{[4]}$ are matrix representations of $G^{[4]}_0$ and $G^{[4]}_1$ on the quotient space $\C^4/(\mathcal{K} + \mathcal{L}) $, respectively.
Thus, we can write $mc^q_{\lambda_1}(B^{[2]}_0, B^{[2]}_1) = (\overline{G}^{[3]}_0, \overline{G}^{[3]}_1)$, where
\begin{align} 
&\overline{G}^{[3]}_0 
= \begin{pmatrix}
    1-q^{\mu_0 + \mu_1 - \lambda_0 - \lambda_1} & q^{\mu_0 + \mu_1 - \lambda_0 - \lambda_1}(1 - \beta/\alpha) & 0 \\
    0                               & 1-q^{\mu_1 - \lambda_1}                & q^{\mu_1 - (\lambda_0 + \lambda_1)} (1 - q^{\mu_0 }) \\
    0                               & 0                                  & 0 
                  \end{pmatrix}, \label{mat:Gb30} \\
&\overline{G}^{[3]}_1
= \begin{pmatrix}
    0 & 0 & 0 \\
    0 & 0 & 0 \\
    f_1 & f_2 & f_3
  \end{pmatrix}.   \label{mat:Gb31} 
\end{align}
The $q$-difference equation obtained after applying $q$-middle convolution $mc^q_{\lambda_1}$ is written as
\begin{align} \label{eq:111-21-111}
\frac{1}{-x} 
\begin{pmatrix}
  g_1^{[3]}(qx) - g_1^{[3]}(x) \\
  g_2^{[3]}(qx) - g_2^{[3]}(x) \\
  g_3^{[3]}(qx) - g_3^{[3]}(x)
\end{pmatrix}
= \bigg[ \frac{\overline{G}^{[3]}_0}{x} + \frac{\overline{G}^{[3]}_1}{x - 1/\alpha} \bigg] 
 \begin{pmatrix}
   g_1^{[3]}(x) \\
   g_2^{[3]}(x) \\
   g_3^{[3]}(x)  
  \end{pmatrix}.
\end{align}
From equation (\ref{eq:111-21-111}), the single third-order $q$-difference equation for $g_1^{[3]}(x)$ is derived as
\begin{align} \label{eq:111-21-111-g1}
&(\alpha x-1) g_1^{[3]}(q^3 x) \\
&\quad - q^{-(\lambda_0 + \lambda_1)}
[ \{ (q^{\mu_1} + q^{\lambda_0})\alpha + q^{\mu_0 + \mu_1}\beta \}x - (q^{\mu_0 + \mu_1} + q^{\mu_1 + \lambda_0} + q^{\lambda_0 + \lambda_1}) ] g_1^{[3]}(q^2 x) \nonumber \\
&\quad +q^{\mu_1 - 2(\lambda_0 + \lambda_1)}
[ \{q^{\lambda_0}\alpha + (q^{\mu_0 + \mu_1} + q^{\mu_0 + \lambda_0})\beta \}x - q^{\lambda_0}(q^{\mu_0 + \mu_1} + q^{\mu_0 + \lambda_1} + q^{\lambda_0 + \lambda_1} ) ] g_1^{[3]}(qx) \nonumber \\
&\quad -q^{\mu_0 + 2\mu_1 - (\lambda_0 + 2\lambda_1)}(q^{-(\lambda_0 + \lambda_1)}\beta x - 1)g_1^{[3]}(x)
=0. \nonumber
\end{align}
Each coefficient of $g_1^{[3]} (q^j x)$ $(j=0,1,2,3)$ is a linear polynomial in $x$.
This is the generalized $q$-hypergeometric equation of order $3$ given in equation \eqref{eq:genqhge3-intro}.
The functions $g_2^{[3]}(x)$ and $g_3^{[3]}(x)$ also satisfy the generalized $q$-hypergeometric equation of order $3$, whose coefficients are of degree $1$ and some terms are different from equation (\ref{eq:111-21-111-g1}).

We investigate the $q$-integral representation of solutions to equation (\ref{eq:111-21-111-g1}).
We assume that the hypotheses of Proposition \ref{prop:ghge3conv} hold.
If $\widetilde{V}^{[4]}(x) $ is a solution to equation (\ref{eq:G'0G'1m}) and we write
\begin{align}
\begin{pmatrix}
  g_1(x) \\
  g_2(x) \\
  g_3(x) \\
  g_4(x) 
\end{pmatrix}
&= (P^{[4]})^{-1} \widetilde{V}^{[4]}(x)
= 
\begin{pmatrix}
     1 & 0 & 0 & 0 \\
     0 & 1 & 0 & 0 \\
     0 & 0 & 1 & -f_0/(q^{\mu_0} - 1) \\
     0 & 0 & 0 & 1/(q^{\mu_0} - 1)
   \end{pmatrix} \widetilde{V}^{[4]}(x) , 
\end{align}
then the function ${}^t(g_1(x), g_2(x), g_3(x) ) $ coincides with ${}^t(g_1^{[3]}(x), g_2^{[3]}(x), g_3^{[3]}(x) ) $ which satisfies equation (\ref{eq:111-21-111}).
In particular, from equation \eqref{eq:tYg}, we have
\begin{equation}
g_1^{[3]}(x) = \int_0^{\xi_1 \infty} \frac{K_{\lambda_1}(x,s_1)}{s_1} s_1^{\mu_1 } \int_0^{\xi_0 \infty} \frac{K_{\lambda_0}(s_1,s_0)}{s_0} y(s_0) d_q s_0 d_q s_1 .
\end{equation}
If we specialize
\begin{equation} 
\label{eq:condPyxi}
K_{\lambda_1}(x,s_1) = K^{(1)}_{\lambda_1}(x, s_1) , \; K_{\lambda_0}(s_1, s_0) = K^{(1)}_{\lambda_0}(s_1, s_0) ,  \; y(s_0)=s_0^{\mu_0}\frac{(\alpha s_0;q)_{\infty}}{(\beta s_0;q)_{\infty}} ,
\end{equation}
then we obtain
\begin{align}
g_1^{[3]}(x) 
&= x^{-\lambda_1} \int_0^{\xi_1 \infty} s_1^{-1} \frac{(q^{\lambda_1 + 1}s_1/x;q)_{\infty}}{(qs_1/x;q)_{\infty}} s_1^{\mu_1 - \lambda_0} \int_0^{\xi_0 \infty}  s_0^{\mu_0 - 1} \frac{(q^{\lambda_0 + 1}s_0/s_1, \alpha s_0 ;q)_{\infty}}{(qs_0/s_1 , \beta s_0 ;q)_{\infty}} d_q s_0 d_q s_1 \label{sol:g1-11121111} \\
&= (1 - q)^2 \xi_0^{\mu_0} \xi_1^{\mu_1 - \lambda_0} x^{-\lambda_1} 
 \notag \\
&\cdot \sum_{m_1 = -\infty}^{\infty} \sum_{m_0 = -\infty}^{\infty} (q^{\mu_1 - \lambda_0})^{m_1}  (q^{\mu_0})^{m_0} \frac{(q^{\lambda_1 + m_1 + 1}\xi_1/x, q^{\lambda_0 + m_0 - m_1 + 1}\xi_0/\xi_1, q^{m_0} \xi_0 \alpha; q)_{\infty}}{(q^{m_1 + 1}\xi_1/x, q^{m_0 - m_1 + 1}\xi_0/\xi_1, q^{m_0}\xi_0 \beta; q)_{\infty}} .   \notag
\end{align}
The function \eqref{sol:g1-11121111} is the solution to equation (\ref{eq:111-21-111-g1}).
By specializing the parameters $\xi_0$ and $\xi_1$ and using the $q$-binomial theorem, the function (\ref{sol:g1-11121111})  can be transformed into the $q$-hypergeometric series ${}_3\phi_2$.
\begin{prop}
Assume that the parameters satisfy equation (\ref{eq:prop41}), and set $\lambda_0 \not \in {\mathbb Z}$, $\lambda_1 \not \in {\mathbb Z}$.
If $|q^{ -(\lambda_0 + \lambda_1 )} \beta x| <1 $ holds, then \eqref{sol:g1-11121111} can be transformed into the following $q$-hypergeometric series ${}_3\phi_2$:
\begin{align}
&g_1^{[3]}(x) | _{\xi_0 = q^{ -(\lambda_0 + \lambda_1) } x ,\; \xi_1 = q ^{-\lambda_1 }x  } \label{eq:g1qlalaxqlax} \\
&\quad = (1 - q)^2 q^{ -\mu_0(\lambda_0 + \lambda_1) - (\mu_1 - \lambda_0) \lambda_1} \;
x^{\mu_0 + \mu_1 - (\lambda_0 + \lambda_1) }
\frac{(q, q, q^{\mu_0 - \lambda_0 + 1} , q^{\mu_0 + \mu_1 - (\lambda_0 + \lambda_1) + 1 } ; q)_{\infty}}{(q^{-\lambda_0 + 1} , q^{-\lambda_1 + 1} , q^{\mu_0 } , q^{\mu_0 + \mu_1 - \lambda_0 } ; q)_{\infty}} \nonumber \\
&\qquad\qquad \cdot  {}_3\phi_2 \Bigl( { \alpha /\beta , q^{\mu_0 }, q^{\mu_0 + \mu_1 - \lambda_0 } \atop{q^{\mu_0 - \lambda_0 + 1} , q^{\mu_0 + \mu_1 - (\lambda_0 + \lambda_1) + 1 }}} ; q, q^{ -(\lambda_0 + \lambda_1) } \beta x \Bigr). \nonumber  
\end{align}
\end{prop}
\begin{proof}
Set $\xi_1 = q ^{-\lambda_1 }x  $ and $\xi_0  = q^{ -(\lambda_0 + \lambda_1) } x$ in equation (\ref{sol:g1-11121111}).
We change the variables $(m_0,m_1)$ in equation (\ref{sol:g1-11121111}) to the variables $(M_0,M_1)$ by setting $m_1 = M_1$ and $m_0 - m_1 = M_0$. 
It follows from $(q^{\ell +1};q)_{\infty} = 0$ for $\ell \in \mathbb{Z}_{<0}$ that
\begin{align}
& \sum_{m_1 = -\infty}^{\infty} \sum_{m_0 = -\infty}^{\infty} (q^{\mu_1 - \lambda_0})^{m_1}  (q^{\mu_0})^{m_0} \frac{(q^{\lambda_1 + m_1 + 1}\xi_1/x, q^{\lambda_0 + m_0 - m_1 + 1}\xi_0/\xi_1, q^{m_0} \xi_0 \alpha; q)_{\infty}}{(q^{m_1 + 1}\xi_1/x, q^{m_0 - m_1 + 1}\xi_0/\xi_1, q^{m_0}\xi_0 \beta; q)_{\infty}} \label{eq:prop1} \\ 
& =  \sum_{M_1 = 0}^{\infty} \sum_{M_0 = 0}^{\infty} (q^{\mu_1 - \lambda_0})^{M_1}  (q^{\mu_0})^{M_0 + M_1} \frac{(q^{M_1 + 1}, q^{M_0 + 1} , q^{ -(\lambda_0 + \lambda_1) + M_0 + M_1} \alpha x ; q)_{\infty}}{(q^{-\lambda_1 + M_1 + 1} , q^{- \lambda_0 + M_0 + 1}, q^{ -(\lambda_0 + \lambda_1) + M_0 + M_1} \beta x ; q)_{\infty}} . \nonumber
\end{align}
Note that the summand in equation (\ref{eq:prop1}) is well-defined, if $\lambda_0 \not \in {\mathbb Z}$ and $\lambda_1 \not \in {\mathbb Z}$.
By the $q$-binomial theorem, we have
\begin{equation}
\frac{(q^{ -(\lambda_0 + \lambda_1) + M_0 + M_1} \alpha x; q)_{\infty}}{(q^{ -(\lambda_0 + \lambda_1) + M_0 + M_1} \beta x; q)_{\infty}} 
= \sum_{N=0}^{\infty} \frac{(\alpha /\beta ; q)_N}{(q; q)_N} (q^{ -(\lambda_0 + \lambda_1) + M_0 + M_1} \beta x)^N ,
\end{equation}
if $|q^{ -(\lambda_0 + \lambda_1) } \beta x | <1$. 
Hence, equation (\ref{eq:prop1}) is equal to 
\begin{align}
&  \sum_{M_1 = 0}^{\infty} \sum_{M_0 = 0}^{\infty} \sum_{N=0}^{\infty} (q^{\mu_0 + \mu_1  - \lambda_0 + N })^{M_1}  (q^{\mu_0 + N} )^ {M_0} \frac{(q^{M_1 + 1}, q^{M_0 + 1} ; q)_{\infty}}{(q^{-\lambda_1 + M_1 + 1} , q^{- \lambda_0 + M_0 + 1} ; q)_{\infty}}  \label{eq:prop2} \\
&\hspace{250pt} \cdot \frac{(\alpha /\beta ; q)_N}{(q; q)_N} (q^{ -(\lambda_0 + \lambda_1) } \beta x)^N  \nonumber \\
& = \frac{(q, q; q)_{\infty}}{(q^{-\lambda_0 + 1} , q^{- \lambda_1 + 1} ; q)_{\infty}} \sum_{N=0}^{\infty} \frac{(\alpha /\beta ; q)_N}{(q; q)_N} (q^{ -(\lambda_0 + \lambda_1) } \beta x)^N \nonumber \\
&\qquad \cdot 
\left( \sum_{M_0 = 0}^{\infty} \frac{(q^{-\lambda_0 + 1} ; q)_{M_0}}{(q; q)_{M_0}} (q^{\mu_0 + N} )^ {M_0} \right)
\left( \sum_{M_1 = 0}^{\infty} \frac{(q^{-\lambda_1 + 1}; q)_{M_1}}{(q; q)_{M_1}} (q^{\mu_0 + \mu_1 - \lambda_0 + N })^{M_1} \right) . \nonumber
\end{align}
It follows from the conditions $ |q^{\mu_0 + \mu_1 - \lambda_0}| < 1$ and $|q^{\mu_0}| < 1$ in equation \eqref{eq:prop41} and the $q$-binomial theorem that equation (\ref{eq:prop2}) is equal to
\begin{align}
& \frac{(q, q; q)_{\infty}}{(q^{-\lambda_0 + 1} , q^{- \lambda_1 + 1} ; q)_{\infty}} \\
&\quad \cdot \sum_{N=0}^{\infty} 
\frac{ (q^{\mu_0 -\lambda_0 + 1 + N} ; q)_{\infty} }{(q^{\mu_0 + N}  ; q)_{\infty}}
\frac{ (q^{\mu_0 + \mu_1 - (\lambda_0 + \lambda_1) + 1 + N } ; q)_{\infty} }{(q^{\mu_0 + \mu_1 - \lambda_0 + N } ; q)_{\infty}}  
\frac{(\alpha /\beta ; q)_N}{(q; q)_N} (q^{ -(\lambda_0 + \lambda_1) } \beta x)^N \nonumber \\
& = \frac{(q, q, q^{\mu_0 - \lambda_0 + 1} , q^{\mu_0 + \mu_1 - (\lambda_0 + \lambda_1) + 1 } ; q)_{\infty}}{(q^{-\lambda_0 + 1} , q^{- \lambda_1 + 1} , q^{\mu_0 } , q^{\mu_0 + \mu_1 - \lambda_0 } ; q)_{\infty}} \nonumber \\
&\qquad \cdot \sum_{N=0}^{\infty}  \frac{(\alpha /\beta , q^{\mu_0 }, q^{\mu_0 + \mu_1 - \lambda_0 } ; q)_N}{(q, q^{\mu_0 - \lambda_0 + 1} , q^{\mu_0 + \mu_1 - (\lambda_0 + \lambda_1) + 1 } ; q)_N} (q^{ -(\lambda_0 + \lambda_1) } \beta x)^N .  \nonumber
\end{align}
By combining with equation (\ref{sol:g1-11121111}), we obtain the proposition.
\end{proof}
If we choose $K_{\lambda_1}(x, s_1)$, $K_{\lambda_0}(s_1, s_0)$, and $y(s_0)$ differently from equation \eqref{eq:condPyxi} and suitably specialize the parameters $\xi_0$ and $\xi_1$, then we obtain $q$-hypergeometric series solutions whose representations differ from that in equation \eqref{eq:g1qlalaxqlax}. 
This will be explained in detail in Section \ref{sec:Intsoldegn}.
In addition, $q$-hypergeometric series solutions to equation \eqref{eq:111-21-111-g1} with representations 
different from that in equation \eqref{eq:g1qlalaxqlax} are listed in Appendix \ref{subsec:qhgs3}.

\section{Generalized $q$-hypergeometric equation of order $n$} \label{se:genphg-n}
In Section \ref{sec:qmcex}, we applied $mc^q_{\lambda_1} \circ add_{\mu_1}$ to the pair of $2 \times 2$ matrices $(G^{[2]}_0, G^{[2]}_1)$ in equation \eqref{eq:G0G1qhyp0} to obtain the pair of $3 \times 3$ matrices $(\overline{G}^{[3]}_0, \overline{G}^{[3]}_1)$ given in equations \eqref{mat:Gb30} and \eqref{mat:Gb31}.
In this section, we further apply the addition and the $q$-middle convolution repeatedly to obtain a pair of $n \times n$ matrices $(\overline{G}^{[n]}_0, \overline{G}^{[n]}_1)$.
We then derive a single $q$-difference equation of order $n$ from the resulting system and show that all its coefficients are polynomials of degree at most $1$ in $x$.
To apply the $q$-middle convolution successively, we impose the conditions on the parameters given in Proposition \ref{prop:Gb0nGb1n} below at each step.
These conditions ensure that the dimension reduction associated with the $q$-middle convolution yields a system of the expected dimension.
\begin{prop} \label{prop:Gb0nGb1n}
Let $(G^{[2]}_0, G^{[2]}_1)$ be the pair of matrices in equation (\ref{eq:G0G1qhyp0}).
Starting from this pair, successively apply 
\begin{equation*}
mc^q_{\lambda_j} \circ add_{\mu_j}, \quad j = 1, \ldots,  n-2 ,
\end{equation*}
where at the $j$-th step we assume that
\begin{equation*}
\det (B^{[j+1]}_0) \neq 0, \quad b^{[j+1]}_1 \neq 0, \quad \det (G^{[2(j+1)]}_0 + G^{[2(j+1)]}_1) \neq 0 .
\end{equation*}
Then for $n \in \mathbb{N}_{\ge 3}$, the resulting pair of $n \times n$ matrices $(\overline{G}^{[n]}_0, \overline{G}^{[n]}_1)$ is given by
\begin{align}
&\overline{G}^{[n]}_0 \label{eq:Gn0}
= \begin{pmatrix}
    a^{[n]}_{11} & a^{[n]}_{12}         &          &        &                  & \\
             & a^{[n]}_{22}         & a^{[n]}_{23} &        & \text{\LARGE{O}} & \\
             &                  & \ddots   & \ddots & \\
             &                  &          & \ddots & \ddots \\
             & \text{\LARGE{O}} &          &        & a^{[n]}_{n-1,n-1}    & a^{[n]}_{n-1,n} \\
             &                  &          &        &                  & 0
  \end{pmatrix} , \\
&\overline{G}^{[n]}_1 \label{eq:Gn1}
= \begin{pmatrix}
          &       &       & \text{\Large{O}} &           &           &   \\
          &       &       & \vdots   &           &           &    \\
          &       &       & \text{\Large{O}} &           &           &   \\
    b^{[n]}_1 & b^{[n]}_2 & b^{[n]}_3 & \dots    & b^{[n]}_{n-2} & b^{[n]}_{n-1} & b^{[n]}_n
\end{pmatrix},
\end{align}
where
\begin{align*}
& a^{[n]}_{ii} = 1 - q^{\sum_{k = i-1}^{n-2}(\mu_k - \lambda_k)}, \quad i =1, \ldots, n - 1, \\
& a^{[n]}_{12} = q^{\sum_{k = 0}^{n-2}(\mu_k - \lambda_k)} (1 - \beta/\alpha), \\
& a^{[n]}_{23} = q^{\sum_{k = 1}^{n-2}\mu_k - \sum_{k = 0}^{n-2}\lambda_k} (1 - q^{\mu_0}), \\
& a^{[n]}_{i, i + 1} = q^{\sum_{k = i-1}^{n-2}\mu_k - \sum_{k = i-2}^{n-2}\lambda_k} (1 - q^{\sum_{k = 0}^{i-2}\mu_k - \sum_{k = 0}^{i-3}\lambda_k}), \quad n \ge 4, \quad i = 3, \ldots, n-1, \\
& b^{[n]}_1 = q^{-\lambda_{n-2}} - q^{\sum_{k = 0}^{n-2}(\mu_k - \lambda_k)}, \\
& b^{[n]}_2 = a^{[n]}_{12} + (q^{-\lambda_{n-2}} - q^{\sum_{k = 1}^{n-2}(\mu_k - \lambda_k)} )  \frac{b^{[n-1]}_2}{b^{[n-1]}_1}, \\
& b^{[n]}_i = q^{\mu_{n-2} - \lambda_{n-2}} a^{[n-1]}_{i-1, i} \frac{b^{[n-1]}_{i-1}}{b^{[n-1]}_1} + (q^{-\lambda_{n-2}} - q^{\sum_{k = i-1}^{n-2}(\mu_k - \lambda_k)} ) \frac{b^{[n-1]}_i}{b^{[n-1]}_1}, \\
&\hspace{260pt} n \ge 4, \quad i = 3, \ldots, n-1, \\
& b^{[n]}_n = 1 - q^{-\lambda_{n-2}} + q^{\mu_{n-2} - \lambda_{n-2}} b^{[n-1]}_{n-1}  \\
& a^{[2]}_{12} = q^{\mu_0 - \lambda_0} (1 - \beta/\alpha), \quad a^{[2]}_{22} = 0, \quad b^{[2]}_1 = q^{-\lambda_0}(1 - q^{\mu_0}), \quad b^{[2]}_2 = a^{[2]}_{12} + 1 - q^{-\lambda_0} .
\end{align*}
\end{prop}
\begin{proof}
We prove the proposition by induction on $n$.

For $n = 3$, the right-hand side of (\ref{eq:Gn0}) is written as
\begin{align*}
&\begin{pmatrix}
  a^{[3]}_{11} & a^{[3]}_{12} & 0 \\
  0        & a^{[3]}_{22} & a^{[3]}_{23} \\
  0        & 0        & 0
\end{pmatrix} \\
&\quad = \begin{pmatrix}
    1 - q^{\mu_0 + \mu_1 - (\lambda_0 + \lambda_1)} & q^{\mu_0 + \mu_1 - (\lambda_0 + \lambda_1)} (1 - \beta/\alpha) & 0 \\
    0                                             & 1 - q^{\mu_1 - \lambda_1}                                    & q^{\mu_1 - (\lambda_0 + \lambda_1)} (1 - q^{\mu_0}) \\
    0                                             & 0                                                            & 0
  \end{pmatrix} ,
\end{align*}
and this matrix coincides with $\overline{G}^{[3]}_0$ given by \eqref{mat:Gb30}.
The right-hand side of (\ref{eq:Gn1}) is written as
\begin{equation} \label{eq:Gb31}
\begin{pmatrix}
  0     & 0      & 0  \\
  0     & 0      & 0  \\
  b^{[3]}_1 & b^{[3]}_2 & b^{[3]}_3
\end{pmatrix}.
\end{equation}
It follows from 
\begin{align*}
& b^{[3]}_1 = q^{-\lambda_1 } - q^{\mu_0 + \mu_1 - (\lambda_0 + \lambda_1)} , \\
& b^{[3]}_2 = q^{\mu_0 + \mu_1 - (\lambda_0 + \lambda_1)} (1 - \beta/\alpha) + q^{-\lambda_1} (1 - q^{\mu_1 })\frac{q^{\mu_0 - \lambda_0} (1 - \beta/\alpha) + 1 - q^{-\lambda_0}}{q^{-\lambda_0} ( 1 - q^{\mu_0})} , \\
& b^{[3]}_3 = q^{\mu_1 - \lambda_1} (1 - q^{-\lambda_0}) + 1 - q^{-\lambda_1} + q^{\mu_0 + \mu_1 - (\lambda_0 + \lambda_1)} (1 - \beta/\alpha)
\end{align*}
that $b^{[3]}_1$, $b^{[3]}_2$, and $b^{[3]}_3$ are equal to $f_1$, $f_2$, and $f_3$ in equation \eqref{eq:f1-3}, respectively.
Namely, matrix (\ref{eq:Gb31}) coincides with $\overline{G}^{[3]}_1$ given by \eqref{mat:Gb31}.
The pair of matrices $(\overline{G}^{[3]}_0, \overline{G}^{[3]}_1)$ is obtained by applying $mc^q_{\lambda_1} \circ add_{\mu_1}$ to $(G^{[2]}_0, G^{[2]}_1)$ under the conditions $\det (B^{[2]}_0) \neq 0$, $b^{[2]}_1 \neq 0$, and $\det (G^{[4]}_0 + G^{[4]}_1) \neq 0$.
Hence the proposition holds for $n = 3$.

Suppose that the proposition holds for some $n \ge 3$.
We show that it also holds for $n+1$.
We apply $mc^q_{\lambda_{n-1}} \circ add_{\mu_{n-1}}$ to the pair of $n \times n$ matrices $(\overline{G}^{[n]}_0, \overline{G}^{[n]}_1)$.
Set $add_{\mu_{n-1}} (\overline{G}^{[n]}_0, \overline{G}^{[n]}_1) = (B^{[n]}_0, B^{[n]}_1)$.
The matrices $B^{[n]}_0$ and $B^{[n]}_1$ are the following $n \times n$ matrices:
\begin{align}
B^{[n]}_0 
&= (1 - q^{\mu_{n-1}})I_n + q^{\mu_{n-1}} \overline{G}^{[n]}_0  \label{mat:Bn+1-0} \\
&= \begin{pmatrix}
  d_{11} & d_{12}            &          &        &                  & \\
         & d_{22}            & d_{23}   &        & \text{\LARGE{O}} & \\
         &                   & \ddots   & \ddots &                  & \\
         &                   &          & \ddots & \ddots           & \\
         & \text{\LARGE{O}}  &          &        & d_{n-1,n-1}      & d_{n-1,n} \\
         &                   &          &        &                  & 1 - q^{\mu_{n-1}}
 \end{pmatrix} , \nonumber  
\end{align}
where for $i = 1,\ldots, n-1$
\begin{align*}
& d_{ii} = 1 - q^{\mu_{n-1}} + q^{\mu_{n-1}}a^{[n]}_{ii}, \\
& d_{i,i+1} = q^{\mu_{n-1}}a^{[n]}_{i,i+1}, 
\end{align*}
\begin{align}
B^{[n]}_1
& = q^{\mu_{n-1}} \overline{G}^{[n]}_1 \label{mat:Bn+1-1}\\
& = \begin{pmatrix}
          &       &       & \text{\Large{O}} &           &           &   \\
          &       &       & \vdots   &           &           &    \\
          &       &       & \text{\Large{O}} &           &           &   \\
    q^{\mu_{n-1}}b^{[n]}_1 & q^{\mu_{n-1}}b^{[n]}_2 & q^{\mu_{n-1}}b^{[n]}_3 & \dots    & q^{\mu_{n-1}}b^{[n]}_{n-2} & q^{\mu_{n-1}} b^{[n]}_{n-1} & q^{\mu_{n-1}}b^{[n]}_n
\end{pmatrix}. \nonumber
\end{align}
We apply $c^q_{\lambda_{n-1}}$ to $(B^{[n]}_0, B^{[n]}_1)$ and obtain the pair of $2n \times 2n$ matrices $(G^{[2n]}_0, G^{[2n]}_1)$ as follows:
\begin{align*}
G^{[2n]}_0 
&= \begin{pmatrix}
   q^{-\lambda_{n-1}} B^{[n]}_0 + (1 - q^{-\lambda_{n-1}}) I_n & q^{-\lambda_{n-1}} B^{[n]}_1 \\
   O_n                                                         & O_n
\end{pmatrix} 
= \begin{pmatrix}
    g_{01} & g_{02} \\
    O_n    & O_n
\end{pmatrix} , 
\end{align*}
where
\begin{align*}
& g_{01} 
= \begin{pmatrix}
  e^0_{11} & e^0_{12}          &          &        &                  & \\
           & e^0_{22}          & e^0_{23} &        & \text{\LARGE{O}} & \\
           &                   & \ddots   & \ddots &                  & \\
           &                   &          & \ddots & \ddots           & \\
           & \text{\LARGE{O}}  &          &        & e^0_{n-1,n-1}    & e^0_{n-1,n} \\
           &                   &          &        &                  & 1 - q^{\mu_{n-1} - \lambda_{n-1}}
 \end{pmatrix} , \\
& g_{02}
= \begin{pmatrix}
          &       &       & \text{\Large{O}} &           &           &   \\
          &       &       & \vdots   &           &           &    \\
          &       &       & \text{\Large{O}} &           &           &   \\
    q^{\mu_{n-1} - \lambda_{n-1}}b^{[n]}_1 & q^{\mu_{n-1} - \lambda_{n-1}}b^{[n]}_2 &   & \dots    &   & q^{\mu_{n-1} - \lambda_{n-1}} b^{[n]}_{n-1} & q^{\mu_{n-1} - \lambda_{n-1}}b^{[n]}_n
\end{pmatrix},
\end{align*}
with
 \begin{align*}
& e^0_{ii} = 1 +  q^{\mu_{n-1} - \lambda_{n-1}} (a^{[n]}_{ii} -1), \\
& e^0_{i,i+1} = q^{\mu_{n-1} - \lambda_{n-1}} a^{[n]}_{i,i+1}, \quad i = 1, \ldots, n-1,
\end{align*}
\begin{align*}
G^{[2n]}_1 
&= \begin{pmatrix}
   O_n                           &  O_n \\
   q^{-\lambda_{n-1}} B^{[n]}_0  & q^{-\lambda_{n-1}} B^{[n]}_1 + (1 - q^{-\lambda_{n-1}}) I_n 
\end{pmatrix} 
= \begin{pmatrix}
    O_n    & O_n \\
    g_{11} & g_{12} 
\end{pmatrix} , 
\end{align*}
where
\begin{align*}
& g_{11}
= \begin{pmatrix}
  e^1_{11} & e^1_{12}          &          &        &                  & \\
           & e^1_{22}          & e^1_{23} &        & \text{\LARGE{O}} & \\
           &                   & \ddots   & \ddots &                  & \\
           &                   &          & \ddots & \ddots           & \\
           & \text{\LARGE{O}}  &          &        & e^1_{n-1,n-1}    & e^1_{n-1,n} \\
           &                   &          &        &                  & (1 - q^{\mu_{n-1}})/ q^{\lambda_{n-1}}
\end{pmatrix} , \\
& g_{12}
= \begin{pmatrix}
   1 - q^{-\lambda_{n-1}}           & 0       & \dots  & \dots  & \dots  & \dots   & 0  \\
        0                                     &  \ddots & \ddots &        &        &         & \vdots \\
        \vdots                                &  \ddots & \ddots & \ddots &        &         & \vdots \\
        \vdots                                &         & \ddots & \ddots & \ddots &         & \vdots \\
        \vdots                                &         &        & \ddots & \ddots & \ddots & \vdots \\
        0                                     & \dots   & \dots  & \dots  &    0       &  1 - q^{-\lambda_{n-1}}  & 0  \\
 q^{\mu_{n-1} - \lambda_{n-1}}b^{[n]}_1 & \dots   & \dots & \dots &  \dots  & q^{\mu_{n-1} - \lambda_{n-1}} b^{[n]}_{n-1} & e^2_{nn}
\end{pmatrix}, 
\end{align*}
with
\begin{align*}
& e^1_{ii} = \frac{1 +  q^{\mu_{n-1}} (a^{[n]}_{ii} -1)}{q^{\lambda_{n-1}}}, \\
& e^1_{i,i+1} = q^{\mu_{n-1} - \lambda_{n-1}} a^{[n]}_{i,i+1}, \quad i = 1, \ldots, n-1, \\
& e^2_{nn} = 1 - q^{-\lambda_{n-1}} + q^{\mu_{n-1} - \lambda_{n-1}} b^{[n]}_n .
\end{align*}

Now we consider the $q$-middle convolution $mc^q_{\lambda_{n-1}}$ to $(B^{[n]}_0, B^{[n]}_1)$.
Since $b^{[n]}_1 \neq 0$ by the assumption of the proposition, we have $\dim \ker (B^{[n]}_1) = n-1$, and a basis of $\ker (B^{[n]}_1)$ is given by
\begin{equation*}
\begin{pmatrix}
 -b^{[n]}_2 \\ b^{[n]}_1 \\  0 \\ \vdots \\ 0 
\end{pmatrix} ,
\begin{pmatrix}
 -b^{[n]}_3 \\ 0 \\ b^{[n]}_1 \\  0 \\ \vdots \\ 0 
\end{pmatrix},
\dots , 
\begin{pmatrix}
 -b^{[n]}_{n-1} \\ 0 \\ \vdots \\ 0\\ b^{[n]}_1 \\ 0
\end{pmatrix},
\begin{pmatrix}
 -b^{[n]}_n \\ 0 \\ \vdots \\ 0\\ b^{[n]}_1
\end{pmatrix} .
\end{equation*}
Under the assumption $\det(B^{[n]}_0) \neq 0$, we have $\dim \mathcal{K} = n-1$.
The assumption $\det (G^{[2n]}_0 + G^{[2n]}_1) \neq 0$ implies that $\mathcal{L} = \{ 0 \}$, and hence $\dim \mathcal{L} = 0$.
Therefore, under the assumptions of the proposition, we have $\dim (\C^{2n}/(\mathcal{K}+ \mathcal{L})) = n+1$.
To obtain matrix representations of $mc^q_{\lambda_{n-1}}(B^{[n]}_0, B^{[n]}_1)$, we apply a simultaneous similarity transformation to matrices $G^{[2n]}_0$ and $G^{[2n]}_1$ by the $2n \times 2n$ invertible matrix $P^{[2n]}$, which is  written as 
\begin{equation}
P^{[2n]}
= \begin{pmatrix}
   I_n & O_n \\
   O_n & p_n
\end{pmatrix} ,
\quad
 p_n 
= \begin{pmatrix}
    1      & -b^{[n]}_2 & -b^{[n]}_3  & \dots  & -b^{[n]}_{n-1} & -b^{[n]}_{n} \\
    0      & b^{[n]}_1  & 0           & \dots  &    0           & 0 \\
    \vdots & 0          & b^{[n]}_1   & \ddots &                & \vdots \\
    \vdots & \vdots     &  0          & \ddots & \ddots         & \vdots \\
    \vdots & \vdots     &  \vdots     & \ddots & b^{[n]}_{1}    & 0 \\
    0      & 0          & 0           & \dots  & 0              & b^{[n]}_{1}
\end{pmatrix} .
\end{equation} 
For $G^{[2n]}_0$, we have
\begin{equation*}
(P^{[2n]})^{-1}G^{[2n]}_0 P^{[2n]}
= \begin{pmatrix}
   I_n & O_n \\
   O_n & p_n^{-1}
  \end{pmatrix}
  \begin{pmatrix}
   g_{01} & g_{02} \\
   O_n    & O_n 
  \end{pmatrix}
  \begin{pmatrix}
   I_n & O_n \\
   O_n & p_n 
  \end{pmatrix}
= \begin{pmatrix}
   g_{01} & g_{02}\, p_n \\
   O_n    & O_n 
  \end{pmatrix} ,
\end{equation*}
where
\begin{equation*}
g_{02}\, p_n 
= \begin{pmatrix}
                                     & O      &       &  \\
                                     & \vdots &       & \\
                                     & O      &       & \\
  q^{\mu_{n-1} - \lambda_{n-1}}b^{[n]}_1 & 0      & \dots & 0
 \end{pmatrix} .
\end{equation*}
The upper-left $(n+1) \times (n+1)$ submatrix of $(P^{[2n]})^{-1}G^{[2n]}_0 P^{[2n]}$ is written as
\begin{align} \label{mat:Gb(n+1)0}
&\left(
\begin{array}{ccc|c}
 &        &  & 0 \\
 & g_{01} &  & \vdots \\
 &        &  & 0 \\
 &        &  & q^{\mu_{n-1} - \lambda_{n-1}}b^{[n]}_1 \\
\hline
0 & \dots & 0 & 0
\end{array}
\right) \\
&= \begin{pmatrix} \nonumber
  e^0_{11} & e^0_{12}        &          &        &               &                                   &   \\
           & e^0_{22}        & e^0_{23} &        &               &  \text{\huge{O}}                  &  \\
           &                 & \ddots   & \ddots &               &                                   &  \\
           &                 &          & \ddots & \ddots        &                                   &  \\
           &                 &          &        & e^0_{n-1,n-1} & e^0_{n-1,n}                       &  \\
           & \text{\huge{O}} &          &        &               & 1 - q^{\mu_{n-1} - \lambda_{n-1}} & q^{\mu_{n-1} - \lambda_{n-1}}b^{[n]}_1 \\
           &                 &          &        &               &                                   & 0
\end{pmatrix} .
\end{align}
It follows from 
\begin{align*}
& e^0_{ii} = 1 +  q^{\mu_{n-1} - \lambda_{n-1}} (a^{[n]}_{ii} -1) = 1 - q^{\sum_{k = i-1}^{n-1}(\mu_k - \lambda_k)} 
= a^{[n+1]}_{ii} \quad  i = 1,  \ldots, n - 1, \\
& 1 - q^{\mu_{n-1} - \lambda_{n-1}} = a^{[n+1]}_{n,n},  \\
& e^0_{i,i+1} = q^{\mu_{n-1} - \lambda_{n-1}} a^{[n]}_{i,i+1} = a^{[n+1]}_{i,i+1} \quad i = 1, \ldots n-1,\\
& q^{\mu_{n-1} - \lambda_{n-1}}b^{[n]}_1 = a^{[n+1]}_{n,n+1} 
\end{align*}
that matrix (\ref{mat:Gb(n+1)0}) coincides with matrix $\overline{G}^{[n+1]}_0$. 
For $G^{[2n]}_1$, we have
\begin{equation*}
(P^{[2n]})^{-1}G^{[2n]}_1 P^{[2n]}
= \begin{pmatrix}
   I_n & O_n \\
   O_n & p_n^{-1}
  \end{pmatrix}
  \begin{pmatrix}
   O_n    & O_n \\
   g_{11} & g_{12}    
  \end{pmatrix}
  \begin{pmatrix}
   I_n & O_n \\
   O_n & p_n
  \end{pmatrix}
= \begin{pmatrix}
   O_n & O_n \\
   p_n^{-1} \, g_{11} & p_n^{-1} \, g_{12} \, p_n  
\end{pmatrix} , 
\end{equation*}
where
\begin{align*}
& p_n^{-1} \, g_{11} 
= \begin{pmatrix}
   h_1 & h_2 & \dots  & h_{n-1} & h_n \\
       &     & \ast      &         &     \\
       &     & \vdots &         &     \\
       &     & \ast      &         &    
 \end{pmatrix}, \quad
 p_n^{-1} \, g_{12} \, p_n 
= \begin{pmatrix}
   h_{n+1} & 0    & \dots & 0 \\
           & \ast &       &   \\
           & \vdots &     &   \\
           & \ast   &      &
\end{pmatrix}.
\end{align*}
The upper-left $(n+1) \times (n+1)$ submatrix of $(P^{[2n]})^{-1}G^{[2n]}_1 P^{[2n]}$ is written as
\begin{equation}
\begin{pmatrix} \label{mat:Gb(n+1)1}
          &       &       & \text{\Large{O}} &           &           &   \\
          &       &       & \vdots           &           &           &    \\
          &       &       & \text{\Large{O}} &           &           &   \\
      h_1 & h_2   & h_3   & \dots            & h_{n-1}   & h_{n}     & h_{n+1}
\end{pmatrix}, 
\end{equation}
where
\begin{align*}
& h_1 = e^1_{11} 
= b^{[n+1]}_1,  \nonumber \\
& h_t = e^1_{t-1,t} b^{[n]}_{t-1}/b^{[n]}_1 + e^1_{tt}b^{[n]}_t/b^{[n]}_1 = b^{[n+1]}_t, \quad t = 2, \ldots, n-1, \nonumber \\
& h_{n} = e^1_{n-1, n} b^{[n]}_{n-1}/b^{[n]}_1 + (1- q^{\mu_{n-1}})b^{[n]}_{n}/(q^{\lambda_{n-1}}b^{[n]}_1) = b^{[n+1]}_{n}, \nonumber \\
& h_{n+1} = 1 - q^{-\lambda_{n-1}} + q^{\mu_{n-1} - \lambda_{n-1}}b^{[n]}_{n} = b^{[n+1]}_{n+1}. \nonumber  
\end{align*}
Hence matrix (\ref{mat:Gb(n+1)1}) coincides with matrix $\overline{G}^{[n+1]}_1$. 
Thus, the proposition holds for $n+1$.
This completes the proof.
\end{proof}
We are now ready to state the main result of this section.
\begin{thm} \label{thm:g1n-single-sec4}
Let $\overline{G}^{[n]}_0$ and $\overline{G}^{[n]}_1$ be the matrices given in Proposition \ref{prop:Gb0nGb1n}.
Assume that $a^{[n]}_{i,i+1} \neq 0$ for $i = 1, \ldots, n-1$.
The single $q$-difference equation of order $n$ for $g_1^{[n]}(x)$ derived from the system 
\begin{equation} \label{eq:genqhge-n}
 \frac{1}{-x} 
 \begin{pmatrix}
   g_1^{[n]}(qx) - g_1^{[n]}(x) \\
   \vdots \\
   g_n^{[n]}(qx) - g_n^{[n]}(x)
 \end{pmatrix} 
 = \left[ \frac{\overline{G}^{[n]}_0}{x} + \frac{\overline{G}^{[n]}_1}{x - 1/\alpha} \right]
 \begin{pmatrix}
   g_1^{[n]}(x) \\
   \vdots \\
   g_n^{[n]}(x)
 \end{pmatrix}
\end{equation}
has coefficients that are polynomials of degree at most $1$ in $x$.
More precisely, it can be written in the form
\begin{align}
&(\alpha x - 1) g_1^{[n]}(q^n x) + (U_{n-1} x + W_{n-1}) g_1^{[n]}(q^{n-1} x) +  (U_{n-2} x + W_{n-2}) g_1^{[n]}(q^{n-2} x) \label{eq:gn1single-thm}\\
&\qquad\qquad + \cdots + (U_{2} x + W_{2}) g_1^{[n]}(q^{2} x) + (U_{1} x + W_{1}) g_1^{[n]}(q x)  \nonumber \\
&\hspace{140pt} + (-1)^n q^{\sum_{k=0}^{n-2}(k+1)(\mu_k - \lambda_k)} (U_{0} x - 1) g_1^{[n]}(x) = 0, \nonumber
\end{align}
where $U_i$ $(i=0, \ldots, n-1)$ and $W_i$ $(i=1, \ldots, n-1)$ are constants.
In particular, if $U_0 \neq 0$, then equation \eqref{eq:gn1single-thm} is a generalized $q$-hypergeometric equation of order $n$.
\end{thm}
\begin{proof}
Equation (\ref{eq:genqhge-n}) can be rewritten as the following $n$ equations:
\begin{align}
&g_1^{[n]}(qx) = (1 - a^{[n]}_{11}) g_1^{[n]}(x) - a^{[n]}_{12} g_2^{[n]}(x) , \label{eq:1}\tag{4.9.1} \\
&g_2^{[n]}(qx) = (1 - a^{[n]}_{22}) g_2^{[n]}(x) - a^{[n]}_{23} g_3^{[n]}(x) ,  \label{eq:2}\tag{4.9.2}\\
&\qquad \vdots \nonumber \\
&g_{k}^{[n]}(qx) = (1 - a^{[n]}_{k, k}) g_{k}^{[n]}(x) - a^{[n]}_{k, k+1} g_{k+1}^{[n]}(x) , \label{eq:k}\tag{4.9-k} \\
 &\qquad \vdots \nonumber \\
&g_{n-1}^{[n]}(qx) = (1 - a^{[n]}_{n-1, n-1}) g_{n-1}^{[n]}(x) - a^{[n]}_{n-1, n} g_n^{[n]}(x) , \label{eq:n-1}\tag{4.9.n-1} \\
&g_n^{[n]}(qx) = \frac{b^{[n]}_1 \alpha x}{1 - \alpha x} g_1^{[n]}(x) + \cdots + \frac{b^{[n]}_{n-1} \alpha x}{1 - \alpha x} g_{n-1}^{[n]}(x) \label{eq:n}\tag{4.9.n}  
+ \left(1 + \frac{b^{[n]}_{n} \alpha x}{1 - \alpha x} \right) g_n^{[n]}(x)   . \nonumber
\end{align}
Equation \eqref{eq:1} can be rewritten as 
\begin{equation}
  g_2^{[n]}(x) = \frac{-1}{a^{[n]}_{12}} g_{1}^{[n]}(qx) + \frac{1 - a^{[n]}_{11}}{a^{[n]}_{12}} g_{1}^{[n]}(x),\label{eq:i-1'}\tag{4.9.1'}
\end{equation}
and substituting equation \eqref{eq:i-1'} into equation \eqref{eq:2} yields 
\begin{equation*}
g_3^{[n]}(x) 
= \frac{(-1)^2}{a^{[n]}_{12}a^{[n]}_{23}} g_{1}^{[n]}(q^2 x) 
- \frac{2 - a^{[n]}_{11} - a^{[n]}_{22}}{a^{[n]}_{12}a^{[n]}_{23}} g_{1}^{[n]}(qx) + \frac{(1 - a^{[n]}_{11})(1 - a^{[n]}_{22})}{a^{[n]}_{12}a^{[n]}_{23}}  g_{1}^{[n]}(x). 
\end{equation*}
By repeating this procedure, $g_2^{[n]}(x), g_3^{[n]}(x), \ldots, g_n^{[n]}(x)$ can be expressed as linear combinations of $g_{1}^{[n]}(x), \ldots, g_{1}^{[n]}(q^{n-1}x)$, i.e., for $k = 2, \ldots, n$,
\begin{equation}
g_{k}^{[n]}(x) = \prod_{i = 1}^{k-1} \frac{(-1)}{a^{[n]}_{i,i+1}} g_{1}^{[n]}(q^{k-1}x) + C_{k-2} g_{1}^{[n]}(q^{k-2}x) + \cdots + C_1 g_{1}^{[n]}(qx) + \prod_{i = 1}^{k - 1} \frac{1 - a^{[n]}_{ii}}{a^{[n]}_{i,i+1}} g_{1}^{[n]}(x), \label{eq:gkbyg1} 
\end{equation}
where $C_1, \ldots, C_{k-2}$ are constants.
Substituting (\ref{eq:gkbyg1}) into equation \eqref{eq:n}, we obtain a single $q$-difference equation of order $n$ for $g_1^{[n]}(x)$ as follows: 
\begin{align} \label{eq:n'} \tag{4.9.n'}
&\prod_{i = 1}^{n-1} \frac{(-1)}{a^{[n]}_{i,i+1}}  g_{1}^{[n]}(q^{n}x) + 
\cdots + \prod_{i = 1}^{n - 1} \frac{1 - a^{[n]}_{ii}}{a^{[n]}_{i,i+1}} g_{1}^{[n]}(qx) \\
&\quad = \frac{b^{[n]}_1 \alpha x}{1 - \alpha x} g_1^{[n]}(x) + \frac{b^{[n]}_2 \alpha x}{1 - \alpha x} \left( \frac{-1}{a^{[n]}_{12}} g_{1}^{[n]}(qx) + \frac{1 - a^{[n]}_{11}}{a^{[n]}_{12}} g_{1}^{[n]}(x) \right) + \cdots \nonumber \\
&\quad + \frac{b^{[n]}_{n-1} \alpha x}{1 - \alpha x} \left( \prod_{i = 1}^{n-2}  \frac{(-1)}{a^{[n]}_{i, i+1}} g_{1}^{[n]}(q^{n-2}x) + 
\cdots + \prod_{i = 1}^{n-2} \frac{1 - a^{[n]}_{ii}}{a^{[n]}_{i,i+1}} g_{1}^{[n]}(x) \right) \nonumber \\
&\quad +  \left( 1 + \frac{b^{[n]}_n \alpha x}{1 - \alpha x} \right) \left( \prod_{i = 1}^{n-1} \frac{(-1)}{a^{[n]}_{i,i+1}}  g_{1}^{[n]}(q^{n-1}x) + 
\cdots + \prod_{i = 1}^{n - 1} \frac{1 - a^{[n]}_{ii}}{a^{[n]}_{i,i+1}} g_{1}^{[n]}(x) \right). \nonumber
\end{align}
The coefficient of $g_{1}^{[n]}(q^{n}x)$ is $\displaystyle \prod_{i = 1}^{n-1}  \frac{(-1)}{a^{[n]}_{i,i+1}}$, 
while each of the other coefficients is of the form $\displaystyle \frac{\text{a linear polynomial in }x}{1 - \alpha x}$.
Therefore, multiplying the entire equation by $\displaystyle \prod_{i = 1}^{n-1} (-a^{[n]}_{i,i+1}) (\alpha x -1)$, we obtain
\begin{align*}
&(\alpha x - 1) g_1^{[n]}(q^n x) + (U_{n-1} x + W_{n-1}) g_1^{[n]}(q^{n-1} x) +  (U_{n-2} x + W_{n-2}) g_1^{[n]}(q^{n-2} x) \\
&\quad + \cdots + (U_{2} x + W_{2}) g_1^{[n]}(q^{2} x) + (U_{1} x + W_{1}) g_1^{[n]}(q x) + A(U_{0} x - 1) g_1^{[n]}(x) = 0,
\end{align*}
where $U_i$ and $W_i$ $(i=1, \ldots, n-1)$ are constants, and 
\begin{equation*}
A = (-1)^{n} \prod_{i=1}^{n-1} (1 - a^{[n]}_{ii}) = (-1)^n q^{\sum_{k=0}^{n-2}(k+1)(\mu_k - \lambda_k)} \neq 0, 
\end{equation*}
\begin{align*}
U_0 &= -\Biggl\{ \prod_{i = 1}^{n - 1} \frac{a^{[n]}_{i,i + 1}}{1 - a^{[n]}_{ii}} b^{[n]}_{1} + \prod_{i = 2}^{n - 1} \frac{a^{[n]}_{i,i + 1}}{1 - a^{[n]}_{ii}} b^{[n]}_{2} + \cdots + \prod_{i = n-2}^{n - 1} \frac{a^{[n]}_{i,i + 1}}{1 - a^{[n]}_{ii}} b^{[n]}_{n-2} \\
&\hspace{140pt} + \cdots + \frac{a^{[n]}_{n - 1,n}}{1 - a^{[n]}_{n-1, n-1}} b^{[n]}_{n-1} + b^{[n]}_{n} - 1 \Biggr\}  \alpha  .
\end{align*}
Thus, all the coefficients in equation \eqref{eq:gn1single-thm} are polynomials of degree at most $1$ in $x$.
Since $\alpha \neq 0$ and $A \neq 0$, if $U_0 \neq 0$, then equation \eqref{eq:gn1single-thm} is a generalized $q$-hypergeometric equation of order $n$. 
\end{proof}
For the explicit equation of order $3$ given by \eqref{eq:111-21-111-g1} and those of orders $4$-$6$ given in Appendix \ref{sec:geneq4-6}, we verify that $U_0 = q^{-(\lambda_0 + \cdots + \lambda_{n-2})}\beta \neq 0$ for $n = 3,4,5,6$.
These examples suggest that $U_0 = q^{-(\lambda_0 + \cdots + \lambda_{n-2})}\beta \neq 0$ also holds for arbitrary $n$.

Single $q$-difference equations satisfied by $g_{k}^{[n]}(x)$ $(k = 2,\ldots, n)$ can also be obtained, as in the case of $g_{1}^{[n]}(x)$, by eliminating all the other functions.
We omit their explicit derivations.

\section{$q$-Integral solutions to generalized $q$-hypergeometric equation of order $n$} \label{sec:Intsoldegn}
In this section, we investigate a $q$-integral representation of solutions to equation \eqref{eq:genqhge-n}.
Successive applications of the $q$-middle convolution yield a multiple integral representation of a solution.
In particular, a solution to equation \eqref{eq:genqhge-n} is represented by an $(n-1)$-fold $q$-integral.
We first give sufficient conditions for the convergence of this representation.
We then determine choices of the functions $K_{\lambda_j}$ and $y$, together with specializations of the 
parameters $\xi_j$, for which the multiple $q$-integral can be transformed into a $q$-hypergeometric series  ${}_n \phi_{n-1}$.

First, based on Theorem \ref{thm:suffcondconv}, we investigate sufficient conditions on parameters under which the $q$-difference equation
\begin{equation} \label{eq:Ygn}
\frac{\widetilde{V}^{[2(n-1)]}(qx) - \widetilde{V}^{[2(n-1)]}(x)}{-x} = \left[ \frac{G^{[2(n-1)]}_0}{x} + \frac{G^{[2(n-1)]}_1}{x - 1/\alpha} \right] \widetilde{V}^{[2(n-1)]}(x) 
\end{equation}
admits $q$-integral solutions.
From equations \eqref{mat:Bn+1-0} and  \eqref{mat:Bn+1-1}, the matrices $I_{n-1} - B^{[n-1]}_0$ and $I_{n-1} - B^{[n-1]}_0 - B^{[n-1]}_1$ are given as follows:
\begin{align*}
I_{n-1} - B^{[n-1]}_0
&= \begin{pmatrix}
  \widetilde{d}_{11} & \widetilde{d}_{12}     &              &        &                     & \\
               & \widetilde{d}_{22}     & \widetilde{d}_{23} &        & \text{\LARGE{O}}    & \\
               &                  & \ddots       & \ddots &                     & \\
               &                  &              & \ddots & \ddots              & \\
               & \text{\LARGE{O}} &              &        & \widetilde{d}_{n-2, n-2}  & \widetilde{d}_{n-2,n-1} \\
               &                  &              &        &                     & q^{\mu_{n-2}}
 \end{pmatrix} , \\
\end{align*}
\begin{align*}
&I_{n-1} - B^{[n-1]}_0 - B^{[n-1]}_1 \\
& = \begin{pmatrix}
  \widetilde{d}_{11} & \widetilde{d}_{12}     &              &        &                     & \\
       0       & \widetilde{d}_{22}     & \widetilde{d}_{23} &        & \text{\LARGE{O}}    & \\
   \vdots      &   \ddots         & \ddots       & \ddots &                     & \\
   \vdots      &                  &   \ddots     & \ddots & \ddots              & \\
       0       &    \cdots        &   \cdots     &   0    & \widetilde{d}_{n-2, n-2}  & \widetilde{d}_{n-2,n-1} \\
-q^{\mu_{n-2}}b^{[n-1]}_1 & -q^{\mu_{n-2}}b^{[n-1]}_2 & \cdots  & \cdots & -q^{\mu_{n-2}}b^{[n-1]}_{n-2} & q^{\mu_{n-2}} (1 - b^{[n-1]}_{n-1})
\end{pmatrix},
\end{align*}
where
\[
\widetilde{d}_{ii} = q^{\mu_{n-2}}(1 - a^{[n-1]}_{ii}) = q^{\sum_{k=i-1}^{n-2}\mu_k - \sum_{k=i-1}^{n-3}\lambda_k}, \quad 1 \le i \le n-2.
\]
Since $I_{n-1} - B^{[n-1]}_0$ is triangular, its $n - 1$ eigenvalues are given by its diagonal elements, namely for $3 \le n$,
\begin{equation}
q^{\sum_{k=i}^{n-2}\mu_k - \sum_{k=i}^{n-3}\lambda_k }, \quad i = 0, \ldots, n-2, \label{eq:evI-B0}  
\end{equation}
where empty sums and empty ranges are taken to be $0$.
Although it is difficult to determine the eigenvalues of $I_{n-1} - B^{[n-1]}_0 - B^{[n-1]}_1$ by direct computation, calculating in specific examples suggests that they are given by 
\begin{align}
& q^{\sum_{k=i}^{n-2}\mu_k - \sum_{k=i-1}^{n-3}\lambda_k }, \quad i = 1, \ldots, n-2,  \label{eq:evI-B0-B1} \\
& q^{\sum_{k=0}^{n-2}\mu_k - \sum_{k=0}^{n-3}\lambda_k }\beta/\alpha \nonumber
\end{align}
for $3 \le n$.
This conjecture is supported by the fact that their sum agrees with $tr(I_{n-1} - B^{[n-1]}_0 - B^{[n-1]}_1)$.
The formal $q$-integral representation of solutions to  equation (\ref{eq:Ygn})
 is given by
\begin{equation} \label{eq:solYg}
\widetilde{V}^{[2(n-1)]}(x)
 = \begin{pmatrix}
   \widetilde{v}_{1}^{[2(n-1)]}(x) \\
   \vdots \\
   \widetilde{v}_{n-1}^{[2(n-1)]}(x) \\
   \widetilde{v}_{n}^{[2(n-1)]}(x) \\
   \vdots \\
   \widetilde{v}_{2n-2}^{[2(n-1)]}(x) 
 \end{pmatrix}
 = \begin{pmatrix}
  \displaystyle \int_0^{\xi_{n-2} \infty} \frac{K_{\lambda_{n-2}}(x,s_{n-2})}{s_{n-2}} s_{n-2}^{\mu_{n-2}} g^{[n-1]}_{1} (s_{n-2}) d_q s_{n-2} \\
  \vdots \\
  \displaystyle \int_0^{\xi_{n-2} \infty} \frac{K_{\lambda_{n-2}}(x,s_{n-2})}{s_{n-2}} s_{n-2}^{\mu_{n-2}} g^{[n-1]}_{n-1} (s_{n-2}) d_q s_{n-2} \\
  \displaystyle \int_0^{\xi_{n-2} \infty} \frac{K_{\lambda_{n-2}}(x,s_{n-2})}{s_{n-2} - 1/\alpha} s_{n-2}^{\mu_{n-2}} g^{[n-1]}_{1} (s_{n-2}) d_q s_{n-2} \\
  \vdots \\
  \displaystyle \int_0^{\xi_{n-2} \infty} \frac{K_{\lambda_{n-2}}(x,s_{n-2})}{s_{n-2} - 1/\alpha} s_{n-2}^{\mu_{n-2}} g^{[n-1]}_{n-1} (s_{n-2}) d_q s_{n-2} 
  \end{pmatrix} .
\end{equation}
By reformulating Theorem \ref{thm:suffcondconv} in terms of equations \eqref{eq:Ygn} and \eqref{eq:solYg}, we obtain the following corollary.
\begin{cor} \label{prop:convn}
Assume that the integrands in equation \eqref{eq:solYg} are well-defined about $s_i = q^{m_i}\xi_i$ $(i = 0, 1, \ldots, n-2)$
for any $m_i \in \mathbb{Z}$.
If for every $j = 0, 1, \ldots, n-2$ $(n \in \mathbb{N}_{\ge 3})$ the absolute value of any eigenvalue of the matrix $I_{j+1} - B^{[j+1]}_0$ is strictly less than $1$, and the absolute value of any eigenvalue of the matrix $I_{j+1} - B^{[j+1]}_0 -B^{[j+1]}_1$ is strictly more than $|q^{\lambda_j}|$,
then the function $\widetilde{V}^{[2(n-1)]}(x)$ in equation \eqref{eq:solYg} converges and satisfies equation \eqref{eq:Ygn}.
\end{cor}
When $j = 0$, the eigenvalues of the matrices $I_1 - B^{[1]}_0$ and $I_1 - B^{[1]}_0 - B^{[1]}_1$ are given by equations \eqref{eq:I1B0} and \eqref{eq:I1B0B1}, respectively.
For $j = 1, \ldots, n-2$, by the same argument as in equation \eqref{eq:evI-B0}, the eigenvalues of the matrix $I_{j+1} - B^{[j+1]}_0$  are given by
\begin{equation}
q^{\sum_{k=i}^{j}\mu_k - \sum_{k=i}^{j-1}\lambda_k}, \quad i = 0, \ldots, j,  \label{eq:evI-B0j}
\end{equation}
where empty sums and empty ranges are understood to be $0$.
The eigenvalues of the matrix $I_{j+1} - B^{[j+1]}_0 - B^{[j+1]}_1$ are conjectured to be given by
\begin{align}
& q^{\sum_{k=i}^{j}\mu_k - \sum_{k=i-1}^{j-1}\lambda_k}, \quad i = 1, \ldots, j, \label{eq:evI-B0-B1j} \\
& q^{\sum_{k=0}^{j}\mu_k - \sum_{k=0}^{j-1}\lambda_k}\beta/\alpha , \nonumber
\end{align}
by the same argument as that for equation \eqref{eq:evI-B0-B1}.
If the conjectured eigenvalues in equation \eqref{eq:evI-B0-B1j} are correct, then the conditions in Corollary \ref{prop:convn} can be expressed explicitly in terms of the parameters.
We assume that the conditions of Corollary \ref{prop:convn} hold and consider the solution to equation (\ref{eq:genqhge-n}).
Let $\widetilde{V}^{[2(n-1)]}(x)$ be a solution to equation (\ref{eq:Ygn}) and  set
\begin{equation}
\begin{pmatrix}
\widetilde{g}_1(x) \\
\vdots \\
\widetilde{g}_n(x) \\
\widetilde{g}_{n+1}(x) \\
\vdots \\
\widetilde{g}_{2n-2}(x) 
\end{pmatrix}
 = (P^{[2(n-1)]})^{-1} \widetilde{V}^{[2(n-1)]}(x)
 = \begin{pmatrix}
    I_{n-1} & O_{n-1} \\
    O_{n-1} & p_{n-1}^{-1}
   \end{pmatrix}
\widetilde{V}^{[2(n-1)]}(x),
\end{equation}
where $P^{[2(n-1)]}$ is the $2(n-1) \times 2(n-1)$ invertible matrix given by
\begin{equation*}
P^{[2(n-1)]}
= \begin{pmatrix}
   I_{n-1} & O_{n-1} \\
   O_{n-1} & p_{n-1}
  \end{pmatrix}, \ 
 p_{n-1} 
= \begin{pmatrix}
    1      & -b^{[n-1]}_2 & -b^{[n-1]}_3  & \dots  & -b^{[n-1]}_{n-2} & -b^{[n-1]}_{n-1} \\
    0      & b^{[n-1]}_1  & 0           & \dots  &    0           & 0 \\
    \vdots & 0          & b^{[n-1]}_1   & \ddots &                & \vdots \\
    \vdots & \vdots     &  0          & \ddots & \ddots         & \vdots \\
    \vdots & \vdots     &  \vdots     & \ddots & b^{[n-1]}_{1}    & 0 \\
    0      & 0          & 0           & \dots  & 0              & b^{[n-1]}_{1}
\end{pmatrix} .
\end{equation*}
Then ${}^t(\widetilde{g}_1(x) , \ldots, \widetilde{g}_{n-1}(x), \widetilde{g}_n(x))$ 
coincides with ${}^t(g^{[n]}_1(x) , \ldots, g^{[n]}_n(x))$ which satisfies equation (\ref{eq:genqhge-n}).
In particular, $g_1^{[n]}(x) = \widetilde{v}^{[2(n-1)]}_{1}(x)$ is expressed as the following iterated $q$-integral representation:  
\begin{align} \label{eq:g1n-2sol}
g_1^{[n]}(x) 
&= \widetilde{v}^{[2(n-1)]}_{1}(x) \\
&= \int^{\xi_{n-2} \infty}_{0} \frac{K_{\lambda_{n-2}}(x, s_{n-2})}{s_{n-2}} s_{n-2}^{\mu_{n-2}}
  \int^{\xi_{n-3} \infty}_{0} \frac{K_{\lambda_{n-3}}(s_{n-2}, s_{n-3})}{s_{n-3}} s_{n-3}^{\mu_{n-3}} \nonumber \\
&\quad \dots
\int^{\xi_{2} \infty}_{0} \frac{K_{\lambda_{2}}(s_{3}, s_{2})}{s_{2}} s_{2}^{\mu_{2}}
\int^{\xi_{1} \infty}_{0} \frac{K_{\lambda_{1}}(s_{2}, s_{1})}{s_{1}} s_{1}^{\mu_{1}} \nonumber \\
&\qquad \cdot \int^{\xi_{0} \infty}_{0} \frac{K_{\lambda_{0}}(s_{1}, s_{0})}{s_{0}} y(s_0)\, d_qs_0 \, d_qs_1 \, d_qs_2 \, \dots d_qs_{n-3} \, d_qs_{n-2}. \nonumber
\end{align}
As demonstrated for $n = 3$ in Section \ref{sec:qmcex}, by choosing the functions $K_{\lambda_j}$ $(j = 0, 1, \ldots, n-2)$ and $y$ appropriately and by suitably specializing the parameters $\xi_j$ in equation \eqref{eq:g1n-2sol}, we can transform the integral representation \eqref{eq:g1n-2sol} into a $q$-hypergeometric series ${}_n \phi_{n-1}$.
The following four cases lead to $q$-hypergeometric series representations.
Note that there are two choices for $K_{\lambda_j}$, namely,
\begin{equation*}
K^{(1)}_{\lambda_j}(x, s) =  x^{-\lambda_j} \frac{(q^{\lambda_j + 1}s/x;q)_{\infty}}{(qs/x;q)_{\infty}} , \quad K^{(2)}_{\lambda_j}(x, s) =  s^{-\lambda_j} \frac{(x/s;q)_{\infty}}{(q^{-\lambda_j}x/s;q)_{\infty}} .
\end{equation*}
\begin{enumerate}[label=Case \arabic*]
\item \label{case:1} : For $j = 0, 1, \ldots, n-2$,
\begin{equation} \label{eq:Klyxi-1} 
K_{\lambda_j} = K^{(1)}_{\lambda_j}, \quad
y(s_0) = s_0^{\mu_0} \frac{(\alpha s_0; q)_{\infty}}{(\beta s_0; q)_{\infty}}, \quad
\xi_j = q^{- \sum_{\ell = j}^{n-2}\lambda_{\ell}} x, \quad
\lambda_{j} \not \in \mathbb{Z} .
\end{equation}
\item \label{case:2} : For $j = 0, 1, \ldots, n-2$ and $r = 0, 1, \ldots, n-2$,
\begin{align}
& K_{\lambda_j} = 
\begin{cases}
K_{\lambda_j}^{(2)}, & j=r, \\
K_{\lambda_j}^{(1)}, & j \neq r,
\end{cases} \quad
y(s_0) = s_0^{\widetilde{\mu}} \frac{(q/(\beta s_0);q)_{\infty}}{(q/(\alpha s_0);q)_{\infty}} \quad (q^{\widetilde{\mu}} = q^{\mu_0}\beta/\alpha), \label{eq:Klyxi-2} \\
& \xi_0 = 1/\beta, \quad 
\xi_j =
\begin{cases}
q^{\sum_{\ell = 0}^{j-1} \lambda_{\ell}} /\beta, & 1 \le j \le r, \\
q^{-\sum_{\ell = j}^{n-2} \lambda_{\ell}}x, & r+1 \le j \le n-2, 
\end{cases} \nonumber \\
& \lambda_{j} \not \in \mathbb{Z} \quad (j \neq r). \nonumber
\end{align}
\item \label{case:3}: For $j = 0, 1, \ldots, n-2$,
\begin{align} 
& K_{\lambda_j} = K^{(2)}_{\lambda_j}, \quad
y(s_0) = s_0^{\widetilde{\mu}} \frac{(q/(\beta s_0); q)_{\infty}}{(q/(\alpha s_0); q)_{\infty}} \quad (q^{\widetilde{\mu}} = q^{\mu_0} \beta/\alpha), \label{eq:Klyxi-3} \\
& \xi_j = x/q^{n-j-1} , \quad
\lambda_{j} \not \in \mathbb{Z} . \nonumber
\end{align}
\item \label{case:4}: For $j = 0, 1, \ldots, n-2$ and $r = 0, 1, \ldots, n-2$,
\begin{align}
& K_{\lambda_j} = 
\begin{cases}
K_{\lambda_j}^{(1)}, & j=r, \\
K_{\lambda_j}^{(2)}, & j \neq r,
\end{cases} \quad
y(s_0) = s_0^{\mu_0} \frac{(\alpha s_0; q)_{\infty}}{(\beta s_0; q)_{\infty}},  \label{eq:Klyxi-4} \\
& \xi_j =
\begin{cases}
q^{j+1}/\alpha, & 0 \le j \le r, \\
x/q^{n-j-1}, & r+1 \le j \le n-2, 
\end{cases} \nonumber \\
& \lambda_{j} \not \in \mathbb{Z} \quad (j \neq r). \nonumber
\end{align}
\end{enumerate}
\begin{prop} \label{prop:qhgssol}
Assume that the conditions on the eigenvalues of matrices $I_{j+1} - B^{[j+1]}_0$ and $I_{j+1} - B^{[j+1]}_0 - B^{[j+1]}_1$ $(j = 0, 1, \ldots, n-2$, $n \in \mathbb{N}_{\ge 3})$ in Corollary \ref{prop:convn} are satisfied. 
Set
\[
L_k = \sum_{i=0}^{k}\lambda_i,\qquad
\widetilde{M}_k=\sum_{i=0}^{k}\mu_i, \qquad 
T_k = \sum_{i = k}^{n-2}\lambda_i ,
\]
with the conventions
\[
L_{-1} = \widetilde{M}_{-1} = 0, \qquad T_{n-1}=0 .
\]
\noindent
(i) If $|q^{ -L_{n-2}} \beta x| < 1 $ holds, then equation \eqref{eq:g1n-2sol} specialized to \ref{case:1} or \ref{case:2} can be transformed as follows: \\
\ref{case:1}:
\begin{align}
g_1^{[n]}(x)
&= (1-q)^{n-1}
q^{-\mu_0 L_{n-2}
-\sum_{j=1}^{n-2}(\mu_j-\lambda_{j-1})T_j } \; x^{\widetilde{M}_{n-2}-L_{n-2}} 
\label{eq:g1-c1} \\
&\cdot \frac{( \overbrace{q, q, \ldots , q}^{n-1}  ;q)_{\infty}}{(q^{-\lambda_0+1},\ldots,q^{-\lambda_{n-2}+1};q)_{\infty}} \;
\frac{(\bar{b}_1,\ldots,\bar{b}_{n-1};q)_{\infty}}{(\bar{a}_1,\ldots,\bar{a}_{n-1};q)_{\infty}} \nonumber \\
&\cdot {}_n\phi_{n-1}
\!\left(
\begin{matrix}
\alpha/\beta,\,
\bar{a}_1,\ldots,\bar{a}_{n-1}\\
\bar{b}_1,\ldots,\bar{b}_{n-1}
\end{matrix}; q,\, q^{-L_{n-2}}\beta x
\right), \nonumber
\end{align}
where
\[
\bar{a}_k=q^{\widetilde{M}_{k-1}-L_{k-2}},\qquad
\bar{b}_k=q^{\widetilde{M}_{k-1}-L_{k-1}+1},
\qquad (k=1,\ldots,n-1).
\]
\ref{case:2}: For $r = 0, 1, \ldots, n-2$, \\
\begin{align}
g_1^{[n]}(x)
&= (1-q)^{n-1} q^{E_r} \;
\beta^{-\widetilde{\mu} - \sum_{j=1}^{r}\mu_j + L_r} \; x^{\sum_{j=r+1}^{n-2}(\mu_j-\lambda_j)} \label{eq:g1-c2} \\
&\quad \cdot
\frac{(\overbrace{q,q,\ldots,q}^{n-1};q)_{\infty}}{(q\beta/\alpha;q)_{\infty} \displaystyle \prod_{\substack{0\le j\le n-2 \\ j\neq r}}(q^{-\lambda_j+1};q)_{\infty}} \;
\frac{(\bar{b}_1,\ldots,\bar{b}_{n-1};q)_{\infty}}
{(\bar{a}_1,\ldots,\bar{a}_{n-1};q)_{\infty}} \nonumber \\
&\quad\cdot
{}_n\phi_{n-1} \left({q^{\lambda_r},\bar{a}_1,\ldots,\bar{a}_{n-1}
\atop \bar{b}_1,\ldots,\bar{b}_{n-1}};q, q^{-L_{n-2}}\beta x \right), \nonumber
\end{align}
where
\begin{align*}
& E_r = \sum_{j=1}^{r} L_{j-1}(\mu_j-\lambda_{j-1}) - \sum_{j=r+1}^{n-2} T_j(\mu_j-\lambda_{j-1}) - \lambda_r(L_{r-1} + T_{r+1}) , \\
& \bar{a}_1 = q^{-\widetilde{M}_r + L_r} \alpha/\beta, \\
& \bar{a}_k =
\begin{cases}
q^{-\sum_{j=k-1}^{r}\mu_j + \sum_{j=k-2}^{r}\lambda_j}, \quad 2\le k \le r+1, \\
q^{\sum_{j=r+1}^{k-1}\mu_j - \sum_{j=r+1}^{k-2}\lambda_j}, \quad r+2\le k \le n-1,
\end{cases} \\
& \bar{b}_k =
\begin{cases}
q^{\sum_{j=k-1}^{r}(-\mu_j + \lambda_j) + 1}, \quad 1\le k \le r+1, \\
q^{\sum_{j=r+1}^{k-1}(\mu_j - \lambda_j) + 1}, \quad r+2\le k \le n-1. 
\end{cases}
\end{align*}
Empty sums and empty ranges are understood to be $0$ and omitted.\\
(ii) If $|q^{n}/( \alpha x)| <1 $ holds, then equation \eqref{eq:g1n-2sol} specialized to \ref{case:3} or \ref{case:4} can be transformed as follows: \\
\ref{case:3}:
\begin{align}
g_1^{[n]}(x)
&=(1-q)^{n-1} q^{\sum_{j=0}^{n-2}(n-1-j)(-\mu_j + \lambda_j)}
\left(\frac{\alpha}{\beta}\right)^{n-1} x^{\widetilde{\mu} + \sum_{j=1}^{n-2}\mu_j - L_{n-2} } \label{eq:g1-c3} \\
&\quad\cdot \frac{(\overbrace{q, q, \ldots , q}^{n-1} ;q)_{\infty}}{(q^{-\lambda_{0}+1}, \ldots, q^{-\lambda_{n-2}+1};q)_{\infty}} \; 
\frac{(\bar{b}_1, \ldots , \bar{b}_{n-1} ;q)_{\infty}}{(\bar{a}_1, \ldots , \bar{a}_{n-1};q)_{\infty}} \nonumber \\
&\quad\cdot
{}_n\phi_{n-1}
\left(
{\alpha/\beta,\bar{a}_1,\ldots,\bar{a}_{n-1}
\atop \bar{b}_1,\ldots,\bar{b}_{n-1}}; q, q^n/(\alpha x)
\right), \nonumber
\end{align}
where
\begin{align*}
& \bar{a}_k = q^{-\widetilde{M}_{k-1} + L_{k-1}} \alpha/\beta, \quad 1\le k \le n-1, \\
& \bar{b}_k = q^{-\widetilde{M}_{k-1} + L_{k-2} + 1} \alpha/\beta, \quad 1 \le k \le n-1.
\end{align*}
\ref{case:4}: For $r = 0, 1, \ldots, n-2$,
\begin{align}
g_1^{[n]}(x)
&=(1-q)^{n-1} q^{F_r} \;
\alpha^{-\widetilde{M}_{r} + L_{r-1}} \;
x^{\sum_{j=r+1}^{n-2}\mu_j - T_r }  \label{eq:g1-c4} \\
&\quad\cdot
\frac{(\overbrace{q, q, \ldots , q}^{n-1} ;q)_{\infty} }
{(q\beta/\alpha;q)_{\infty} \displaystyle \prod_{\substack{0\le j\le n-2 \\ j\neq r}}(q^{-\lambda_j+1};q)_{\infty}} \; 
\frac{(\bar{b}_1, \ldots , \bar{b}_{n-1} ;q)_{\infty}}{(\bar{a}_1, \ldots , \bar{a}_{n-1};q)_{\infty}} \nonumber \\
&\quad\cdot
{}_n\phi_{n-1}
\left(
{q^{\lambda_r},\bar a_1,\ldots,\bar a_{n-1}
\atop
\bar b_1,\ldots,\bar b_{n-1}}; q, q^n/(\alpha x)
\right), \nonumber
\end{align}
where
\begin{align*}
F_r
&= \sum_{j=0}^{r}(j+1)\mu_j - \sum_{j=0}^{r-1}(j+1)\lambda_j - \sum_{j=r+1}^{n-2}(n-1-j)\mu_j + \sum_{j=r}^{n-2}(n-1-j)\lambda_j - \lambda_r,
\end{align*}
\begin{align*}
&\bar{a}_k=
\begin{cases}
q^{\sum_{j=k-1}^{r}\mu_j - \sum_{j=k-1}^{r-1}\lambda_j}, &1\le k \le r+1,\\
q^{-\sum_{j=r+1}^{k-1}\mu_j + \sum_{j=r}^{k-1}\lambda_j}, &r+2\le k \le n-1,
\end{cases} \\
&\bar{b}_1
= q^{\widetilde{M}_r - L_{r-1} + 1} \beta/\alpha, \\
&\bar{b}_k=
\begin{cases}
q^{\sum_{j=k-1}^{r}\mu_j - \sum_{j=k-2}^{r-1}\lambda_j + 1}, &2\le k \le r+1,\\
q^{-\sum_{j=r+1}^{k-1}\mu_j + \sum_{j=r}^{k-2}\lambda_j + 1}, &r+2\le k \le n-1.
\end{cases}
\end{align*}
Empty sums and empty ranges are understood to be $0$ and omitted.
\end{prop}
\begin{proof}
We give the proof for \ref{case:2}, as the computation in this case is particularly involved. 
From equations \eqref{eq:g1n-2sol} and \eqref{eq:Klyxi-2}, we obtain 
\begin{align} 
g_1^{[n]}(x)
&= \int^{\xi_{n-2} \infty}_{0} \frac{K^{(1)}_{\lambda_{n-2}}(x, s_{n-2})}{s_{n-2}} s_{n-2}^{\mu_{n-2}}
  \int^{\xi_{n-3} \infty}_{0} \frac{K^{(1)}_{\lambda_{n-3}}(s_{n-2}, s_{n-3})}{s_{n-3}} s_{n-3}^{\mu_{n-3}} \label{eq:g1-c2Ks} \\
&\quad \cdots s_{r+2}^{\mu_{r+2}} \int^{\xi_{r+1} \infty}_{0} \frac{K^{(1)}_{\lambda_{r+1}}(s_{r+2}, s_{r+1})}{s_{r+1}}
s_{r+1}^{\mu_{r+1}}\int^{\xi_{r} \infty}_{0} \frac{K^{(2)}_{\lambda_{r}}(s_{r+1}, s_{r})}{s_{r}}
s_{r}^{\mu_{r}} \nonumber \\
&\qquad \cdot \int^{\xi_{r-1} \infty}_{0} \frac{K^{(1)}_{\lambda_{r-1}}(s_{r}, s_{r-1})}{s_{r-1}} \cdots
s_{3}^{\mu_{3}} \int^{\xi_{2} \infty}_{0} \frac{K^{(1)}_{\lambda_{2}}(s_{3}, s_{2})}{s_{2}} s_{2}^{\mu_{2}}
\int^{\xi_{1} \infty}_{0} \frac{K^{(1)}_{\lambda_{1}}(s_{2}, s_{1})}{s_{1}} s_{1}^{\mu_{1}} \nonumber \\
&\qquad \cdot \int^{\xi_{0} \infty}_{0} \frac{K^{(1)}_{\lambda_{0}}(s_{1}, s_{0})}{s_{0}} s_0^{\widetilde{\mu}} \frac{(q/(\beta s_0); q)_{\infty}}{(q/(\alpha s_0); q)_{\infty}}\, d_qs_0 \; d_qs_1 \cdots d_qs_r \cdots d_qs_{n-3} \, d_qs_{n-2}. \nonumber
\end{align}
Note that, when $r = 0$,
\begin{align*} 
g_1^{[n]}(x)
&= \int^{\xi_{n-2} \infty}_{0} \frac{K^{(1)}_{\lambda_{n-2}}(x, s_{n-2})}{s_{n-2}} s_{n-2}^{\mu_{n-2}}
  \int^{\xi_{n-3} \infty}_{0} \frac{K^{(1)}_{\lambda_{n-3}}(s_{n-2}, s_{n-3})}{s_{n-3}} s_{n-3}^{\mu_{n-3}}  \\
&\quad \cdots s_{3}^{\mu_{3}} 
\int^{\xi_{2} \infty}_{0} \frac{K^{(1)}_{\lambda_{2}}(s_{3}, s_{2})}{s_{2}} s_{2}^{\mu_{2}}
\int^{\xi_{1} \infty}_{0} \frac{K^{(1)}_{\lambda_{1}}(s_{2}, s_{1})}{s_{1}} s_{1}^{\mu_{1}} \\
&\qquad \cdot \int^{\xi_{0} \infty}_{0} \frac{K^{(2)}_{\lambda_{0}}(s_{1}, s_{0})}{s_{0}} s_0^{\widetilde{\mu}} \frac{(q/(\beta s_0); q)_{\infty}}{(q/(\alpha s_0); q)_{\infty}}\, d_qs_0 \, d_qs_1 \cdots d_qs_{n-3} \, d_qs_{n-2},
\end{align*}
whereas, when $r = n-2$,
\begin{align*} 
g_1^{[n]}(x)
&= \int^{\xi_{n-2} \infty}_{0} \frac{K^{(2)}_{\lambda_{n-2}}(x, s_{n-2})}{s_{n-2}} s_{n-2}^{\mu_{n-2}}
  \int^{\xi_{n-3} \infty}_{0} \frac{K^{(1)}_{\lambda_{n-3}}(s_{n-2}, s_{n-3})}{s_{n-3}} s_{n-3}^{\mu_{n-3}}  \\
&\quad \cdots s_{3}^{\mu_{3}}
\int^{\xi_{2} \infty}_{0} \frac{K^{(1)}_{\lambda_{2}}(s_{3}, s_{2})}{s_{2}} s_{2}^{\mu_{2}}
\int^{\xi_{1} \infty}_{0} \frac{K^{(1)}_{\lambda_{1}}(s_{2}, s_{1})}{s_{1}} s_{1}^{\mu_{1}} \nonumber \\
&\qquad \cdot \int^{\xi_{0} \infty}_{0} \frac{K^{(1)}_{\lambda_{0}}(s_{1}, s_{0})}{s_{0}} s_0^{\widetilde{\mu}} \frac{(q/(\beta s_0); q)_{\infty}}{(q/(\alpha s_0); q)_{\infty}}\, d_qs_0 \, d_qs_1 \cdots d_qs_{n-3} \, d_qs_{n-2}. \nonumber
\end{align*}
For $n \ge 4$, we now consider the case $1 \le r \le n-3$. 
Equation \eqref{eq:g1-c2Ks} can be written as follows:
\begin{align}
&g_1^{[n]}(x) 
= x^{-\lambda_{n-2}} \int_0^{\xi_{n-2} \infty} s_{n-2}^{-1} \frac{(q^{\lambda_{n-2} + 1}s_{n-2}/x;q)_{\infty}}{(qs_{n-2}/x;q)_{\infty}} s_{n-2}^{\mu_{n-2}} \cdot  s_{n-2}^{-\lambda_{n-3}}   \label{eq:gn1-p1}\\
& \quad \cdot \int_0^{\xi_{n-3} \infty} s_{n-3}^{-1} \frac{(q^{\lambda_{n-3} + 1}s_{n-3}/s_{n-2} ;q)_{\infty}}{(qs_{n-3}/s_{n-2}  ;q)_{\infty}} s_{n-3}^{\mu_{n-3}} \cdot s_{n-3}^{-\lambda_{n-4}} \nonumber \\
& \cdots
s_{r+2}^{\mu_{r+2}} \cdot s_{r+2}^{-\lambda_{r+1}}
\int_0^{\xi_{r+1} \infty} s_{r+1}^{-1} \frac{(q^{\lambda_{r+1} + 1}s_{r+1}/s_{r+2} ;q)_{\infty}}{(qs_{r+1}/s_{r+2} ;q)_{\infty}} s_{r+1}^{\mu_{r+1}} \nonumber \\ 
& \cdot \int_0^{\xi_{r} \infty} s_{r}^{- \lambda_{r} -1} \frac{(s_{r+1}/s_{r}; q)_{\infty}}{(q^{-\lambda_{r}} s_{r+1}/s_{r}; q)_{\infty}} 
s_{r}^{\mu_{r}} \cdot s_{r}^{-\lambda_{r-1}} \int_0^{\xi_{r-1} \infty} s_{r-1}^{-1} \frac{(q^{\lambda_{r-1} + 1}s_{r-1}/s_{r} ;q)_{\infty}}{(qs_{r-1}/s_{r}  ;q)_{\infty}} s_{r-1}^{\mu_{r-1}} \cdot s_{r-1}^{-\lambda_{r-2}} \nonumber \\
& \cdots
s_3^{\mu_3} \cdot s_3^{-\lambda_2}  \int_0^{\xi_{2} \infty} s_2^{-1} \frac{(q^{\lambda_2 + 1}s_2/s_3 ;q)_{\infty}}{(q s_2/s_3 ;q)_{\infty}} s_2^{\mu_2} \cdot s_2^{-\lambda_1}  \int_0^{\xi_{1} \infty} s_1^{-1} \frac{(q^{\lambda_1 + 1}s_1/s_2 ;q)_{\infty}}{(q s_1/s_2 ;q)_{\infty}} s_1^{\mu_1} \cdot s_1^{-\lambda_0} \nonumber \\ 
& \cdot \int_0^{\xi_{0} \infty} s_0^{\widetilde{\mu} - 1} \frac{(q^{\lambda_0 + 1}s_0/s_1, q/(\beta s_0) ;q)_{\infty}}{(q s_0/s_1, q/(\alpha s_0) ;q)_{\infty}} \; d_q s_0 d_q s_1 \cdots d_q s_r \cdots d_q s_{n-3} d_q s_{n-2} \nonumber .
\end{align}
Equation \eqref{eq:gn1-p1} can be rewritten further as follows: 
\begin{align}
&g_1^{[n]}(x) 
= (1 - q)^{n-1} \label{eq:gn1-p2} \\
&\cdot \xi_0^{\widetilde{\mu}} \xi_1^{\mu_1 - \lambda_0} \cdots \xi_{r-1}^{\mu_{r-1} - \lambda_{r-2}} \xi_{r}^{\mu_{r} - (\lambda_{r-1} + \lambda_{r})} \xi_{r+1}^{\mu_{r+1}} \xi_{r+2}^{\mu_{r+2} - \lambda_{r+1}} \cdots \xi_{n-3}^{\mu_{n-3} - \lambda_{n-4}} \xi_{n-2}^{\mu_{n-2} - \lambda_{n-3}} x^{-\lambda_{n-2}} \nonumber \\
&\cdot  
\sum_{m_{n-2} = -\infty}^{\infty} \sum_{m_{n-3} = -\infty}^{\infty} \cdots \sum_{m_{r+1} = -\infty}^{\infty} \sum_{m_{r} = -\infty}^{\infty} \sum_{m_{r-1} = -\infty}^{\infty} \cdots \sum_{m_{1} = -\infty}^{\infty} \sum_{m_{0} = -\infty}^{\infty} \nonumber \\
& \cdot (q^{\mu_{n-2} - \lambda_{n-3}})^{m_{n-2}} (q^{\mu_{n-3} - \lambda_{n-4}})^{m_{n-3}}  \cdots (q^{\mu_{r+2} - \lambda_{r+1}})^{m_{r+2}} (q^{\mu_{r+1}})^{m_{r+1}} \nonumber \\
&\cdot (q^{\mu_{r} - (\lambda_{r-1} + \lambda_{r})})^{m_{r}} (q^{\mu_{r-1} - \lambda_{r-2}})^{m_{r-1}} \cdots
(q^{\mu_{2} - \lambda_{1}})^{m_{2}} (q^{\mu_{1} - \lambda_{0}})^{m_{1}} (q^{\widetilde{\mu}})^{m_0} \nonumber \\ 
& \cdot \frac{(q^{\lambda_{n-2} + m_{n-2} + 1}\xi_{n-2}/x, q^{\lambda_{n-3} + m_{n-3} - m_{n-2} + 1}\xi_{n-3}/\xi_{n-2}, \ldots,  
q^{\lambda_{r+1} + m_{r+1} - m_{r+2} + 1}\xi_{r+1}/\xi_{r+2}; q)_{\infty}}
{(q^{m_{n-2} + 1}\xi_{n-2}/x, q^{m_{n-3} - m_{n-2} + 1}\xi_{n-3}/\xi_{n-2}, \ldots , q^{m_{r+1} - m_{r+2} + 1}\xi_{r+1}/\xi_{r+2}; q)_{\infty}} \nonumber \\
& \cdot \frac{(q^{m_{r+1} - m_{r}}\xi_{r+1}/\xi_{r}; q)_{\infty}}{(q^{-\lambda_{r} + m_{r+1} - m_{r}}\xi_{r+1}/\xi_{r}; q)_{\infty}} \nonumber \\
& \cdot \frac{(q^{\lambda_{r-1} + m_{r-1} - m_{r} + 1}\xi_{r-1}/\xi_{r}, \ldots , q^{\lambda_{1} + m_{1} - m_{2} + 1}\xi_{1}/\xi_{2}, q^{\lambda_{0} + m_{0} - m_{1} + 1}\xi_{0}/\xi_{1}, q^{1 - m_0}/(\beta\xi_0) ; q)_{\infty}}{(q^{m_{r-1} - m_{r} + 1}\xi_{r-1}/\xi_{r}, \ldots , q^{m_{1} - m_{2} + 1}\xi_{1}/\xi_{2}, q^{m_{0} - m_{1} + 1}\xi_{0}/\xi_{1}, q^{1 - m_0}/(\alpha\xi_0); q)_{\infty}}   \nonumber 
\end{align}
Set in equation \eqref{eq:gn1-p2}
\[
  \xi_0 = 1/\beta, \quad \xi_j =
\begin{cases}
q^{L_{j-1}} /\beta, & 1 \le j \le r, \\
q^{-T_j}x, & r+1 \le j \le n-2 .
\end{cases}
\]
Then we have
\begin{align}
&g_1^{[n]}(x) \label{eq:gn1-12} \\
&= (1 - q)^{n-1} \beta^{-\widetilde{\mu}} 
(q^{\lambda_0}/\beta)^{\mu_1-\lambda_0} (q^{\lambda_0 + \lambda_1}/\beta)^{\mu_2-\lambda_1} \cdots (q^{\lambda_0 + \cdots + \lambda_{r-2}}/\beta)^{\mu_{r-1}-\lambda_{r-2}} \nonumber \\
&\cdot (q^{\lambda_0 + \cdots + \lambda_{r-1}}/\beta)^{\mu_{r}-(\lambda_{r-1} + \lambda_{r})} (q^{ -(\lambda_{r+1} + \cdots + \lambda_{n-2}) } x)^{\mu_{r+1}}  \nonumber \\
&\cdot (q^{ -(\lambda_{r+2} + \cdots + \lambda_{n-2}) } x)^{\mu_{r+2}-\lambda_{r+1}} \cdots (q^{ -(\lambda_{n-3} + \lambda_{n-2}) } x)^{\mu_{n-3}-\lambda_{n-4}} (q^{ -\lambda_{n-2}} x)^{\mu_{n-2}-\lambda_{n-3}}
x^{-\lambda_{n-2}}  \nonumber \\
&\cdot
\sum_{m_{n-2} = -\infty}^{\infty} \sum_{m_{n-3} = -\infty}^{\infty} \cdots \sum_{m_{r+1} = -\infty}^{\infty} \sum_{m_{r} = -\infty}^{\infty} \sum_{m_{r-1} = -\infty}^{\infty} \cdots \sum_{m_{1} = -\infty}^{\infty} \sum_{m_{0} = -\infty}^{\infty} \nonumber \\
& \cdot (q^{\mu_{n-2} - \lambda_{n-3}})^{m_{n-2}} (q^{\mu_{n-3} - \lambda_{n-4}})^{m_{n-3}}  \cdots (q^{\mu_{r+2} - \lambda_{r+1}})^{m_{r+2}} (q^{\mu_{r+1}})^{m_{r+1}} \nonumber \\
&\cdot (q^{\mu_{r} - (\lambda_{r-1} + \lambda_{r})})^{m_{r}} (q^{\mu_{r-1} - \lambda_{r-2}})^{m_{r-1}} \cdots
(q^{\mu_{2} - \lambda_{1}})^{m_{2}} (q^{\mu_{1} - \lambda_{0}})^{m_{1}} (q^{\widetilde{\mu}})^{m_0} \nonumber \\ 
& \cdot \frac{(q^{ m_{n-2} + 1}, q^{ m_{n-3} - m_{n-2} + 1}, \ldots,  q^{ m_{r+1} - m_{r+2} + 1}; q)_{\infty}}
{(q^{-\lambda_{n-2} + m_{n-2} + 1}, q^{-\lambda_{n-3} + m_{n-3} - m_{n-2} + 1}, \ldots , q^{-\lambda_{r+1} + m_{r+1} - m_{r+2} + 1}; q)_{\infty}} \nonumber \\
&\cdot \frac{(q^{-L_{n-2} + \lambda_r + m_{r+1} - m_r} \beta x ; q)_{\infty}}{(q^{-L_{n-2} + m_{r+1} - m_r} \beta x ; q)_{\infty}} \nonumber \\
&\cdot \frac{(q^{m_{r-1} - m_{r} + 1}, \ldots ,  q^{m_{1} - m_{2} + 1}, q^{m_{0} - m_{1} + 1}, q^{1 - m_0} ;q)_{\infty}}{(q^{-\lambda_{r-1} + m_{r-1} - m_{r} + 1}, \ldots ,  q^{-\lambda_{1} + m_{1} - m_{2} + 1}, q^{-\lambda_{0} + m_{0} - m_{1} + 1}, q^{1 - m_0}\beta/\alpha;q)_{\infty}} . \nonumber 
\end{align}
We replace the variables $-m_0$ and $m_{n-2}$ with $M_0$ and  $M_{n-2}$, respectively.
Moreover, by setting $m_{i-1} - m_{i} = M_{i}$ for $i= 1, \ldots, r$, the variables $(m_1, \ldots, m_{r})$ are transformed into $(-(M_0 + M_1), \ldots, -(M_0 + \cdots + M_{r}))$, and by setting $m_j - m_{j+1} = M_{j}$ for $j=r+1, \ldots, n-3$, the variables $(m_{r+1}, \ldots, m_{n-3})$ are transformed into $(M_{r+1} + \cdots + M_{n-2}, \ldots, M_{n-3} + M_{n-2})$.
It follows from $(q^{\ell+1};q)_{\infty} = 0$ for $\ell \in \mathbb{Z}_{<0}$ that
\begin{align}
&g_1^{[n]}(x) \label{eq:gn1-13} \\
&= (1 - q)^{n-1} q^{E_r}
\; \beta^{-\widetilde{\mu} - \sum_{j=1}^{r}\mu_j + L_{r}} \; x^{\sum_{j=r+1}^{n-2} (\mu_j - \lambda_j)}  \nonumber \\
&\sum_{M_{n-2} = 0}^{\infty} \sum_{M_{n-3} = 0}^{\infty} \cdots \sum_{M_{r+1} = 0}^{\infty} \sum_{M_{r} = 0}^{\infty} \sum_{M_{r-1} = 0}^{\infty} \cdots \sum_{M_{1} = 0}^{\infty} \sum_{M_{0} = 0}^{\infty} \nonumber \\
&\cdot (q^{\mu_{n-2} - \lambda_{n-3}})^{M_{n-2}} (q^{\mu_{n-3} - \lambda_{n-4}})^{M_{n-3} + M_{n-2}} \cdots (q^{\mu_{r+2} - \lambda_{r+1}})^{M_{r+2} + \cdots + M_{n-2}} \nonumber \\
&\cdot (q^{\mu_{r+1}})^{M_{r+1} + \cdots + M_{n-2}} (q^{\mu_{r} - (\lambda_{r-1}+ \lambda_{r})})^{-(M_0 + \cdots + M_{r})}  \nonumber \\
&\cdot (q^{\mu_{r-1} - \lambda_{r-2}})^{-(M_0 + \cdots + M_{r-1})} \cdots (q^{\mu_1 - \lambda_0})^{-(M_0 + M_1)} (q^{\mu_0}\beta/\alpha)^{-M_0}  \nonumber \\
&\cdot \frac{(q^{M_{n-2}+1}, q^{M_{n-3}+1}, \ldots, q^{M_{r+1}+1};q)_{\infty}}{(q^{-\lambda_{n-2} + M_{n-2}+1}, q^{-\lambda_{n-3} + M_{n-3}+1}, \ldots, q^{-\lambda_{r+1} + M_{r+1}+1};q)_{\infty}} \nonumber \\
&\cdot \frac{(q^{-L_{n-2} + \lambda_r + \sum_{j=0}^{n-2}M_j}\beta x ;q)_{\infty}}{(q^{-L_{n-2} + \sum_{j=0}^{n-2}M_j}\beta x  ;q)_{\infty}} \nonumber \\
&\cdot \frac{(q^{M_{r} + 1}, \ldots, q^{M_1 +1}, q^{M_0 + 1} ;q)_{\infty}}{(q^{-\lambda_{r-1} + M_{r}+ 1}, \ldots , q^{-\lambda_0 + M_1 + 1}, q^{M_0 + 1}\beta/\alpha ;q)_{\infty}} . \nonumber
\end{align}
Note that the summand in equation \eqref{eq:gn1-13} is well defined, if $\lambda_{j} \notin \mathbb{Z}$ for $j = 0, \ldots, r-1, r+1, \ldots, n-2$.
By the $q$-binomial theorem, we have
\begin{equation*}
\frac{(q^{-L_{n-2} + \lambda_r + \sum_{j=0}^{n-2}M_j }\beta x ;q)_{\infty}}{(q^{-L_{n-2} + \sum_{j=0}^{n-2}M_j }\beta x  ;q)_{\infty}} = \sum_{N = 0}^{\infty} \frac{(q^{\lambda_{r}} ;q)_{N}}{(q ;q)_{N}} (q^{-L_{n-2} + \sum_{j=0}^{n-2}M_j } \beta x )^N, 
\end{equation*}
if $|q^{-L_{n-2} }\beta x | < 1$.
Hence, the sixth and subsequent lines of equation \eqref{eq:gn1-13} are given by
\begin{align}
&  \sum_{M_{n-2} = 0}^{\infty} \sum_{M_{n-3} = 0}^{\infty} \cdots \sum_{M_{r+1} = -\infty}^{\infty} \sum_{M_{r} = 0}^{\infty} \sum_{M_{r-1} = -\infty}^{\infty}  \cdots \sum_{M_{1} = 0}^{\infty} \sum_{M_{0} = 0}^{\infty} \sum_{N = 0}^{\infty} \label{eq:gn1-14} \\
&\cdot (q^{\mu_{r+1} + \cdots + \mu_{n-2} - (\lambda_{r+1} + \cdots + \lambda_{n-3}) + N})^{M_{n-2}} (q^{\mu_{r+1} + \cdots + \mu_{n-3} - (\lambda_{r+1} + \cdots + \lambda_{n-4}) + N})^{M_{n-3}} \nonumber \\
&\cdots (q^{\mu_{r+1} + \mu_{r+2} - \lambda_{r+1} + N})^{M_{r+2}} (q^{\mu_{r+1} + N})^{M_{r+1}} (q^{-\mu_{r} + \lambda_{r-1} + \lambda_{r} + N})^{M_{r}} \nonumber \\
&\cdots (q^{-(\mu_1 + \cdots + \mu_{r}) + \lambda_0 + \cdots + \lambda_{r} + N})^{M_1} (q^{-(\mu_0 + \cdots + \mu_{r}) + \lambda_0 + \cdots + \lambda_{r} + N}\alpha/\beta)^{M_0} \nonumber \\
&\cdot \frac{(q^{M_{n-2}+1}, q^{M_{n-3}+1}, \ldots, q^{M_{r+1}+1} ;q)_{\infty}}{(q^{-\lambda_{n-2} + M_{n-2}+1}, q^{-\lambda_{n-3} + M_{n-3}+1}, \ldots, q^{-\lambda_{r+1} + M_{r+1} + 1} ;q)_{\infty}} \nonumber \\
&\cdot \frac{(q^{M_{r} + 1}, \ldots, q^{M_1 +1}, q^{M_0 + 1} ;q)_{\infty}}{(q^{-\lambda_{r-1} + M_{r}+ 1}, \ldots , q^{-\lambda_0 + M_1 + 1}, q^{M_0 + 1}\beta/\alpha ;q)_{\infty}} \nonumber \\
&\cdot \frac{(q^{\lambda_{r}} ;q)_{N}}{(q ;q)_{N}} (q^{-L_{n-2} } \beta x )^N \nonumber \\
&= \frac{(\overbrace{q, q, \ldots , q}^{n-1}   ; q)_{\infty}}{(q\beta/\alpha, q^{-\lambda_0 + 1} , q^{-\lambda_1 + 1} , \ldots , q^{-\lambda_{r-1} + 1}, q^{-\lambda_{r+1} + 1}, \ldots, q^{-\lambda_{n-2} + 1}; q)_{\infty}} \nonumber \\
&\cdot \sum_{N = 0}^{\infty} \frac{(q^{\lambda_{r}} ;q)_{N}}{(q ;q)_{N}} (q^{-L_{n-2} }\beta x )^N \nonumber \\
&\quad \cdot \left( \sum_{M_{0} = 0}^{\infty} \frac{(q\beta/\alpha ;q)_{M_{0}}}{(q ;q)_{M_{0}}} (q^{-(\mu_0 + \cdots + \mu_{r}) + \lambda_0 + \cdots + \lambda_{r} + N}\alpha/\beta)^{M_{0}} \right) \nonumber \\
&\quad \cdot \left( \sum_{M_{1} = 0}^{\infty} \frac{(q^{-\lambda_{0}+1} ;q)_{M_{1}}}{(q ;q)_{M_{1}}} (q^{-(\mu_1 + \cdots + \mu_{r}) + \lambda_0 + \cdots + \lambda_{r} + N})^{M_1} \right) \nonumber \\
&\quad \cdots \left( \sum_{M_{r} = 0}^{\infty} \frac{(q^{-\lambda_{r-1}+1} ;q)_{M_{r}}}{(q ;q)_{M_{r}}} (q^{-\mu_{r} + \lambda_{r-1} + \lambda_{r} + N})^{M_{r}} \right) \nonumber \\
&\quad \cdot \left( \sum_{M_{r+1} = 0}^{\infty} \frac{(q^{-\lambda_{r+1}+1} ;q)_{M_{r+1}}}{(q ;q)_{M_{r+1}}} (q^{\mu_{r+1} + N})^{M_{r+1}} \right) \nonumber \\ 
&\quad \cdot \left( \sum_{M_{r+2} = 0}^{\infty} \frac{(q^{-\lambda_{r+2}+1} ;q)_{M_{r+2}}}{(q ;q)_{M_{r+2}}} (q^{\mu_{r+1} + \mu_{r+2} - \lambda_{r+1} + N})^{M_{r+2}} \right) \nonumber \\
&\quad \cdots \left( \sum_{M_{n-3} = 0}^{\infty} \frac{(q^{-\lambda_{n-3}+1} ;q)_{M_{n-3}}}{(q ;q)_{M_{n-3}}} (q^{\mu_{r+1} + \cdots + \mu_{n-3} - (\lambda_{r+1} + \cdots + \lambda_{n-4}) + N})^{M_{n-3}} \right) \nonumber \\
&\quad \cdot \left( \sum_{M_{n-2} = 0}^{\infty} \frac{(q^{-\lambda_{n-2}+1} ;q)_{M_{n-2}}}{(q ;q)_{M_{n-2}}} (q^{\mu_{r+1} + \cdots + \mu_{n-2} - (\lambda_{r+1} + \cdots + \lambda_{n-3}) + N})^{M_{n-2}} \right) . \nonumber 
\end{align}
The conditions in Corollary \ref{prop:convn} concerning the eigenvalues of the matrices $I_{j+1} - B_0^{[j+1]}$ and $I_{j+1} - B_0^{[j+1]} - B_1^{[j+1]}$ ($j = 0, 1, \ldots, n-2$) ensure that the $q$-binomial theorem is applicable.
Hence, equation \eqref{eq:gn1-14} can be rewritten as
\begin{align}
& \frac{(\overbrace{q, q, \ldots , q}^{n-1}   ; q)_{\infty}}{(q\beta/\alpha, q^{-\lambda_0 + 1} , q^{-\lambda_1 + 1 } , \ldots , q^{-\lambda_{r-1} + 1}, q^{-\lambda_{r+1} + 1}, \ldots, q^{-\lambda_{n-2} + 1} ; q)_{\infty}} \label{eq:gn1-15} \\
&\cdot \sum_{N = 0}^{\infty}  
\frac{(q^{-(\mu_0 + \cdots + \mu_{r}) + \lambda_0 + \cdots + \lambda_{r} + N + 1}; q)_{\infty}}{(q^{-(\mu_0 + \cdots + \mu_{r}) + \lambda_0 + \cdots + \lambda_{r} + N}\alpha/\beta; q)_{\infty}}
\frac{(q^{-(\mu_1 + \cdots + \mu_{r}) + \lambda_1 + \cdots + \lambda_{r} + N + 1}; q)_{\infty}}{(q^{-(\mu_1 + \cdots + \mu_{r}) + \lambda_0 + \cdots + \lambda_{r} + N}; q)_{\infty}} \nonumber \\
&\cdots \frac{(q^{-\mu_{r} + \lambda_{r} + N + 1}; q)_{\infty}}{(q^{-\mu_{r}+ \lambda_{r-1} + \lambda_{r} + N}; q)_{\infty}}
\frac{(q^{\mu_{r+1} - \lambda_{r+1} + N + 1}; q)_{\infty}}{(q^{\mu_{r+1} + N}; q)_{\infty}}  \frac{(q^{\mu_{r+1} + \mu_{r+2}- (\lambda_{r+1} + \lambda_{r+2}) + N + 1}; q)_{\infty}}{(q^{\mu_{r+1} + \mu_{r+2}- \lambda_{r+1}+ N}; q)_{\infty}} \nonumber \\
&\cdots 
\frac{(q^{\mu_{r+1} + \cdots + \mu_{n-3} - (\lambda_{r+1} + \cdots + \lambda_{n-3}) + N + 1} ;q)_{\infty}}{(q^{\mu_{r+1} + \cdots + \mu_{n-3} - (\lambda_{r+1} + \cdots + \lambda_{n-4}) + N} ;q)_{\infty}}
\frac{(q^{\mu_{r+1} + \cdots + \mu_{n-2} - (\lambda_{r+1} + \cdots + \lambda_{n-2}) + N + 1} ;q)_{\infty}}{(q^{\mu_{r+1} + \cdots + \mu_{n-2} - (\lambda_{r+1} + \cdots + \lambda_{n-3}) + N} ;q)_{\infty}} \nonumber \\
&\cdot \frac{(q^{\lambda_{r}} ;q)_{N}}{(q ;q)_{N}} (q^{-L_{n-2} }\beta x )^N \nonumber \\
&= \frac{(\overbrace{q, q, \ldots , q}^{n-1}   ; q)_{\infty}}{(q\beta/\alpha, q^{-\lambda_0 + 1} , q^{-\lambda_1 + 1 } , \ldots , q^{-\lambda_{r-1} + 1}, q^{-\lambda_{r+1} + 1}, \ldots, q^{-\lambda_{n-2} + 1} ; q)_{\infty}}  \nonumber \\
&\cdot \frac{(\bar{b}_1, \ldots, \bar{b}_{n-1} ; q)_{\infty}}{(\bar{a}_1, \ldots, \bar{a}_{n-1}  ; q)_{\infty}} \;
 \sum_{N = 0}^{\infty} \frac{ (q^{\lambda_{r}}, \bar{a}_1, \ldots,  \bar{a}_{n-1}  ; q)_{N}}
{(q, \bar{b}_1, \ldots, \bar{b}_{n-1} ;q)_{N}}
(q^{ -L_{n-2} } \beta x )^N , \nonumber 
\end{align}
where
\begin{align*}
& \bar{a}_1 = q^{-(\mu_0 + \cdots + \mu_{r}) + \lambda_0 + \cdots + \lambda_{r}}\alpha/\beta, \  
\bar{a}_2 = q^{-(\mu_1 + \cdots + \mu_{r}) + \lambda_0 + \cdots + \lambda_{r}},  \ldots , \bar{a}_{r+1} = q^{-\mu_{r} + \lambda_{r-1} + \lambda_{r}},  \\
& \bar{a}_{r+2} = q^{\mu_{r+1}}, \  \bar{a}_{r+3} = q^{\mu_{r+1} + \mu_{r+2} - \lambda_{r+1}}, \ldots, \bar{a}_{n-2} = q^{\mu_{r+1} + \cdots + \mu_{n-3} - (\lambda_{r+1} + \cdots + \lambda_{n-4})}, \\ 
& \bar{a}_{n-1} = q^{\mu_{r+1} + \cdots + \mu_{n-2} - (\lambda_{r+1} + \cdots + \lambda_{n-3})} ,
\end{align*}
\begin{align*}
& \bar{b}_1 = q^{-(\mu_0 + \cdots + \mu _{r}) + \lambda_0 + \cdots + \lambda_{r} + 1} , \ 
\bar{b}_2 =  q^{-(\mu_1 + \cdots + \mu_{r}) + \lambda_1 + \cdots + \lambda_{r} + 1}, \ldots , \bar{b}_{r+1} = q^{-\mu_{r} + \lambda_{r} + 1}, \\
& \bar{b}_{r+2} = q^{\mu_{r+1} - \lambda_{r+1} + 1}, \  
\bar{b}_{r+3} = q^{\mu_{r+1} + \mu_{r+2} - (\lambda_{r+1} + \lambda_{r+2}) + 1}, \ldots, \\
& \bar{b}_{n-2} = q^{\mu_{r+1} + \cdots + \mu_{n-3} - (\lambda_{r+1} + \cdots + \lambda_{n-3}) + 1 }, \ 
\bar{b}_{n-1} = q^{\mu_{r+1} + \cdots + \mu_{n-2} - (\lambda_{r+1} + \cdots + \lambda_{n-2}) + 1 } .
\end{align*}
By combining with equation \eqref{eq:gn1-13}, we obtain equation \eqref{eq:g1-c2}.
The cases $r = 0$ and $r = n-2$ follow by analogous calculations. 
\end{proof}
The proofs of the remaining cases are omitted here, since they can be obtained by similar calculations involving suitable changes of summation indices and applications of the $q$-binomial theorem.
Outlines of proofs of equations \eqref{eq:g1-c1}, \eqref{eq:g1-c3}, and \eqref{eq:g1-c4} are provided in Appendix \ref{sec:prfg1}.
Explicit examples of $q$-hypergeometric series solutions to equations of orders $3$-$6$ are presented in Appendix \ref{sec:exqhgs}.

Equations \eqref{eq:g1-c1} and \eqref{eq:g1-c2} give $q$-hypergeometric series solutions to equation \eqref{eq:gn1single-thm} at $x = 0$, while equations \eqref{eq:g1-c3} and \eqref{eq:g1-c4} give those at $x = \infty$.
Thus, we obtain $n$ solutions at each of $x = 0$ and $x = \infty$.
By applying the analytic continuation formula \cite[(4.5.2)]{GR}
\begin{align*}
{}_{r+1}\phi_{r} \left( {a_1, a_2 , \ldots, a_{r+1} \atop{b_1, \ldots, b_r}} ; q, z \right) 
&= \frac{(a_2, \ldots, a_{r+1}, b_1/a_1, \ldots, b_r/a_1, a_1 z, q/(a_1 z) ;q)_{\infty}}{(b_1, \ldots, b_{r}, a_2/a_1, \ldots, a_{r+1}/a_1, z, q/z ;q)_{\infty}} \\
&\quad \times {}_{r+1}\phi_{r} \left( {a_1, a_1 q/b_1 , \ldots, a_1 q/b_r \atop{a_1 q/a_2, \ldots, a_1 q/a_{r+1}}} ; q, \frac{q b_1 \cdots b_r}{z a_1 \cdots a_{r+1}} \right) \\
&\quad + idem (a_1; a_2, \ldots, a_{r+1}) 
\end{align*}
to the $q$-hypergeometric series solutions obtained above, the solutions at $x = 0$ can be expressed as linear combinations of the solutions at $x = \infty$.
For the connection relation, we introduce the following notation for the solutions in Proposition \ref{prop:qhgssol}.
Let $g^{(0)}_{1}(x)$ denote the solution \eqref{eq:g1-c1} in \ref{case:1}, and let $g^{(0)}_{2R}(x)$, $R = 0, 1, \ldots, n-2$, denote the solutions \eqref{eq:g1-c2} in \ref{case:2}, where the index $r$ used in \ref{case:2} is relabeled as $R$.
Similarly, let $g^{(\infty)}_{3}(x)$ denote the solution \eqref{eq:g1-c3} in \ref{case:3}, and let $g^{(\infty)}_{4r}(x)$, $r = 0, 1, \ldots, n-2$, denote the solutions \eqref{eq:g1-c4} in \ref{case:4}.
To simplify the expression for the connection coefficients, set 
\[
\widehat{a}_k=q^{\widetilde{M}_{k-1}-L_{k-2}},\qquad  \widehat{a}_n = \alpha/\beta, \qquad \widehat{b}_k=q^{\widetilde{M}_{k-1}-L_{k-1}+1} \quad (k=1,\ldots,n-1), 
\]
and
\begin{align*}
& \widetilde{a}_1 = q^{-\widetilde{M}_R + L_R} \alpha/\beta, \\
& \widetilde{a}_k =
\begin{cases}
q^{-\sum_{j=k-1}^{R}\mu_j + \sum_{j=k-2}^{R}\lambda_j}, \quad 2 \le k \le R+1, \\
q^{\sum_{j=R+1}^{k-1}\mu_j - \sum_{j=R+1}^{k-2}\lambda_j}, \quad R+2\le k\le n-1,
\end{cases} \\
& \widetilde{a}_n = q^{\lambda_R}, \\
& \widetilde{b}_k =
\begin{cases}
q^{\sum_{j=k-1}^{R}(-\mu_j + \lambda_j) + 1}, \quad 1 \le k \le R+1, \\
q^{\sum_{j=R+1}^{k-1}(\mu_j - \lambda_j) + 1}, \quad R+2 \le k \le n-1,
\end{cases} \\
& z = q^{-L_{n-2}}\beta x .
\end{align*}
The following proposition gives the connection relation among the solutions.
\begin{prop}
Assume that the parameters and $x$ are such that all the expressions below are well-defined.
The solutions $g^{(0)}_1(x)$, $g^{(0)}_{2R}(x)$ $(R = 0, 1, \ldots, n-2)$, $g^{(\infty)}_3(x)$, and $g^{(\infty)}_{4r}(x)$ $(r = 0, 1, \ldots, n-2)$ satisfy the following connection relation:
\[
\begin{pmatrix}
  g^{(0)}_1(x) \\
  g^{(0)}_{20}(x) \\
  g^{(0)}_{21}(x) \\
  \vdots \\
  g^{(0)}_{2,n-2}(x)
\end{pmatrix}
= \begin{pmatrix}
      C^*_3(x)     & C^*_{40}(x)     & C^*_{41}(x)     & \cdots & C^*_{4,n-2}(x) \\
      C^0_3(x)     & C^0_{40}(x)     & C^0_{41}(x)     & \cdots & C^0_{4,n-2}(x) \\
      C^1_3(x)     & C^1_{40}(x)     & C^1_{41}(x)     & \cdots & C^1_{4,n-2}(x) \\ 
      \vdots    & \vdots       & \vdots       & \ddots & \vdots \\
      C^{n-2}_3(x) & C^{n-2}_{40}(x) & C^{n-2}_{41}(x) & \cdots & C^{n-2}_{4,n-2}(x)      
\end{pmatrix} 
\begin{pmatrix}
      g^{(\infty)}_3(x) \\
      g^{(\infty)}_{40}(x) \\
      g^{(\infty)}_{41}(x) \\
      \vdots \\
      g^{(\infty)}_{4,n-2}(x)
\end{pmatrix} ,
\]
where
\begin{align*}
C^*_3(x)
&= q^{-\mu_0 L_{n-2} - \sum_{j=1}^{n-2}(\mu_j - \lambda_{j-1})T_j + \sum_{j=0}^{n-2}(n-1-j)(\mu_{j} - \lambda_{j})} \left( \frac{\beta}{\alpha} \right)^{n-1} x^{-\widetilde{\mu} + \mu_0} \\
&\quad \cdot \frac{(\widehat{b}_1/\widehat{a}_n, \ldots, \widehat{b}_{n-1}/\widehat{a}_n, \widehat{a}_n q/\widehat{b}_1, \ldots, \widehat{a}_n q/\widehat{b}_{n-1}, \widehat{a}_n z, q/(\widehat{a}_n z) ; q)_{\infty}}{(\widehat{a}_1/\widehat{a}_n, \ldots, \widehat{a}_{n-1}/\widehat{a}_n, \widehat{a}_n q/\widehat{a}_1, \ldots, \widehat{a}_n q/\widehat{a}_{n-1}, z, q/z ; q)_{\infty}} ,
\end{align*}
\begin{align*}
C^*_{4r}(x) 
&= q^{-\mu_0 L_{n-2} - \sum_{j=1}^{n-2}(\mu_j - \lambda_{j-1})T_j - F_r} \alpha^{\widetilde{M}_r - L_{r-1}} x^{\widetilde{M}_r - L_{r-1}} \\
&\quad \cdot 
\frac{\prod_{j=1}^{n-1} (\widehat{b}_j/\widehat{a}_{r+1}; q)_{\infty}}{\prod_{\substack{1 \le j\le n \\ j\neq r+1}}(\widehat{a}_j/\widehat{a}_{r+1}; q)_{\infty}} 
\frac{(q\beta/\alpha, \widehat{a}_n; q)_{\infty}}{(q^{-\lambda_r + 1}; q)_{\infty}} 
\frac{\prod_{\substack{1 \le j\le n-1 \\ j\neq r+1}}(\widehat{a}_{r+1}q/\widehat{b}_j; q)_{\infty}}
{\prod_{\substack{1 \le j\le n \\ j\neq r+1}}(\widehat{a}_{r+1}q/\widehat{a}_j; q)_{\infty}} \\
&\quad \cdot \frac{(\widehat{a}_{r+1} z, q/(\widehat{a}_{r+1} z); q)_{\infty}}{(z, q/z; q)_{\infty}},
\end{align*}
\begin{align*}
C^R_{3}(x) 
&= q^{E_R + \sum_{j=0}^{n-2}(n-1-j)(\mu_j - \lambda_j)} \alpha^{1-n} \beta^{-\widetilde{\mu} - \sum_{j=1}^R \mu_{j} + L_R + n-1} x^{-\widetilde{\mu} - \sum_{j=1}^{R} \mu_j + L_{R}}  \\
&\quad \cdot \frac{(\widetilde{b}_1/\widetilde{a}_1, \ldots, \widetilde{b}_{n-1}/\widetilde{a}_1, q^{-\lambda_R +1}, \widetilde{a}_n, \widetilde{a}_1 q/\widetilde{b}_2, \ldots, \widetilde{a}_1 q/\widetilde{b}_{n-1}, \widetilde{a}_1 z, q/(\widetilde{a}_1 z) ; q)_{\infty}}{(\widetilde{a}_2/\widetilde{a}_1, \ldots, \widetilde{a}_{n}/\widetilde{a}_1, q\beta/\alpha, \widetilde{a}_1 q/\widetilde{a}_2, \ldots, \widetilde{a}_1 q/\widetilde{a}_{n}, z, q/z ; q)_{\infty}} ,
\end{align*}
and
\begin{align*}
C^R_{4r}(x)
&= q^{E_R - F_r} \alpha^{\widetilde{M}_r - L_{r-1}} \beta^{-\widetilde{\mu} - \sum_{j=1}^R \mu_{j} + L_R} x^{-\sum_{j= r+1}^R \mu_{j} + \sum_{j=r}^R \lambda_{j}} \\
&\quad \cdot 
\frac{\prod_{j=1}^{n-1} (\widetilde{b}_j/\widetilde{a}_{r+2}; q)_{\infty}}{\prod_{\substack{1 \le j\le n \\ j\neq r+2}}(\widetilde{a}_j/\widetilde{a}_{r+2}; q)_{\infty}} 
\frac{(q^{-\lambda_R +1}, \widetilde{a}_{n}; q)_{\infty}}{(q^{-\lambda_r + 1}; q)_{\infty}}
\frac{\prod_{\substack{1 \le j\le n-1 \\ j\neq r+2}}(\widetilde{a}_{r+2}q/\widetilde{b}_j; q)_{\infty}}
{\prod_{\substack{1 \le j\le n \\ j\neq r+2}}(\widetilde{a}_{r+2}q/\widetilde{a}_j; q)_{\infty}} \\
&\quad \cdot \frac{(\widetilde{a}_{r+2} z, q/(\widetilde{a}_{r+2} z); q)_{\infty}}{(z, q/z; q)_{\infty}}, 
\qquad\qquad 0 \le r < R,
\end{align*}
\begin{align*}
C^R_{4R}(x)
&= q^{E_R - F_R} \alpha^{\widetilde{M}_R - L_{R-1}} \beta^{-\widetilde{\mu} - \sum_{j=1}^R \mu_{j} + L_R} x^{\lambda_{R}} \\
&\quad \cdot \frac{(\widetilde{b}_1/\widetilde{a}_{n}, \ldots, \widetilde{b}_{n-1}/\widetilde{a}_{n}, \widetilde{a}_{n}q/\widetilde{b}_1, \ldots, \widetilde{a}_{n}q/\widetilde{b}_{n-1}, \widetilde{a}_n z, q/(\widetilde{a}_{n} z); q)_{\infty}}{(\widetilde{a}_1/\widetilde{a}_{n}, \ldots, \widetilde{a}_{n-1}/\widetilde{a}_{n}, \widetilde{a}_{n}q/\widetilde{a}_1, \ldots, \widetilde{a}_{n}q/\widetilde{a}_{n-1}, z, q/z ; q)_{\infty}} ,
\end{align*}
\begin{align*}
C^R_{4r}(x)
&= q^{E_R - F_r} \alpha^{\widetilde{M}_r - L_{r-1}} \beta^{-\widetilde{\mu} - \sum_{j=1}^R \mu_{j} + L_R} x^{\sum_{j= R+1}^r \mu_{j} - \sum_{j=R+1}^{r-1} \lambda_{j}} \\
&\quad \cdot 
\frac{\prod_{j=1}^{n-1} (\widetilde{b}_j/\widetilde{a}_{r+1}; q)_{\infty}}{\prod_{\substack{1 \le j\le n \\ j\neq r+1}}(\widetilde{a}_j/\widetilde{a}_{r+1}; q)_{\infty}} 
\frac{(q^{-\lambda_R +1}, \widetilde{a}_{n}; q)_{\infty}}{(q^{-\lambda_r + 1}; q)_{\infty}}
\frac{\prod_{\substack{1 \le j\le n-1 \\ j\neq r+1}}(\widetilde{a}_{r+1}q/\widetilde{b}_j; q)_{\infty}}
{\prod_{\substack{1 \le j\le n \\ j\neq r+1}}(\widetilde{a}_{r+1}q/\widetilde{a}_j; q)_{\infty}} \\
&\quad \cdot \frac{(\widetilde{a}_{r+1} z, q/(\widetilde{a}_{r+1} z); q)_{\infty}}{(z, q/z; q)_{\infty}},
\qquad\qquad R < r \le n-2 .
\end{align*}
\end{prop}
We omit the details of the calculation.

Ohyama \cite{Ohy} obtained a connection formula for a higher-order $q$-difference equation satisfied by $q$-hypergeometric series 
${}_r \phi_{r-1}$ with all upper parameters equal to $0$, using $q$-Borel and $q$-Laplace transformations.

\section*{Acknowledgements}
The author is grateful to Professor Kouichi Takemura for his valuable comments and helpful suggestions. 

\appendix
\section{Generalized $q$-hypergeometric equations of orders $4$-$6$} \label{sec:geneq4-6}
We present explicit generalized $q$-hypergeometric equations of orders $4$-$6$ obtained by applying Theorem \ref{thm:g1n-single-sec4}.
\begin{align} 
&(\alpha x-1) g_1^{[4]}(q^4 x)  \label{eq:g1-o4}\\
&\quad - q^{-(\lambda_0 + \lambda_1 + \lambda_2)}
[ \{ ( q^{\mu_1 + \mu_2} + q^{\mu_2 + \lambda_0} + q^{\lambda_0 + \lambda_1})\alpha + q^{\mu_0 + \mu_1 + \mu_2}\beta \}x \nonumber \\
&\quad\qquad - (q^{\mu_0 + \mu_1 + \mu_2} + q^{\mu_1 + \mu_2 + \lambda_0} + q^{\mu_2 + \lambda_0 + \lambda_1} + q^{\lambda_0 + \lambda_1 + \lambda_2}) ] g_1^{[4]}(q^3 x)  \nonumber \\
&\quad + q^{\mu_2 - 2(\lambda_0 + \lambda_1 + \lambda_2)}
[ \{(q^{\mu_1 + \mu_2 + \lambda_0} + q^{\mu_1 + \lambda_0 + \lambda_1} + q^{2\lambda_0 + \lambda_1}) \alpha + (q^{\mu_0 + 2\mu_1 + \mu_2} \nonumber \\
&\quad\qquad + q^{\mu_0 + \mu_1 + \mu_2 +\lambda_0} + q^{\mu_0 + \mu_1 +\lambda_0 + \lambda_1})\beta \} x - q^{\lambda_0}(q^{\mu_0 + 2\mu_1 + \mu_2} + q^{\mu_0 + \mu_1 + \mu_2 + \lambda_1} \nonumber \\
&\quad\qquad + q^{\mu_1 + \mu_2 + \lambda_0 + \lambda_1} + q^{\mu_0 + \mu_1 + \lambda_1 + \lambda_2} + q^{\mu_1 + \lambda_0 + \lambda_1 + \lambda_2 } + q^{\lambda_0 + 2\lambda_1 + \lambda_2}) ] g_1^{[4]}(q^2x) \nonumber \\
&\quad - q^{\mu_1 + 2\mu_2 - (2\lambda_0 + 3\lambda_1 + 3\lambda_2)} 
[ \{ q^{\lambda_0 + \lambda_1}\alpha + (q^{\mu_0 + \mu_1 + \mu_2} + q^{\mu_0 + \mu_1 + \lambda_1} + q^{\mu_0 + \lambda_0 + \lambda_1}) \beta \} x \nonumber \\
&\quad\qquad - q^{\lambda_0 + \lambda_1}(q^{\mu_0 + \mu_1 + \mu_2} + q^{\mu_0 + \mu_1 + \lambda_2} + q^{\mu_0 + \lambda_1 + \lambda_2} + q^{\lambda_0 + \lambda_1 + \lambda_2})] g_1^{[4]}(qx) \nonumber \\
&\quad + q^{\mu_0 + 2\mu_1 + 3\mu_2 - (\lambda_0 + 2\lambda_1 + 3\lambda_2)}(q^{-(\lambda_0 + \lambda_1 + \lambda_2)}\beta x - 1) g_1^{[4]}(x)
=0. \nonumber
\end{align}
\begin{align}
&(\alpha x-1) g_1^{[5]}(q^5 x) \label{eq:g1-o5} \\
&\qquad + R_4(x) g_1^{[5]}(q^4 x) + R_3(x) g_1^{[5]}(q^3 x) + R_2(x) g_1^{[5]}(q^2 x) + R_1(x) g_1^{[5]}(qx)  \nonumber \\
&\qquad- q^{\mu_0 + 2\mu_1 + 3\mu_2 + 4\mu_3 - (\lambda_0 + 2\lambda_1 + 3\lambda_2 + 4\lambda_3)} (q^{-(\lambda_0 + \lambda_1 + \lambda_2 + \lambda_3)}\beta x - 1) g_1^{[5]}(x)
=0,  \nonumber
\end{align}
where
\begin{align*} 
R_4(x) 
&= - q^{-(\lambda_0 + \lambda_1 + \lambda_2 + \lambda_3)}
[ \{ ( q^{\mu_1 + \mu_2 + \mu_3} + q^{\mu_2 + \mu_3 + \lambda_0} + q^{\mu_3 + \lambda_0 + \lambda_1} + q^{\lambda_0 + \lambda_1  + \lambda_2})\alpha \\
&\quad + q^{\mu_0 + \mu_1 + \mu_2 + \mu_3}\beta \}x \\
&\quad - (q^{\mu_0 + \mu_1 + \mu_2 + \mu_3} + q^{\mu_1 + \mu_2 + \mu_3 + \lambda_0} + q^{\mu_2 + \mu_3 + \lambda_0 + \lambda_1} + q^{\mu_3 + \lambda_0 + \lambda_1 + \lambda_2} \\
&\quad + q^{\lambda_0 + \lambda_1 + \lambda_2 + \lambda_3}) ]  ,
\end{align*}
\begin{align*}
& R_3(x) 
= q^{\mu_3 - 2(\lambda_0 + \lambda_1 + \lambda_2 + \lambda_3)} \\
&\quad \cdot [ \{(q^{\mu_1 + 2\mu_2 + \mu_3 + \lambda_0} + q^{\mu_1 + \mu_2 + \mu_3 + \lambda_0 + \lambda_1} + q^{\mu_1 + \mu_2 + \lambda_0 + \lambda_1 +\lambda_2} + q^{\mu_2 + \mu_3 + 2\lambda_0 + \lambda_1} \\
&\quad + q^{\mu_2 + 2\lambda_0 + \lambda_1 + \lambda_2} + q^{2\lambda_0 + 2\lambda_1 + \lambda_2}) \alpha   \\
&\quad + (q^{\mu_0 + 2\mu_1 + 2\mu_2 + \mu_3}+ q^{\mu_0 + \mu_1 + 2\mu_2 + \mu_3 + \lambda_0} + q^{\mu_0 + \mu_1 + \mu_2 + \mu_3 +\lambda_0 + \lambda_1}\\
&\quad +  q^{\mu_0 + \mu_1 + \mu_2 +\lambda_0 + \lambda_1 + \lambda_2})\beta \} x \\
&\quad - q^{\lambda_0}(q^{\mu_0 + 2\mu_1 + 2\mu_2 + \mu_3} + q^{\mu_0 + \mu_1 + 2\mu_2 + \mu_3 + \lambda_1}  + q^{\mu_0 + \mu_1 + \mu_2 + \mu_3 + \lambda_1 + \lambda_2} \\
&\quad + q^{ \mu_1 + 2\mu_2 + \mu_3 + \lambda_0 + \lambda_1} + q^{\mu_0 + \mu_1 + \mu_2 + \lambda_1 + \lambda_2 + \lambda_3}  + q^{\mu_1 + \mu_2 + \mu_3  + \lambda_0 + \lambda_1 + \lambda_2} \\
&\quad + q^{\mu_1 + \mu_2 + \lambda_0 + \lambda_1 + \lambda_2 + \lambda_3} + q^{\mu_2 + \mu_3 + \lambda_0 + 2\lambda_1 + \lambda_2 } + q^{\mu_2 + \lambda_0 + 2\lambda_1 + \lambda_2 + \lambda_3} + q^{\lambda_0 + 2\lambda_1 + 2\lambda_2 + \lambda_3} ) ]  ,
\end{align*}
\begin{align*}
& R_2(x)
= - q^{\mu_2 + 2\mu_3 - (2\lambda_0 + 3\lambda_1 + 3\lambda_2 + 3\lambda_3)} \\
&\quad \cdot [ \{ (q^{\mu_1 + \mu_2 + \mu_3 + \lambda_0 + \lambda_1} + q^{\mu_1 + \mu_2 + \lambda_0 + \lambda_1 + \lambda_2} + q^{\mu_1 + \lambda_0 + 2\lambda_1 + \lambda_2 } \\
&\quad +  q^{2\lambda_0 + 2\lambda_1 + \lambda_2}) \alpha \\
&\quad + (q^{\mu_0 + 2\mu_1 + 2\mu_2 + \mu_3}+ q^{\mu_0 + 2\mu_1 + \mu_2 + \mu_3 + \lambda_1} + q^{\mu_0 + \mu_1 + \mu_2 + \mu_3 +\lambda_0 + \lambda_1} +  q^{\mu_0 + 2\mu_1 + \mu_2 +\lambda_1 + \lambda_2} \\
&\quad + q^{\mu_0 + \mu_1 + \mu_2 +\lambda_0 + \lambda_1 + \lambda_2} + q^{\mu_0 + \mu_1 + \lambda_0 + 2\lambda_1 + \lambda_2})\beta \} x \\
&\quad - q^{\lambda_0 + \lambda_1}(q^{\mu_0 + 2\mu_1 + 2\mu_2 + \mu_3} + q^{\mu_0 + 2\mu_1 + \mu_2 + \mu_3 + \lambda_2} + q^{\mu_0 + 2\mu_1 + \mu_2 + \lambda_2 + \lambda_3} \\
&\quad  + q^{\mu_0 +  \mu_1 + \mu_2 + \mu_3 + \lambda_1 + \lambda_2} + q^{\mu_0 + \mu_1 + \mu_2 + \lambda_1 + \lambda_2 + \lambda_3} + q^{\mu_1 + \mu_2 + \mu_3 + \lambda_0 + \lambda_1 + \lambda_2} \\
&\quad + q^{\mu_1 + \mu_2 + \lambda_0 + \lambda_1 + \lambda_2 + \lambda_3} + q^{\mu_0 + \mu_1 + \lambda_1 + 2\lambda_2 + \lambda_3 } + q^{\mu_1 + \lambda_0 + \lambda_1 + 2\lambda_2 + \lambda_3 } + q^{\lambda_0 + 2\lambda_1 + 2\lambda_2 + \lambda_3}) ] ,
\end{align*}
\begin{align*}
R_1(x)
&= q^{\mu_1 + 2\mu_2 + 3\mu_3 - (2\lambda_0 + 3\lambda_1 + 4\lambda_2 + 4\lambda_3)} \\
&\quad \cdot [ \{ q^{\lambda_0 + \lambda_1 + \lambda_2}\alpha + (q^{\mu_0 + \mu_1 + \mu_2 + \mu_3} + q^{\mu_0 + \mu_1 + \mu_2 + \lambda_2} + q^{\mu_0 + \mu_1 + \lambda_1 + \lambda_2} \\
&\quad + q^{\mu_0 + \lambda_0 + \lambda_1 + \lambda_2}) \beta \} x \\
&\quad - q^{\lambda_0 + \lambda_1 + \lambda_2}(q^{\mu_0 + \mu_1 + \mu_2 + \mu_3} + q^{\mu_0 + \mu_1 + \mu_2 + \lambda_3} + q^{\mu_0 + \mu_1 + \lambda_2 + \lambda_3} + q^{\mu_0 + \lambda_1 + \lambda_2 + \lambda_3} \\
&\quad + q^{\lambda_0 + \lambda_1 + \lambda_2 + \lambda_3})] .
\end{align*} 
\begin{align}
&(\alpha x-1) g_1^{[6]}(q^6 x) \label{eq:g1-o6} \\
&\quad + S_5(x) g_1^{[6]}(q^5 x) + S_4(x) g_1^{[6]}(q^4 x) + S_3(x) g_1^{[6]}(q^3 x) + S_2(x) g_1^{[6]}(q^2 x) + S_1(x) g_1^{[6]}(qx)  \nonumber \\
&\quad + q^{\mu_0 + 2\mu_1 + 3\mu_2 + 4\mu_3 + 5\mu_4 - (\lambda_0 + 2\lambda_1 + 3\lambda_2 + 4\lambda_3 + 5\lambda_4)} (q^{-(\lambda_0 + \lambda_1 + \lambda_2 + \lambda_3 + \lambda_4)}\beta x - 1) g_1^{[6]}(x)
=0,  \nonumber
\end{align}
where
\begin{align*} 
&S_5(x)
= - q^{-(\lambda_0 + \lambda_1 + \lambda_2 + \lambda_3 + \lambda_4)} \\
&\quad \cdot [ \{ ( q^{\mu_1 + \mu_2 + \mu_3 + \mu_4} + q^{\mu_2 + \mu_3 + \mu_4 + \lambda_0} + q^{\mu_3 + \mu_4 + \lambda_0 + \lambda_1} + q^{\mu_4 + \lambda_0 + \lambda_1 + \lambda_2}  \\
&\quad + q^{\lambda_0 + \lambda_1  + \lambda_2 + \lambda_3})\alpha + q^{\mu_0 + \mu_1 + \mu_2 + \mu_3 + \mu_4}\beta \}x \\
&\quad - (q^{\mu_0 + \mu_1 + \mu_2 + \mu_3 + \mu_4} + q^{\mu_1 + \mu_2 + \mu_3 + \mu_4 + \lambda_0} + q^{\mu_2 + \mu_3 + \mu_4 + \lambda_0 + \lambda_1} + q^{\mu_3 + \mu_4 + \lambda_0 + \lambda_1 + \lambda_2} \\
&\quad + q^{\mu_4 + \lambda_0 + \lambda_1 + \lambda_2 + \lambda_3} + q^{\lambda_0 + \lambda_1 + \lambda_2 + \lambda_3 + \lambda_4}) ] ,
\end{align*}
\begin{align*}
&S_4(x)
= q^{\mu_4 - 2(\lambda_0 + \lambda_1 + \lambda_2 + \lambda_3 + \lambda_4)} \\
&\quad \cdot [ \{(q^{\mu_1 + 2\mu_2 + 2\mu_3 + \mu_4 + \lambda_0} + q^{\mu_1 + \mu_2 + 2\mu_3 + \mu_4 + \lambda_0 + \lambda_1} + q^{ \mu_2 + 2\mu_3 + \mu_4 + 2\lambda_0 + \lambda_1}  \\
&\quad + q^{\mu_1 + \mu_2 + \mu_3 + \mu_4 + \lambda_0 + \lambda_1 +\lambda_2} + q^{\mu_2 + \mu_3 + \mu_4 + 2\lambda_0 + \lambda_1 + \lambda_2} + q^{\mu_3 + \mu_4 + 2\lambda_0 + 2\lambda_1 +\lambda_2}  \\
&\quad + q^{\mu_1 + \mu_2 + \mu_3 + \lambda_0 + \lambda_1 + \lambda_2 + \lambda_3} + q^{\mu_2 + \mu_3 + 2\lambda_0 + \lambda_1 + \lambda_2 + \lambda_3} + q^{\mu_3 + 2\lambda_0 + 2\lambda_1 + \lambda_2 + \lambda_3} \\
&\quad + q^{2\lambda_0 + 2\lambda_1 + 2\lambda_2 + \lambda_3}) \alpha + (q^{\mu_0 + 2\mu_1 + 2\mu_2 + 2\mu_3 + \mu_4}+ q^{\mu_0 + \mu_1 + 2\mu_2 + 2\mu_3+ \mu_4 + \lambda_0} \\
&\quad  + q^{\mu_0 + \mu_1 + \mu_2 + 2\mu_3 + \mu_4 + \lambda_0 + \lambda_1} + q^{\mu_0 + \mu_1 + \mu_2 + \mu_3 + \mu_4 + \lambda_0 + \lambda_1 + \lambda_2} \\
&\quad + q^{\mu_0 + \mu_1 + \mu_2 + \mu_3 + \lambda_0 + \lambda_1 + \lambda_2 + \lambda_3} )\beta \} x \\
&\quad - q^{\lambda_0}(q^{\mu_0 + 2\mu_1 + 2\mu_2 + 2\mu_3 + \mu_4} + q^{\mu_0 + \mu_1 + 2\mu_2 + 2\mu_3 + \mu_4 + \lambda_1}  + q^{\mu_0 + \mu_1 + \mu_2 + 2\mu_3 + \mu_4 + \lambda_1 + \lambda_2} \\
&\quad + q^{\mu_0 + \mu_1 + \mu_2 + \mu_3 + \mu_4 + \lambda_1 + \lambda_2 + \lambda_3} + q^{\mu_0 + \mu_1 + \mu_2 + \mu_3 + \lambda_1 + \lambda_2 + \lambda_3 + \lambda_4} + q^{ \mu_1 + 2\mu_2 + 2\mu_3 + \mu_4 + \lambda_0 + \lambda_1} \\
&\quad + q^{\mu_1 + \mu_2 + 2\mu_3 + \mu_4 + \lambda_0 + \lambda_1 + \lambda_2} + q^{\mu_2 + 2\mu_3 + \mu_4 + \lambda_0 + 2\lambda_1 + \lambda_2} + q^{\mu_1 + \mu_2 + \mu_3  + \mu_4 + \lambda_0 + \lambda_1 + \lambda_2 + \lambda_3} \\
&\quad + q^{\mu_2 + \mu_3  + \mu_4 + \lambda_0 + 2\lambda_1 + \lambda_2 + \lambda_3} + q^{\mu_3 + \mu_4 + \lambda_0 + 2\lambda_1 + 2\lambda_2 + \lambda_3} + q^{\mu_1 + \mu_2 + \mu_3 + \lambda_0 + \lambda_1 + \lambda_2 + \lambda_3 + \lambda_4}\\
&\quad + q^{\mu_2 + \mu_3 + \lambda_0 + 2\lambda_1 + \lambda_2 + \lambda_3 + \lambda_4} + q^{\mu_3 + \lambda_0 + 2\lambda_1 + 2\lambda_2 + \lambda_3 + \lambda_4} + q^{\lambda_0 + 2\lambda_1 + 2\lambda_2 + 2\lambda_3 + \lambda_4}) ] ,
\end{align*}
\begin{align*}
&S_3(x)
= - q^{\mu_3 + 2\mu_4 - (2\lambda_0 + 3\lambda_1 + 3\lambda_2 + 3\lambda_3 + 3\lambda_4)} \\
&\quad \cdot [ \{(q^{\mu_1 + 2\mu_2 + 2\mu_3 + \mu_4 + \lambda_0 + \lambda_1} + q^{\mu_1 + 2\mu_2 + \mu_3 + \mu_4 + \lambda_0 + \lambda_1+ \lambda_2} + q^{\mu_1 + \mu_2 + \mu_3 + \mu_4 + \lambda_0 + 2\lambda_1 + \lambda_2} \\
&\quad + q^{\mu_2 + \mu_3 + \mu_4 + 2\lambda_0 + 2\lambda_1 + \lambda_2} + q^{\mu_1 + 2\mu_2 + \mu_3 + \lambda_0 + \lambda_1 + \lambda_2 + \lambda_3} + q^{\mu_1 + \mu_2 + \mu_3 + \lambda_0 + 2\lambda_1 + \lambda_2 + \lambda_3} \\
&\quad + q^{\mu_2 + \mu_3 + 2\lambda_0 + 2\lambda_1 + \lambda_2 + \lambda_3} + q^{\mu_1 + \mu_2 + \lambda_0 + 2\lambda_1 + 2\lambda_2 + \lambda_3} + q^{\mu_2 + 2\lambda_0 + 2\lambda_1 + 2\lambda_2 + \lambda_3} \\
&\quad + q^{2\lambda_0 + 3\lambda_1 + 2\lambda_2 + \lambda_3}) \alpha + (q^{\mu_0 + 2\mu_1 + 3\mu_2 + 2\mu_3 + \mu_4}+ q^{\mu_0 + 2\mu_1 + 2\mu_2 + 2\mu_3+ \mu_4 + \lambda_1} \\
&\quad + q^{\mu_0 + 2\mu_1 + 2\mu_2 + \mu_3 + \mu_4 + \lambda_1 + \lambda_2} + q^{\mu_0 + 2\mu_1 + 2\mu_2 + \mu_3 + \lambda_1 + \lambda_2 + \lambda_3} + q^{\mu_0 + \mu_1 + 2\mu_2 + 2\mu_3 + \mu_4 + \lambda_0 + \lambda_1} \\
&\quad + q^{\mu_0 + \mu_1 + 2\mu_2 + \mu_3 + \mu_4 + \lambda_0 + \lambda_1 + \lambda_2} 
+ q^{\mu_0 + \mu_1 + \mu_2 + \mu_3 + \mu_4 + \lambda_0 + 2\lambda_1 + \lambda_2} \\
&\quad + q^{\mu_0 + \mu_1 + 2\mu_2 + \mu_3 + \lambda_0 + \lambda_1 + \lambda_2 + \lambda_3} \\
&\quad + q^{\mu_0 + \mu_1 + \mu_2 + \mu_3 + \lambda_0 + 2\lambda_1 + \lambda_2 + \lambda_3} + q^{\mu_0 + \mu_1 + \mu_2 + \lambda_0 + 2\lambda_1 + 2\lambda_2 + \lambda_3})\beta \} x \\
&\quad - q^{\lambda_0 + \lambda_1}(q^{\mu_0 + 2\mu_1 + 3\mu_2 + 2\mu_3 + \mu_4} + q^{\mu_0 + 2\mu_1 + 2\mu_2 + 2\mu_3 + \mu_4 + \lambda_2} \\
&\quad + q^{\mu_0 + 2\mu_1 + 2\mu_2 + \mu_3 + \mu_4 + \lambda_2 + \lambda_3} + q^{\mu_0 + 2\mu_1 + 2\mu_2 + \mu_3 + \lambda_2 + \lambda_3 + \lambda_4} + q^{\mu_0 + \mu_1 + 2\mu_2 + 2\mu_3 + \mu_4 + \lambda_1 + \lambda_2} \\
&\quad + q^{\mu_0 +  \mu_1 + 2\mu_2 + \mu_3 + \mu_4 + \lambda_1 + \lambda_2 + \lambda_3} + q^{\mu_0 + \mu_1 + \mu_2 + \mu_3 + \mu_4 + \lambda_1 + 2\lambda_2 + \lambda_3} \\
&\quad + q^{\mu_0 + \mu_1 + 2\mu_2 + \mu_3 + \lambda_1 + \lambda_2 + \lambda_3 + \lambda_4} + q^{\mu_0 + \mu_1 + \mu_2  + \mu_3 + \lambda_1 + 2\lambda_2 + \lambda_3 + \lambda_4} \\
&\quad + q^{\mu_0 + \mu_1  + \mu_2 + \lambda_1 + 2\lambda_2 + 2\lambda_3 + \lambda_4} + q^{\mu_1 + 2\mu_2 + 2\mu_3 + \mu_4 + \lambda_0 + \lambda_1 + \lambda_2} \\
&\quad + q^{\mu_1 + 2\mu_2 + \mu_3 + \mu_4 + \lambda_0 + \lambda_1 + \lambda_2 + \lambda_3} + q^{\mu_1 + \mu_2 + \mu_3+ \mu_4 + \lambda_0 + \lambda_1 + 2\lambda_2 + \lambda_3} + q^{\mu_2 + \mu_3 + \mu_4 + \lambda_0 + 2\lambda_1 + 2\lambda_2 + \lambda_3 } \\
&\quad + q^{\mu_1 + 2\mu_2 + \mu_3 + \lambda_0 + \lambda_1 + \lambda_2 + \lambda_3 + \lambda_4} + q^{\mu_1 + \mu_2 + \mu_3 + \lambda_0 + \lambda_1 + 2\lambda_2 + \lambda_3 + \lambda_4} + q^{\mu_1 + \mu_2 + \lambda_0 + \lambda_1 + 2\lambda_2 + 2\lambda_3 + \lambda_4}\\
&\quad + q^{\mu_2 + \mu_3 + \lambda_0 + 2\lambda_1 + 2\lambda_2 + \lambda_3 + \lambda_4} + q^{\mu_2 + \lambda_0 + 2\lambda_1 + 2\lambda_2 + 2\lambda_3 + \lambda_4} + q^{\lambda_0 + 2\lambda_1 + 3\lambda_2 + 2\lambda_3 + \lambda_4}) ] ,
\end{align*}
\begin{align*}
&S_2(x)
= q^{\mu_2 + 2\mu_3 + 3\mu_4 - (2\lambda_0 + 3\lambda_1 + 4\lambda_2 + 4\lambda_3 + 4\lambda_4 )} \\
&\quad \cdot [ \{ (q^{\mu_1 + \mu_2 + \mu_3 + \mu_4 + \lambda_0 + \lambda_1 + \lambda_2} + q^{\mu_1 + \mu_2 + \mu_3 + \lambda_0 + \lambda_1 + \lambda_2 + \lambda_3} + q^{\mu_1 + \mu_2 + \lambda_0 + \lambda_1 + 2\lambda_2 + \lambda_3} \\
&\quad + q^{\mu_1 + \lambda_0 + 2\lambda_1 + 2\lambda_2 + \lambda_3}+  q^{2\lambda_0 + 2\lambda_1 + 2\lambda_2 + \lambda_3}) \alpha \\
&\quad + (q^{\mu_0 + 2\mu_1 + 2\mu_2 + 2\mu_3 + \mu_4}+ q^{\mu_0 + 2\mu_1 + 2\mu_2 + \mu_3 + \mu_4 + \lambda_2} + q^{\mu_0 + 2\mu_1 + \mu_2 + \mu_3 + \mu_4 + \lambda_1 + \lambda_2} \\
&\quad + q^{\mu_0 + 2\mu_1 + 2\mu_2 + \mu_3 + \lambda_2 + \lambda_3} + q^{\mu_0 + \mu_1 + \mu_2 + \mu_3 + \mu_4 + \lambda_0 + \lambda_1 + \lambda_2} + q^{\mu_0 + 2\mu_1 + \mu_2 + \mu_3 + \lambda_1 + \lambda_2 + \lambda_3} \\
&\quad + q^{\mu_0 + \mu_1 + \mu_2 + \mu_3 + \lambda_0 + \lambda_1 + \lambda_2 + \lambda_3} + q^{\mu_0 + 2\mu_1 + \mu_2 + \lambda_1 + 2\lambda_2 + \lambda_3} + q^{\mu_0 + \mu_1 + \mu_2 + \lambda_0 + \lambda_1 + 2\lambda_2 + \lambda_3} \\
&\quad + q^{\mu_0 + \mu_1 + \lambda_0 + 2\lambda_1 + 2\lambda_2 + \lambda_3})\beta \} x \\
&\quad - q^{\lambda_0 + \lambda_1 + \lambda_2}(q^{\mu_0 + 2\mu_1 + 2\mu_2 + 2\mu_3 + \mu_4} + q^{\mu_0 + 2\mu_1 + 2\mu_2 + \mu_3 + \mu_4 + \lambda_3} \\
&\quad + q^{\mu_0 + 2\mu_1 + \mu_2 + \mu_3 + \mu_4 + \lambda_2 + \lambda_3} + q^{\mu_0 + 2\mu_1 + 2\mu_2 + \mu_3 + \lambda_3 + \lambda_4} + q^{\mu_0 + \mu_1 + \mu_2 + \mu_3 + \mu_4 + \lambda_1 + \lambda_2 + \lambda_3} \\
&\quad + q^{\mu_0 + 2\mu_1 + \mu_2 + \mu_3 + \lambda_2 + \lambda_3 + \lambda_4} + q^{\mu_1 + \mu_2 + \mu_3 + \mu_4 + \lambda_0 + \lambda_1 + \lambda_2 + \lambda_3} \\
&\quad + q^{\mu_0 + \mu_1 + \mu_2 + \mu_3 + \lambda_1 + \lambda_2 + \lambda_3 + \lambda_4} + q^{\mu_0 + 2\mu_1 + \mu_2 + \lambda_2 + 2\lambda_3 + \lambda_4} + q^{\mu_1 + \mu_2 + \mu_3 + \lambda_0 + \lambda_1 + \lambda_2 + \lambda_3 + \lambda_4} \\
&\quad + q^{\mu_0 + \mu_1 + \mu_2 + \lambda_1 + \lambda_2 + 2\lambda_3 + \lambda_4} + q^{\mu_0 + \mu_1 + \lambda_1 + 2\lambda_2 + 2\lambda_3 + \lambda_4} + q^{\mu_1 + \mu_2 + \lambda_0 + \lambda_1 + \lambda_2 + 2\lambda_3 + \lambda_4} \\
&\quad + q^{\mu_1 + \lambda_0 + \lambda_1 + 2\lambda_2 + 2\lambda_3 + \lambda_4} + q^{\lambda_0 + 2\lambda_1 + 2\lambda_2 + 2\lambda_3 + \lambda_4}) ] ,
\end{align*}
\begin{align*}
&S_1(x)
= - q^{\mu_1 + 2\mu_2 + 3\mu_3 + 4\mu_4 - (2\lambda_0 + 3\lambda_1 + 4\lambda_2 + 5\lambda_3 + 5\lambda_4)} \\
&\quad \cdot [ \{ q^{\lambda_0 + \lambda_1 + \lambda_2 + \lambda_3}\alpha + (q^{\mu_0 + \mu_1 + \mu_2 + \mu_3 + \mu_4} + q^{\mu_0 + \mu_1 + \mu_2 + \mu_3 + \lambda_3} + q^{\mu_0 + \mu_1 + \mu_2 + \lambda_2 + \lambda_3} \\
&\quad  + q^{\mu_0 + \mu_1 + \lambda_1 + \lambda_2 + \lambda_3} + q^{\mu_0 + \lambda_0 + \lambda_1 + \lambda_2 + \lambda_3}) \beta \} x \\
&\quad - q^{\lambda_0 + \lambda_1 + \lambda_2 + \lambda_3}(q^{\mu_0 + \mu_1 + \mu_2 + \mu_3 + \mu_4} + q^{\mu_0 + \mu_1 + \mu_2 + \mu_3 + \lambda_4} + q^{\mu_0 + \mu_1 + \mu_2 + \lambda_3 + \lambda_4}  \\
&\quad + q^{\mu_0 + \mu_1 + \lambda_2 + \lambda_3 + \lambda_4} + q^{\mu_0 + \lambda_1 + \lambda_2 + \lambda_3 + \lambda_4}+ q^{\lambda_0 + \lambda_1 + \lambda_2 + \lambda_3 + \lambda_4})] .
\end{align*} 

\section{Proof outlines for Proposition \ref{prop:qhgssol}} \label{sec:prfg1}
\subsection{Outline of the calculation for \eqref{eq:g1-c1}}  $ $  \label{subsec:g1c1} \\
From equations \eqref{eq:g1n-2sol} and \eqref{eq:Klyxi-1}, we obtain
\begin{align}
&g_1^{[n]}(x) 
= x^{-\lambda_{n-2}} \int_0^{\xi_{n-2} \infty} s_{n-2}^{-1} \frac{(q^{\lambda_{n-2} + 1}s_{n-2}/x;q)_{\infty}}{(qs_{n-2}/x;q)_{\infty}} s_{n-2}^{\mu_{n-2}} \cdot s_{n-2}^{-\lambda_{n-3}} \label{eq:g1-c1-app-0} \\
& \quad \cdot \int_0^{\xi_{n-3} \infty} s_{n-3}^{-1} \frac{(q^{\lambda_{n-3} + 1}s_{n-3}/s_{n-2} ;q)_{\infty}}{(qs_{n-3}/s_{n-2}  ;q)_{\infty}} s_{n-3}^{\mu_{n-3}} \cdot s_{n-3}^{-\lambda_{n-4}} 
\cdots
\int_0^{\xi_{1} \infty} s_1^{-1} \frac{(q^{\lambda_1 + 1}s_1/s_2 ;q)_{\infty}}{(q s_1/s_2 ;q)_{\infty}} \nonumber \\
&\cdot s_1^{\mu_1} \cdot s_1^{-\lambda_0}  \int_0^{\xi_{0} \infty} s_0^{\mu_0 - 1} \frac{(q^{\lambda_0 + 1}s_0/s_1, \alpha s_0 ;q)_{\infty}}{(q s_0/s_1, \beta s_0 ;q)_{\infty}} \; d_q s_0 d_q s_1 \cdots d_q s_{n-3} d_q s_{n-2} .  \nonumber
\end{align}
We rewrite equation \eqref{eq:g1-c1-app-0} as a multiple infinite sum by setting $s_i = q^{m_i}\xi_i$, $i = 0, 1, \ldots, n-2$.
Then, introducing new summation indices $M_{n-2} = m_{n-2}$ and $M_i = m_i - m_{i+1}$ for $i = 0, \ldots, n-3$, we have
\begin{align}
g_1^{[n]}(x) &= (1 - q)^{n-1} q^{-\mu_0 L_{n-2} - \sum_{j=1}^{n-2}(\mu_j - \lambda_{j-1})T_j}
\; x^{\widetilde{M}_{n-2}- L_{n-2}}  \label{eq:g1-c1-1} \\
&\quad \cdot  
\sum_{M_{n-2} = 0}^{\infty} \sum_{M_{n-3} = 0}^{\infty} \cdots \sum_{M_{1} = 0}^{\infty} \sum_{M_{0} = 0}^{\infty} \nonumber \\
&\quad \cdot (q^{\mu_{n-2} - \lambda_{n-3}})^{M_{n-2}} (q^{\mu_{n-3} - \lambda_{n-4}})^{M_{n-3} + M_{n-2}} \nonumber \\
&\qquad\qquad \cdots (q^{\mu_{2} - \lambda_{1}})^{M_{2} + \cdots + M_{n-2}} (q^{\mu_{1} - \lambda_{0}})^{M_{1} + \cdots + M_{n-2}} (q^{\mu_0})^{M_{0} + \cdots + M_{n-2}} \nonumber \\ 
&\quad \cdot \frac{(q^{M_{n-2}+1}, q^{M_{n-3}+1}, \ldots, q^{M_1+1}, q^{M_0+1} ;q)_{\infty}}{(q^{-\lambda_{n-2}+M_{n-2}+1}, q^{-\lambda_{n-3}+M_{n-3}+1}, \ldots, q^{-\lambda_{1}+M_1+1}, q^{-\lambda_{0}+M_0+1} ;q)_{\infty}} \nonumber \\
&\quad \cdot \frac{(q^{-L_{n-2} + \sum_{j=0}^{n-2}M_j} \alpha x ;q)_{\infty}}{(q^{-L_{n-2} + \sum_{j=0}^{n-2}M_j } \beta x ;q)_{\infty}} . \nonumber
\end{align}
By the $q$-binomial theorem, we have
\begin{equation*}
\frac{(q^{-L_{n-2} + \sum_{j=0}^{n-2} M_j} \alpha x ;q)_{\infty}}{(q^{-L_{n-2} + \sum_{j=0}^{n-2} M_j} \beta x ;q)_{\infty}} 
= \sum_{N = 0}^{\infty} \frac{(\alpha/\beta ;q)_{N}}{(q ;q)_{N}} (q^{-L_{n-2} + \sum_{j=0}^{n-2} M_j} \beta x )^N,
\end{equation*}
if $|q^{-L_{n-2}} \beta x | < 1$.
Hence, equation \eqref{eq:g1-c1-1} is given by
\begin{align}
&g_1^{[n]}(x) 
= (1 - q)^{n-1} q^{-\mu_0 L_{n-2} - \sum_{j=1}^{n-2}(\mu_j - \lambda_{j-1})T_j} \; x^{\widetilde{M}_{n-2} - L_{n-2}} \label{eq:g1-c1-2} \\
&\quad \cdot \frac{(\overbrace{q, q, \ldots , q}^{n-1}   ; q)_{\infty}}{(q^{-\lambda_0 + 1} , q^{-\lambda_1 + 1} , \ldots , q^{-\lambda_{n-2} + 1} ; q)_{\infty}} \nonumber \\
&\quad \cdot \sum_{N = 0}^{\infty} \frac{(\alpha/\beta ;q)_{N}}{(q ;q)_{N}} (q^{-L_{n-2}}\beta x )^N \nonumber \\
&\qquad \cdot \left( \sum_{M_{0} = 0}^{\infty} \frac{(q^{-\lambda_{0}+1} ;q)_{M_{0}}}{(q ;q)_{M_{0}}} (q^{\mu_0 + N})^{M_{0}} \right) \left( \sum_{M_{1} = 0}^{\infty} \frac{(q^{-\lambda_{1}+1} ;q)_{M_{1}}}{(q ;q)_{M_{1}}} (q^{\mu_0 + \mu_1 - \lambda_0 + N})^{M_{1}} \right) \nonumber \\
&\qquad \cdots \left( \sum_{M_{n-3} = 0}^{\infty} \frac{(q^{-\lambda_{n-3}+1} ;q)_{M_{n-3}}}{(q ;q)_{M_{n-3}}} (q^{\mu_0 + \cdots + \mu_{n-3} - (\lambda_0 + \cdots + \lambda_{n-4}) + N})^{M_{n-3}} \right) \nonumber \\
&\qquad \cdot \left( \sum_{M_{n-2} = 0}^{\infty} \frac{(q^{-\lambda_{n-2}+1} ;q)_{M_{n-2}}}{(q ;q)_{M_{n-2}}} (q^{\mu_0 + \cdots + \mu_{n-2} - (\lambda_0 + \cdots + \lambda_{n-3}) + N})^{M_{n-2}} \right) . \nonumber  
\end{align}
By the $q$-binomial theorem, equation \eqref{eq:g1-c1-2} can be written as 
\begin{align}
&g_1^{[n]}(x) 
= (1 - q)^{n-1} q^{-\mu_0 L_{n-2} - \sum_{j=1}^{n-2}(\mu_j - \lambda_{j-1})T_j} \; x^{\widetilde{M}_{n-2} - L_{n-2}} \label{eq:g1-c1-3} \\
&\quad \cdot \frac{(\overbrace{q, q, \ldots , q}^{n-1}   ; q)_{\infty}}{(q^{-\lambda_0 + 1} , q^{-\lambda_1 + 1} , \ldots , q^{- \lambda_{n-2} + 1} ; q)_{\infty}} 
\sum_{N = 0}^{\infty}  
\prod_{j=0}^{n-2} \frac{(q^{\widetilde{M}_j - L_j + N + 1} ;q)_{\infty}}{(q^{\widetilde{M}_j - L_{j-1}  + N} ;q)_{\infty}} \nonumber \\
&\qquad \cdot \frac{(\alpha/\beta ;q)_{N}}{(q ;q)_{N}} (q^{-L_{n-2}}\beta x )^N . \nonumber
\end{align}
Equation \eqref{eq:g1-c1-3} can be rewritten as
\begin{align*}
&g_1^{[n]}(x)  = (1 - q)^{n-1} q^{-\mu_0 L_{n-2} - \sum_{j=1}^{n-2}(\mu_j - \lambda_{j-1})T_j} \; x^{\widetilde{M}_{n-2} - L_{n-2}} \\
&\quad \cdot \frac{(\overbrace{q, q, \ldots , q}^{n-1}   ; q)_{\infty}}{(q^{-\lambda_0 + 1} , q^{-\lambda_1 + 1} , \ldots , q^{- \lambda_{n-2} + 1} ; q)_{\infty}} 
\frac{(\bar{b}_1,\ldots,\bar{b}_{n-1};q)_{\infty}}{(\bar{a}_1,\ldots,\bar{a}_{n-1};q)_{\infty}} \\
&\quad \cdot \sum_{N = 0}^{\infty}  \frac{(\alpha/\beta, \bar{a}_1,\ldots,\bar{a}_{n-1};q)_{N}}{(q, \bar{b}_1,\ldots,\bar{b}_{n-1};q)_{N}}
(q^{-L_{n-2}}\beta x )^N , \nonumber
\end{align*}
where
\[
\bar{a}_k = q^{\widetilde{M}_{k-1}-L_{k-2}}, \qquad
\bar{b}_k = q^{\widetilde{M}_{k-1}-L_{k-1}+1}, \qquad (k = 1,\ldots,n-1).
\]
Thus, we obtain equation \eqref{eq:g1-c1}.

\subsection{Outline of the calculation for \eqref{eq:g1-c3}}  $ $ \label{subsec:g1c3} \\
From equations \eqref{eq:g1n-2sol} and \eqref{eq:Klyxi-3}, we obtain
\begin{align}
&g_1^{[n]}(x) 
= \int_0^{\xi_{n-2} \infty} s_{n-2}^{-\lambda_{n-2} - 1} \frac{(x/s_{n-2};q)_{\infty}}{(q^{-\lambda_{n-2}}x/s_{n-2};q)_{\infty}} s_{n-2}^{\mu_{n-2}} \label{eq:g1-c3-app-0}  \\
& \quad \cdot \int_0^{\xi_{n-3} \infty} s_{n-3}^{-\lambda_{n-3} - 1} \frac{(s_{n-2}/s_{n-3} ;q)_{\infty}}{(q^{-\lambda_{n-3}}s_{n-2}/s_{n-3};q)_{\infty}} s_{n-3}^{\mu_{n-3}} 
\cdots
\int_0^{\xi_{1} \infty} s_{1}^{-\lambda_{1} - 1} \frac{(s_2/s_1 ;q)_{\infty}}{(q^{-\lambda_1} s_2/s_1 ;q)_{\infty}} \nonumber \\
&\quad \cdot s_1^{\mu_1}  \int_0^{\xi_{0} \infty} s_{0}^{\widetilde{\mu}-\lambda_{0} - 1} \frac{(s_1/s_0, q/(\beta s_0) ;q)_{\infty}}{(q^{-\lambda_0} s_1/s_0, q/(\alpha s_0) ;q)_{\infty}} \; d_q s_0 d_q s_1 \cdots d_q s_{n-3} d_q s_{n-2} . \nonumber
\end{align}
We rewrite equation \eqref{eq:g1-c3-app-0} as a multiple infinite sum by setting $s_i = q^{m_i}\xi_i$, $i = 0, 1, \ldots, n-2$.
Then, introducing new summation indices $M_{n-2} = -m_{n-2}$ and $M_i = m_{i+1} - m_{i}$ for $i = 0, \ldots, n-3$, we have
\begin{align}
&g_1^{[n]}(x) = (1 - q)^{n-1} q^{\sum_{j=0}^{n-2}(n-1-j)(-\mu_j + \lambda_j)} 
 \left( \frac{\alpha}{\beta} \right)^{n-1}  x^{\widetilde{\mu} + \sum_{j=1}^{n-2}\mu_j - L_{n-2}} \label{eq:g1-c3-1} \\
&\quad \cdot  
\sum_{M_{n-2} = 0}^{\infty} \sum_{M_{n-3} = 0}^{\infty} \cdots \sum_{M_{1} = 0}^{\infty} \sum_{M_{0} = 0}^{\infty} \nonumber \\
&\quad \cdot (q^{-(\mu_{n-2} - \lambda_{n-2})})^{M_{n-2}} (q^{-(\mu_{n-3} - \lambda_{n-3})})^{M_{n-3} + M_{n-2}} \nonumber \\
&\qquad \cdots (q^{-(\mu_{2} - \lambda_{2})})^{M_{2} + \cdots + M_{n-2}} (q^{-(\mu_{1} - \lambda_{1})})^{M_{1} + \cdots + M_{n-2}} (q^{-(\mu_0-\lambda_0)}\alpha/\beta)^{M_{0} + \cdots + M_{n-2}} \nonumber \\ 
&\quad \cdot \frac{(q^{M_{n-2}+1}, q^{M_{n-3}+1}, \ldots, q^{M_1+1}, q^{M_0+1} ;q)_{\infty}}{(q^{-\lambda_{n-2}+M_{n-2}+1}, q^{-\lambda_{n-3}+M_{n-3}+1}, \ldots, q^{-\lambda_{1}+M_1+1}, q^{-\lambda_{0}+M_0+1} ;q)_{\infty}} \nonumber \\
&\quad \cdot \frac{(q^{n + \sum_{j=0}^{n-2} M_j}/(\beta x) ;q)_{\infty}}{(q^{n + \sum_{j=0}^{n-2} M_j }/(\alpha x) ;q)_{\infty}} . \nonumber
\end{align}
By the $q$-binomial theorem, we have
\begin{equation*}
\frac{(q^{n + \sum_{j=0}^{n-2} M_j }/(\beta x) ;q)_{\infty}}{(q^{n + \sum_{j=0}^{n-2} M_j }/(\alpha x) ;q)_{\infty}} 
= \sum_{N = 0}^{\infty} \frac{(\alpha/\beta ;q)_{N}}{(q ;q)_{N}} (q^{n + \sum_{j=0}^{n-2} M_j }/(\alpha x) )^N, 
\end{equation*}
if $|q^{n}/(\alpha x) | < 1$.
Hence, equation \eqref{eq:g1-c3-1} is given by
\begin{align}
&g_1^{[n]}(x) 
= (1 - q)^{n-1} q^{\sum_{j=0}^{n-2}(n-1-j)(-\mu_j + \lambda_j)}  
\left( \frac{\alpha}{\beta} \right)^{n-1} x^{\widetilde{\mu} + \sum_{j=1}^{n-2}\mu_j - L_{n-2}} \label{eq:g1-c3-2} \\
&\quad \cdot \frac{(\overbrace{q, q, \ldots , q}^{n-1}   ; q)_{\infty}}{(q^{-\lambda_0 + 1} , q^{-\lambda_1 + 1} , \ldots , q^{-\lambda_{n-2} + 1} ; q)_{\infty}} \nonumber \\
&\quad \cdot \sum_{N = 0}^{\infty} \frac{(\alpha/\beta ;q)_{N}}{(q ;q)_{N}} (q^{n}/(\alpha x) )^N \nonumber \\
&\qquad \cdot \left( \sum_{M_{0} = 0}^{\infty} \frac{(q^{-\lambda_{0}+1} ;q)_{M_{0}}}{(q ;q)_{M_{0}}} (q^{-\mu_0 + \lambda_0 + N}\alpha/\beta)^{M_{0}} \right) \nonumber \\
&\qquad \cdot \left( \sum_{M_{1} = 0}^{\infty} \frac{(q^{-\lambda_{1}+1} ;q)_{M_{1}}}{(q ;q)_{M_{1}}} (q^{-(\mu_0 + \mu_1) + \lambda_0 + \lambda_1 + N}\alpha/\beta)^{M_{1}} \right) \nonumber \\
&\qquad \cdots \left( \sum_{M_{n-3} = 0}^{\infty} \frac{(q^{-\lambda_{n-3}+1} ;q)_{M_{n-3}}}{(q ;q)_{M_{n-3}}} (q^{-\widetilde{M}_{n-3} + L_{n-3} + N}\alpha/\beta)^{M_{n-3}} \right) \nonumber \\
&\qquad \cdot \left( \sum_{M_{n-2} = 0}^{\infty} \frac{(q^{-\lambda_{n-2}+1} ;q)_{M_{n-2}}}{(q ;q)_{M_{n-2}}} (q^{-\widetilde{M}_{n-2} + L_{n-2} + N}\alpha/\beta)^{M_{n-2}} \right) . \nonumber  
\end{align}
By the $q$-binomial theorem, equation \eqref{eq:g1-c3-2} can be written as 
\begin{align}
&g_1^{[n]}(x) = (1 - q)^{n-1} q^{\sum_{j=0}^{n-2}(n-1-j)(-\mu_j + \lambda_j)}   
\left( \frac{\alpha}{\beta} \right)^{n-1} x^{\widetilde{\mu} + \sum_{j=1}^{n-2}\mu_j - L_{n-2}} \label{eq:g1-c3-3} \\
&\quad \cdot \frac{(\overbrace{q, \ldots , q}^{n-1} ; q)_{\infty}}{(q^{-\lambda_0 + 1} , q^{-\lambda_1 + 1} , \ldots , q^{- \lambda_{n-2} + 1} ; q)_{\infty}} \nonumber \\
&\quad \cdot \sum_{N = 0}^{\infty} 
\prod_{j=0}^{n-2} \frac{(q^{-\widetilde{M}_j + L_{j-1} + N + 1}\alpha/\beta ;q)_{\infty}}{(q^{-\widetilde{M}_j + L_{j} + N}\alpha/\beta ;q)_{\infty}} \frac{(\alpha/\beta ;q)_{N}}{(q ;q)_{N}} (q^{n }/(\alpha x) )^N  . \nonumber 
\end{align}
Equation \eqref{eq:g1-c3-3} can be rewritten as
\begin{align*}
&g_1^{[n]}(x) = (1 - q)^{n-1} q^{\sum_{j=o}^{n-2}(n-1-j)(-\mu_j + \lambda_j)} \left( \frac{\alpha}{\beta} \right)^{n-1} x^{\widetilde{\mu} + \sum_{j=1}^{n-2}\mu_j - L_{n-2}} \\
&\quad \cdot \frac{(\overbrace{q, q, \ldots , q}^{n-1}; q)_{\infty}}{(q^{-\lambda_0 + 1} , q^{-\lambda_1 + 1} , \ldots , q^{- \lambda_{n-2} + 1} ; q)_{\infty}} 
\frac{(\bar{b}_1,\ldots,\bar{b}_{n-1};q)_{\infty}}{(\bar{a}_1,\ldots,\bar{a}_{n-1};q)_{\infty}} \\
&\quad \cdot \sum_{N = 0}^{\infty}  \frac{(\alpha/\beta, \bar{a}_1,\ldots,\bar{a}_{n-1};q)_{N}}{(q, \bar{b}_1,\ldots,\bar{b}_{n-1};q)_{N}}
(q^{n }/(\alpha x) )^N , 
\end{align*}
where
\begin{align*}
& \bar{a}_k = q^{-\widetilde{M}_{k-1} + L_{k-1}} \alpha/\beta, \quad 1\le k \le n-1, \\
& \bar{b}_k = q^{-\widetilde{M}_{k-1} + L_{k-2} + 1} \alpha/\beta, \quad 1 \le k \le n-1.
\end{align*}
Thus, we obtain equation \eqref{eq:g1-c3}.

\subsection{Outline of the calculation for \eqref{eq:g1-c4}}  $ $ \label{subsec:g1c4} \\
We first consider the case $1 \le r \le n-3$ for $n \ge 4$.
From equations \eqref{eq:g1n-2sol} and \eqref{eq:Klyxi-4}, we obtain
\begin{align}
&g_1^{[n]}(x) 
= \int_0^{\xi_{n-2} \infty} s_{n-2}^{-\lambda_{n-2} - 1} \frac{(x/s_{n-2};q)_{\infty}}{(q^{-\lambda_{n-2}}x/s_{n-2};q)_{\infty}} s_{n-2}^{\mu_{n-2}} \label{eq:g1-c4-app-0} \\
& \quad \cdot \int_0^{\xi_{n-3} \infty} s_{n-3}^{-\lambda_{n-3} - 1} \frac{(s_{n-2}/s_{n-3} ;q)_{\infty}}{(q^{-\lambda_{n-3}}s_{n-2}/s_{n-3};q)_{\infty}} s_{n-3}^{\mu_{n-3}} 
\cdots
\int_0^{\xi_{r+1} \infty} s_{r+1}^{-\lambda_{r+1} - 1} \frac{(s_{r+2}/s_{r+1} ;q)_{\infty}}{(q^{-\lambda_{r+1}}s_{r+2}/s_{r+1};q)_{\infty}} \nonumber \\
&\quad \cdot s_{r+1}^{\mu_{r+1}} \cdot s_{r+1}^{-\lambda_{r}}
\int_0^{\xi_{r} \infty} s_{r}^{-1} \frac{(q^{\lambda_{r} + 1}s_{r}/s_{r+1} ;q)_{\infty}}{(q s_{r}/s_{r+1};q)_{\infty}} s_{r}^{\mu_{r}} 
\int_0^{\xi_{r-1} \infty} s_{r-1}^{-\lambda_{r-1} - 1} \frac{(s_{r}/s_{r-1} ;q)_{\infty}}{(q^{-\lambda_{r-1}}s_{r}/s_{r-1};q)_{\infty}} s_{r-1}^{\mu_{r-1}} \nonumber \\
&\quad \cdots \int_0^{\xi_{1} \infty} s_{1}^{-\lambda_{1} - 1} \frac{(s_2/s_1 ;q)_{\infty}}{(q^{-\lambda_1} s_2/s_1 ;q)_{\infty}} 
s_1^{\mu_1}  \int_0^{\xi_{0} \infty} s_{0}^{\mu_0-\lambda_{0} - 1} \frac{(s_1/s_0, \alpha s_0 ;q)_{\infty}}{(q^{-\lambda_0} s_1/s_0, \beta s_0 ;q)_{\infty}} \nonumber \\
& \hspace{280pt} \cdot d_q s_0 d_q s_1 \cdots d_q s_{n-3} d_q s_{n-2} .  \nonumber
\end{align}
We rewrite equation \eqref{eq:g1-c4-app-0} as a multiple infinite sum by setting $s_i = q^{m_i}\xi_i$, $i = 0, 1, \ldots, n-2$.
Then, introducing new summation indices $M_0 = m_0$, $M_{n-2} = -m_{n-2}$, $M_i = m_{i} - m_{i-1}$ for $i = 1, \ldots, r$, and $M_j = m_{j+1} - m_j$ for $j = r+1, \ldots, n-3$, we have
\begin{align}
&g_1^{[n]}(x)  = (1 - q)^{n-1} q^{F_r} \; \alpha^{-\widetilde{M}_r + L_{r-1}} \; x^{\sum_{j=r+1}^{n-2}\mu_j - T_r} \label{eq:g1-c4-1} \\
&\quad \cdot  
\sum_{M_{n-2} = 0}^{\infty} \sum_{M_{n-3} = 0}^{\infty} \cdots \sum_{M_{r+1} = 0}^{\infty} \sum_{M_{r} = 0}^{\infty} \sum_{M_{r-1} = 0}^{\infty} \cdots \sum_{M_{1} = 0}^{\infty} \sum_{M_{0} = 0}^{\infty} \nonumber \\
&\quad \cdot (q^{\mu_{n-2} - \lambda_{n-2}})^{-M_{n-2}} (q^{\mu_{n-3} - \lambda_{n-3}})^{-(M_{n-3} + M_{n-2})} \cdots (q^{\mu_{r+2} - \lambda_{r+2}})^{-(M_{r+2} + \cdots + M_{n-2})} \nonumber \\
& \quad \cdot (q^{\mu_{r+1} - (\lambda_{r} + \lambda_{r+1})})^{-(M_{r+1} + \cdots + M_{n-2})} \nonumber \\
& \quad \cdot (q^{\mu_{r}})^{M_0 + \cdots + M_{r}}
(q^{\mu_{r-1} - \lambda_{r-1}})^{M_{0} + \cdots + M_{r-1}} 
\cdots (q^{\mu_{1} - \lambda_{1}})^{M_{0} + M_{1}} (q^{\mu_0 - \lambda_0})^{M_{0}} \nonumber \\ 
&\quad \cdot \frac{(q^{M_{n-2}+1}, \ldots, q^{M_{r+1}+1}, q^{M_{r-1}+1}, \ldots, q^{M_0+1} ;q)_{\infty}}{(q^{-\lambda_{n-2}+M_{n-2}+1}, \ldots, q^{-\lambda_{r+1}+M_{r+1}+1}, q^{-\lambda_{r-1}+M_{r-1}+1}, \ldots, q^{-\lambda_{0}+M_0+1} ;q)_{\infty}} \nonumber \\
&\quad \cdot \frac{(q^{\lambda_r + n + \sum_{j=0}^{n-2}M_j}/(\alpha x) ;q)_{\infty}}{(q^{n + \sum_{j=0}^{n-2}M_j } /(\alpha x) ;q)_{\infty}} . \nonumber
\end{align}
By the $q$-binomial theorem, we have
\begin{equation*}
\frac{(q^{\lambda_r + n + \sum_{j=0}^{n-2}M_j }/(\alpha x) ;q)_{\infty}}{(q^{n + \sum_{j=0}^{n-2}M_j }/(\alpha x) ;q)_{\infty}}
= \sum_{N = 0}^{\infty} \frac{(q^{\lambda_r} ;q)_{N}}{(q ;q)_{N}} (q^{n + \sum_{j=0}^{n-2}M_j }/(\alpha x) )^N, 
\end{equation*}
if $|q^{n}/(\alpha x) | < 1$.
Hence, the fourth and subsequent lines of equation \eqref{eq:g1-c4-1} are given by
\begin{align}
&\frac{(\overbrace{q, q, \ldots , q}^{n-1}; q)_{\infty}}{(q\beta/\alpha, q^{-\lambda_0 + 1} , q^{-\lambda_1 + 1} , \ldots , q^{-\lambda_{r-1} + 1}, q^{-\lambda_{r+1} + 1}, \ldots, q^{-\lambda_{n-2} + 1} ; q)_{\infty}} \label{eq:g1-c4-2}  \\
&\cdot \sum_{N = 0}^{\infty} \frac{(q^{\lambda_r} ;q)_{N}}{(q ;q)_{N}} (q^{n}/(\alpha x) )^N \nonumber \\
&\qquad \cdot \left( \sum_{M_{0} = 0}^{\infty} \frac{(q\beta/\alpha ;q)_{M_{0}}}{(q ;q)_{M_{0}}} (q^{\widetilde{M}_r  - L_{r-1}+ N})^{M_{0}} \right) \nonumber \\
&\qquad \cdot \left( \sum_{M_{1} = 0}^{\infty} \frac{(q^{-\lambda_{0}+1} ;q)_{M_{1}}}{(q ;q)_{M_{1}}} (q^{\mu_1 + \cdots + \mu_r - (\lambda_1 + \cdots + \lambda_{r-1})+ N} )^{M_{1}} \right) \nonumber \\
&\qquad \cdots \left( \sum_{M_{r-1} = 0}^{\infty} \frac{(q^{-\lambda_{r-2}+1} ;q)_{M_{r-1}}}{(q ;q)_{M_{r-1}}} (q^{\mu_{r-1} + \mu_r - \lambda_{r-1}+ N} )^{M_{r-1}} \right) \nonumber \\
&\qquad \cdot \left( \sum_{M_{r} = 0}^{\infty} \frac{(q^{-\lambda_{r-1}+1} ;q)_{M_{r}}}{(q ;q)_{M_{r}}} (q^{\mu_r + N} )^{M_{r}} \right) \nonumber \\
&\qquad \cdot \left( \sum_{M_{r+1} = 0}^{\infty} \frac{(q^{-\lambda_{r+1}+1} ;q)_{M_{r+1}}}{(q ;q)_{M_{r+1}}} (q^{-\mu_{r+1} + \lambda_r + \lambda_{r+1} + N} )^{M_{r+1}} \right) \nonumber \\
&\qquad \cdots \left( \sum_{M_{n-2} = 0}^{\infty} \frac{(q^{-\lambda_{n-2}+1} ;q)_{M_{n-2}}}{(q ;q)_{M_{n-2}}} (q^{-(\mu_{r+1} + \cdots + \mu_{n-2}) + \lambda_r + \cdots + \lambda_{n-2} + N})^{M_{n-2}} \right) . \nonumber  
\end{align}
By the $q$-binomial theorem, equation \eqref{eq:g1-c4-2} can be written as 
\begin{align}
& \frac{(\overbrace{q, q, \ldots , q}^{n-1}   ; q)_{\infty}}{(q\beta/\alpha, q^{-\lambda_0 + 1} , q^{-\lambda_1 + 1} , \ldots , q^{-\lambda_{r-1} + 1}, q^{-\lambda_{r+1} + 1}, \ldots, q^{-\lambda_{n-2} + 1} ; q)_{\infty}} \label{eq:g1-c4-3}\\
&\sum_{N = 0}^{\infty}  
\frac{(q^{\mu_0 + \cdots + \mu_{r} - (\lambda_0 + \cdots + \lambda_{r-1}) + N + 1} \beta/\alpha ; q)_{\infty}}{(q^{\mu_0 + \cdots + \mu_{r} - (\lambda_0 + \cdots + \lambda_{r-1}) + N} ; q)_{\infty}}
\frac{(q^{\mu_1 + \cdots + \mu_{r} - (\lambda_0 + \cdots + \lambda_{r-1}) + N + 1} ; q)_{\infty}}{(q^{\mu_1 + \cdots + \mu_{r} - (\lambda_1 + \cdots + \lambda_{r-1}) + N} ; q)_{\infty}} \nonumber \\
&\quad \cdots \frac{(q^{\mu_{r-1} + \mu_{r} - (\lambda_{r-2} + \lambda_{r-1}) + N + 1} ; q)_{\infty}}{(q^{\mu_{r-1} + \mu_{r} - \lambda_{r-1} + N} ; q)_{\infty}}
\frac{(q^{\mu_{r} - \lambda_{r-1} + N + 1} ; q)_{\infty}}{(q^{\mu_{r} + N} ; q)_{\infty}} \nonumber \\
&\quad \cdot\frac{(q^{-\mu_{r+1} + \lambda_{r} + N + 1} ;q)_{\infty}}{(q^{-\mu_{r+1} + \lambda_{r} + \lambda_{r+1} + N} ;q)_{\infty}}
\frac{(q^{-(\mu_{r+1} + \mu_{r+2}) + \lambda_{r} + \lambda_{r+1} + N + 1} ;q)_{\infty}}{(q^{-(\mu_{r+1} + \mu_{r+2}) + \lambda_{r} + \lambda_{r+1} + \lambda_{r+2} + N} ;q)_{\infty}} \nonumber \\
&\quad \cdots 
\frac{(q^{-(\mu_{r+1} + \cdots + \mu_{n-2}) + \lambda_{r} + \cdots + \lambda_{n-3} + N + 1} ;q)_{\infty}}{(q^{-(\mu_{r+1} + \cdots + \mu_{n-2}) + \lambda_{r} + \cdots + \lambda_{n-2} + N} ;q)_{\infty}} 
\frac{(q^{\lambda_{r}} ;q)_{N}}{(q ;q)_{N}} (q^{n }/(\alpha x) )^N .  \nonumber
\end{align}
Equation \eqref{eq:g1-c4-3} can be rewritten as
\begin{align*}
& \frac{(\overbrace{q, q, \ldots , q}^{n-1}   ; q)_{\infty}}{(q\beta/\alpha, q^{-\lambda_0 + 1} , q^{-\lambda_1 + 1} , \ldots , q^{-\lambda_{r-1} + 1}, q^{-\lambda_{r+1} + 1}, \ldots, q^{-\lambda_{n-2} + 1} ; q)_{\infty}}  \\
&\cdot \frac{(\bar{b}_1,\ldots,\bar{b}_{n-1};q)_{\infty}}{(\bar{a}_1,\ldots,\bar{a}_{n-1};q)_{\infty}}  \;
\sum_{N = 0}^{\infty}  \frac{(q^{\lambda_{r}}, \bar{a}_1,\ldots,\bar{a}_{n-1};q)_{N}}{(q, \bar{b}_1,\ldots,\bar{b}_{n-1};q)_{N}}
(q^{n }/(\alpha x) )^N , 
\end{align*}
where
\begin{align*}
&\bar{a}_k=
\begin{cases}
q^{\sum_{j=k-1}^{r}\mu_j - \sum_{j=k-1}^{r-1}\lambda_j}, &1\le k \le r+1,\\
q^{-\sum_{j=r+1}^{k-1}\mu_j + \sum_{j=r}^{k-1}\lambda_j}, &r+2\le k \le n-1,
\end{cases} \\
&\bar{b}_1
= q^{\widetilde{M}_r - L_{r-1} + 1} \beta/\alpha, \\
&\bar{b}_k=
\begin{cases}
q^{\sum_{j=k-1}^{r}\mu_j - \sum_{j=k-2}^{r-1}\lambda_j + 1}, &2\le k \le r+1,\\
q^{-\sum_{j=r+1}^{k-1}\mu_j + \sum_{j=r}^{k-2}\lambda_j + 1}, &r+2\le k \le n-1.
\end{cases}
\end{align*}
By combining with \eqref{eq:g1-c4-1}, we obtain equation \eqref{eq:g1-c4}.
The cases $r = 0$ and $r = n-2$ can be treated similarly.

\section{$q$-Hypergeometric series solutions} \label{sec:exqhgs}
\subsection{$q$-hypergeometric series solutions to equation \eqref{eq:111-21-111-g1}} $ $ \label{subsec:qhgs3} \\
\ref{case:1}: The following solution coincides with equation \eqref{eq:g1qlalaxqlax}.
\begin{align*}
g^{[3]}_1(x) 
&= (1 - q)^2 q^{ -\mu_0(\lambda_0 + \lambda_1) - (\mu_1 - \lambda_0) \lambda_1} \; x^{\mu_0 + \mu_1 - (\lambda_1 + \lambda_0) }  \\
&\qquad \cdot \frac{(q, q, q^{\mu_0 -\lambda_0 + 1} , q^{\mu_0 +\mu_1 - (\lambda_0 + \lambda_1) + 1 } ; q)_{\infty}}{(q^{-\lambda_0 + 1} , q^{-\lambda_1 + 1} , q^{\mu_0 } , q^{\mu_0 + \mu_1 - \lambda_0 } ; q)_{\infty}} \\
&\qquad\qquad \cdot 
{}_3\phi_2 \Bigl( { \alpha /\beta , q^{\mu_0 }, q^{\mu_0 + \mu_1  - \lambda_0 } \atop{q^{\mu_0 -\lambda_0 + 1} , q^{\mu_0 +\mu_1  - (\lambda_0 + \lambda_1) + 1 }}} ; q, q^{ -(\lambda_0 + \lambda_1) } \beta x \Bigr).
\end{align*}
\ref{case:2} $(r = 0)$ :
\begin{align*}
&g^{[3]}_1(x) = (1 - q)^2 q^{- \mu_1 \lambda_1} \; \beta^{- \widetilde{\mu} + \lambda_0} \; x^{\mu_1 - \lambda_1} 
\frac{(q, q, q^{- \mu_0 + \lambda_0 + 1 }, q^{\mu_1 - \lambda_1 + 1}  ; q)_{\infty}}{(q\beta/\alpha, q^{-\lambda_1 + 1}, q^{- \mu_0  + \lambda_0} \alpha/\beta, q^{ \mu_1 }  ; q)_{\infty}}  \\
&\hspace{120pt}  \cdot {}_3\phi_2 \Bigl( {q^{\lambda_0} , q^{- \mu_0 + \lambda_0} \alpha /\beta, q^{ \mu_1}  \atop{q^{- \mu_0 + \lambda_0 + 1 }, q^{ \mu_1 - \lambda_1 + 1} }} ; q, q^{ -(\lambda_0 + \lambda_1) } \beta x \Bigr) . 
\end{align*}
\ref{case:2} $(r = 1)$ :
\begin{align*}
g^{[3]}_1(x)  
&= (1 - q)^2 q^{\lambda_0 (\mu_1 -(\lambda_0 + \lambda_1))} \; \beta^{- (\widetilde{\mu} + \mu_1) + \lambda_0 + \lambda_1}  \\
&\qquad \frac{(q, q, q^{- (\mu_0 + \mu_1) + \lambda_0 + \lambda_1 + 1}, q^{-\mu_1 + \lambda_1 + 1} ; q)_{\infty}}{(q\beta/\alpha , q^{-\lambda_0 + 1} , q^{- (\mu_0 + \mu_1) + \lambda_0 + \lambda_1} \alpha/\beta, q^{- \mu_1 + \lambda_0 + \lambda_1} ; q)_{\infty}} \\
&\qquad\qquad \cdot {}_3\phi_2 \Bigl( {q^{\lambda_1 }, q^{- (\mu_0 + \mu_1) + \lambda_0 + \lambda_1} \alpha /\beta , q^{- \mu_1 + \lambda_0 + \lambda_1} \atop{ q^{- (\mu_0 + \mu_1) + \lambda_0 + \lambda_1 + 1 }, q^{- \mu_1 + \lambda_1 + 1}}} ; q, q^{ -(\lambda_0 + \lambda_1) } \beta x \Bigr) . 
\end{align*}
\ref{case:3} :
\begin{align*}
&g^{[3]}_1(x)  = (1 - q)^2 q^{-(2\mu_0 + \mu_1) + 2\lambda_0 + \lambda_1}  \left( \frac{\alpha}{\beta} \right)^2 x^{ \widetilde{\mu} + \mu_1 - (\lambda_0 + \lambda_1)} \\
&\qquad \frac{(q, q, q^{ -\mu_0 + 1 }\alpha/\beta, q^{-(\mu_0 + \mu_1) + \lambda_0 + 1} \alpha/\beta ; q)_{\infty}}{(q^{-\lambda_0 + 1} , q^{-\lambda_1 + 1} , q^{ -\mu_0 + \lambda_0}\alpha/\beta, q^{-(\mu_0 + \mu_1) + \lambda_0 + \lambda_1 }\alpha/\beta  ; q)_{\infty}}  \\
&\hspace{120pt} \cdot  \;  {}_3\phi_2 \Bigl( {\alpha/\beta, q^{-\mu_0 + \lambda_0 }\alpha/\beta  , q^{-(\mu_0 + \mu_1) + \lambda_0 + \lambda_1}\alpha/\beta    \atop{q^{-\mu_0 + 1 }\alpha/\beta , q^{-(\mu_0 + \mu_1) + \lambda_0 + 1}\alpha/\beta }} ; q, q^{3}/(\alpha x) \Bigr) .
\end{align*}
\ref{case:4} $(r = 0)$ :
\begin{align*}
&g^{[3]}_1(x) = (1 - q)^2 q^{\mu_0 - \mu_1 + \lambda_0 + \lambda_1} \; \alpha^{- \mu_0} \; x^{\mu_1 - (\lambda_0 + \lambda_1)}
\frac{(q, q, q^{ \mu_0 + 1 }\beta/\alpha, q^{-\mu_1 + \lambda_0 + 1} ; q)_{\infty}}{(q\beta/\alpha , q^{-\lambda_1 + 1} , q^{ \mu_0}, q^{- \mu_1 + \lambda_0 + \lambda_1  } ; q)_{\infty}} \\
&\hspace{160pt} \cdot {}_3\phi_2 \Bigl( {q^{\lambda_0 }, q^{\mu_0 }, q^{- \mu_1 + \lambda_0 + \lambda_1} \atop{q^{\mu_0 + 1 }\beta/\alpha, q^{- \mu_1 + \lambda_0 + 1} }} ; q, q^{3}/(\alpha x) \Bigr) . 
\end{align*}
\ref{case:4} $(r = 1)$ :
\begin{align*}
&g^{[3]}_1(x) = (1 - q)^2 q^{\mu_0 + 2\mu_1 - \lambda_0} \; \alpha^{-(\mu_0 + \mu_1) + \lambda_0} \; x^{ - \lambda_1} 
\frac{(q, q, q^{\mu_0 + \mu_1 - \lambda_0 + 1} \beta/\alpha , q^{ \mu_1 - \lambda_0 + 1 } ; q)_{\infty}}{(q\beta/\alpha , q^{-\lambda_0 + 1} , q^{\mu_0 + \mu_1 - \lambda_0 } , q^{ \mu_1} ; q)_{\infty}}\\
&\hspace{160pt} \cdot {}_3\phi_2 \Bigl( {q^{\lambda_1 }, q^{\mu_0 + \mu_1 - \lambda_0} , q^{\mu_1 }  \atop{q^{\mu_0 + \mu_1 - \lambda_0 + 1}\beta/\alpha , q^{\mu_1 - \lambda_0 + 1 }}} ; q, q^{3}/(\alpha x) \Bigr) . 
\end{align*}

\subsection{$q$-hypergeometric series solutions to equation \eqref{eq:g1-o4}} $ $ \label{subsec:qhgs4} \\
\ref{case:1} :
\begin{align*}
&g^{[4]}_1(x) 
= (1 - q)^3 q^{ -\mu_0 (\lambda_0 + \lambda_1 + \lambda_2) - (\mu_1 - \lambda_0)(\lambda_1 + \lambda_2) - (\mu_2 - \lambda_1)\lambda_2} \;
x^{\mu_0 + \mu_1 + \mu_2 - (\lambda_0 + \lambda_1 + \lambda_2) } \\  
&\quad \cdot \frac{(q, q, q, q^{\mu_0 - \lambda_0 + 1} , q^{\mu_0 + \mu_1 - (\lambda_0 + \lambda_1) + 1 } , q^{\mu_0 + \mu_1 + \mu_2 - (\lambda_0 + \lambda_1 + \lambda_2) + 1 }; q)_{\infty}}{(q^{-\lambda_0 + 1 } , q^{-\lambda_1 + 1 } , q^{-\lambda_2 + 1 } , q^{\mu_0 } , q^{\mu_0 + \mu_1 - \lambda_0 } , q^{\mu_0 + \mu_1 + \mu_2 - (\lambda_0 + \lambda_1) }; q)_{\infty}} \\
&\quad \cdot {}_4\phi_3 \Bigl( { \alpha /\beta , q^{\mu_0 }, q^{\mu_0 +\mu_1  - \lambda_0 } , q^{\mu_0 + \mu_1 + \mu_2 - (\lambda_0 + \lambda_1)} 
\atop{q^{\mu_0 - \lambda_0 + 1} , q^{\mu_0 + \mu_1 - (\lambda_0 + \lambda_1) + 1 } , q^{\mu_0 + \mu_1 + \mu_2 - (\lambda_0 + \lambda_1 + \lambda_2) + 1 }}}; q, q^{ -(\lambda_0 + \lambda_1 + \lambda_2) } \beta x \Bigr).
\end{align*}
\ref{case:2} $(r = 0)$ :
\begin{align*}
g^{[4]}_1(x) 
&= (1 - q)^3 q^{-(\lambda_1 + \lambda_2)\mu_1 - \lambda_2(\mu_2 - \lambda_1) } \; \beta^{- \widetilde{\mu} + \lambda_0 } \;  x^{\mu_1 + \mu_2 - (\lambda_1 + \lambda_2)} \\
&\quad \cdot \frac{(q, q, q, q^{-\mu_0 + \lambda_0 + 1}, q^{\mu_1 - \lambda_1 + 1}, q^{\mu_1 + \mu_2 - (\lambda_1 + \lambda_2) + 1}  ; q)_{\infty}}{(q\beta/\alpha , q^{-\lambda_1 + 1 } , q^{-\lambda_2 + 1 } , q^{- \mu_0 + \lambda_0 } \alpha/\beta , q^{ \mu_1 } , q^{\mu_1  + \mu_2 - \lambda_1}; q)_{\infty}}  \\
&\quad \cdot {}_4\phi_3 \Bigl( {q^{\lambda_0} , q^{- \mu_0 + \lambda_0 } \alpha/\beta , q^{ \mu_1} ,  q^{\mu_1 + \mu_2 -\lambda_1}  
\atop{q^{-\mu_0 + \lambda_0 + 1}, q^{\mu_1 - \lambda_1 + 1}, q^{\mu_1 + \mu_2 - (\lambda_1 + \lambda_2) + 1}}} ; q, q^{ -(\lambda_0 + \lambda_1 + \lambda_2) } \beta x \Bigr) . 
\end{align*}
\ref{case:2} $(r = 1)$ :
\begin{align*}
g^{[4]}_1(x) 
&= (1 - q)^3 q^{\lambda_0(\mu_1 - (\lambda_0 + \lambda_1)) - \mu_2\lambda_2} \; \beta^{- (\widetilde{\mu} + \mu_1) + \lambda_0 + \lambda_1} \; x^{\mu_2 - \lambda_2} \\
&\quad \cdot \frac{(q, q, q, q^{-(\mu_0 + \mu_1) + \lambda_0 + \lambda_1 + 1}, q^{- \mu_1 + \lambda_1 + 1}, q^{\mu_2 - \lambda_2 + 1} ; q)_{\infty}}{(q\beta/\alpha , q^{-\lambda_0 + 1 } , q^{-\lambda_2 + 1 } , q^{- (\mu_0 + \mu_1) + \lambda_0 + \lambda_1} \alpha/\beta, q^{- \mu_1  + \lambda_0 + \lambda_1}, q^{ \mu_2 } ; q)_{\infty}}  \\
&\quad \cdot {}_4\phi_3 \Bigl( {q^{\lambda_1} , q^{- (\mu_0 + \mu_1) + \lambda_0 + \lambda_1} \alpha /\beta, q^{- \mu_1  + \lambda_0 + \lambda_1},  q^{ \mu_2} 
\atop{q^{-(\mu_0 + \mu_1) + \lambda_0 + \lambda_1 + 1 }, q^{- \mu_1 + \lambda_1 + 1}, q^{ \mu_2 - \lambda_2 + 1} }} ; q, q^{ -(\lambda_0 + \lambda_1 + \lambda_2) } \beta x \Bigr) . 
\end{align*}
\ref{case:2} $(r = 2)$ :
\begin{align*}
g^{[4]}_1(x) 
&= (1 - q)^3 q^{\lambda_0(\mu_1 - (\lambda_0 + \lambda_1)) + (\lambda_0 + \lambda_1)(\mu_2 - (\lambda_1 + \lambda_2))}  \; \beta^{- (\widetilde{\mu} + \mu_1 + \mu_2) + \lambda_0 + \lambda_1 + \lambda_2} \\
&\quad \cdot \frac{(q, q, q, \bar{b}_1, \bar{b}_2, \bar{b}_3 ; q)_{\infty}}{(q\beta/\alpha , q^{-\lambda_0 + 1 } , q^{-\lambda_1 + 1 }, \bar{a}_1, \bar{a}_2, \bar{a}_3 ; q)_{\infty}} \\
&\quad \cdot {}_4\phi_3 \Bigl( {q^{\lambda_2} , \bar{a}_1, \bar{a}_2, \bar{a}_3  
\atop{\bar{b}_1, \bar{b}_2, \bar{b}_3  }} ; q, q^{ -(\lambda_0 + \lambda_1 + \lambda_2) } \beta x \Bigr) , 
\end{align*}
where
\begin{align*}
& \bar{a}_1 = q^{- (\mu_0 + \mu_1 + \mu_2) + \lambda_0 + \lambda_1 + \lambda_2}\alpha/\beta, \quad  
\bar{a}_2 = q^{ -(\mu_1 + \mu_2) + \lambda_0 + \lambda_1 + \lambda_2 }, \quad  
\bar{a}_3 = q^{- \mu_2 + \lambda_1 + \lambda_2}, \\
& \bar{b}_1 = q^{-(\mu_0 + \mu_1 + \mu_2) + \lambda_0 + \lambda_1 + \lambda_2 + 1}, \quad  
\bar{b}_2 = q^{- (\mu_1 + \mu_2) + \lambda_1 + \lambda_2 + 1}, \quad 
\bar{b}_3 = q^{-\mu_2 + \lambda_2 + 1}.
\end{align*}
\ref{case:3}:
\begin{align*}
g^{[4]}_1(x)  
&= (1 - q)^3 q^{-(3\mu_0 + 2\mu_1 + \mu_2) + 3\lambda_0 + 2\lambda_1 + \lambda_2}  \left( \frac{\alpha}{\beta} \right)^3 x^{ \widetilde{\mu} + \mu_1 + \mu_2 - (\lambda_0 + \lambda_1 + \lambda_2)} \\
&\quad \cdot \frac{(q, q, q, \bar{b}_1, \bar{b}_2, \bar{b}_3  ; q)_{\infty}}{(q^{-\lambda_0+ 1} , q^{-\lambda_1 + 1}, q^{-\lambda_2 + 1}, \bar{a}_1, \bar{a}_2, \bar{a}_3 ; q)_{\infty}}  \;
{}_4\phi_3 \Bigl( {\alpha/\beta , \bar{a}_1, \bar{a}_2, \bar{a}_3  
\atop{\bar{b}_1, \bar{b}_2, \bar{b}_3 }} ; q, q^{4}/(\alpha x) \Bigr) ,
\end{align*}
where
\begin{align*}
& \bar{a}_1 = q^{-\mu_0 + \lambda_0}\alpha/\beta , \quad
\bar{a}_2 = q^{-(\mu_0 + \mu_1) + \lambda_0 + \lambda_1}\alpha/\beta , \quad
\bar{a}_3 = q^{ -(\mu_0  + \mu_1 + \mu_2) + \lambda_0 + \lambda_1 + \lambda_2}\alpha/\beta , \\
& \bar{b}_1 = q^{-\mu_0 + 1 }\alpha/\beta, \quad
\bar{b}_2 = q^{-(\mu_0 + \mu_1) + \lambda_0 + 1}\alpha/\beta, \quad
\bar{b}_3 = q^{ -(\mu_0  + \mu_1 + \mu_2) + \lambda_0 + \lambda_1 + 1}\alpha/\beta .
\end{align*}
\ref{case:4} $(r = 0)$ :
\begin{align*}
g^{[4]}_1(x)  
&= (1 - q)^3 q^{\mu_0 - 2(\mu_1 - (\lambda_0 + \lambda_1)) - (\mu_2 - \lambda_2)} \; \alpha^{-\mu_0} \; x^{\mu_1 + \mu_2 - (\lambda_0 + \lambda_1 + \lambda_2)} \\
&\quad \cdot \frac{(q, q, q, q^{\mu_0 + 1}\beta/\alpha, q^{-\mu_1 + \lambda_0 + 1}, q^{ -(\mu_1 + \mu_2) + \lambda_0 + \lambda_1 + 1} ; q)_{\infty}}{(q\beta/\alpha, q^{-\lambda_1 + 1}, q^{-\lambda_2 + 1},  q^{\mu_0}, q^{ -\mu_1 + \lambda_0 + \lambda_1}, q^{ -(\mu_1 + \mu_2) + \lambda_0 + \lambda_1 + \lambda_2} ; q)_{\infty}}  \\
&\quad \cdot   {}_4\phi_3 \Bigl( {q^{\lambda_0} , q^{\mu_0}, q^{ -\mu_1 + \lambda_0 + \lambda_1}, q^{ -(\mu_1 + \mu_2) + \lambda_0 + \lambda_1 + \lambda_2}  
\atop{q^{\mu_0 + 1}\beta/\alpha, q^{-\mu_1 + \lambda_0 + 1}, q^{ -(\mu_1 + \mu_2) + \lambda_0 + \lambda_1 + 1} }} ; q, q^{4}/(\alpha x) \Bigr) .
\end{align*}
\ref{case:4} $(r = 1)$ :
\begin{align*}
g^{[4]}_1(x) 
&= (1 - q)^3 q^{\mu_0 + 2\mu_1 - \mu_2 - \lambda_0 + (\lambda_1 + \lambda_2)} \; \alpha^{-(\mu_0 + \mu_1) + \lambda_0}  \; x^{\mu_2 - (\lambda_1 + \lambda_2)} \\
&\quad \cdot \frac{(q, q, q, q^{\mu_0 + \mu_1 - \lambda_0 + 1 }\beta/\alpha, q^{\mu_1 - \lambda_0 + 1}, q^{-\mu_2 + \lambda_1 + 1} ; q)_{\infty}}{(q\beta/\alpha , q^{-\lambda_0 + 1 }, q^{-\lambda_2 + 1 }, q^{\mu_0 + \mu_1 - \lambda_0}, q^{\mu_1} , q^{-\mu_2 + \lambda_1 + \lambda_2} ; q)_{\infty}} \\
&\quad \cdot {}_4\phi_3 \Bigl( {q^{\lambda_1}, q^{\mu_0 + \mu_1 - \lambda_0}, q^{\mu_1} , q^{-\mu_2 + \lambda_1 + \lambda_2}  \atop{q^{\mu_0 + \mu_1 - \lambda_0 + 1 }\beta/\alpha, q^{\mu_1 - \lambda_0 + 1}, q^{-\mu_2 + \lambda_1 + 1} }} ; q, q^{4}/(\alpha x) \Bigr) .
\end{align*}
\ref{case:4} $(r = 2)$ :
\begin{align*}
g^{[4]}_1(x) 
&= (1 - q)^3 q^{\mu_0 + 2\mu_1 + 3\mu_2 - (\lambda_0 + 2\lambda_1)} \; \alpha^{-(\mu_0 + \mu_1 + \mu_2) + \lambda_0 + \lambda_1} \; x^{ -\lambda_2} \\
&\quad \cdot \frac{(q, q, q, q^{\mu_0 + \mu_1 + \mu_2 - (\lambda_0 + \lambda_1) + 1}\beta/\alpha , q^{\mu_1 + \mu_2 - (\lambda_0 + \lambda_1) + 1}, q^{\mu_2 - \lambda_1 + 1} ; q)_{\infty}}{(q\beta/\alpha, q^{-\lambda_0 + 1} , q^{-\lambda_1 + 1},  q^{\mu_0 + \mu_1 + \mu_2 - (\lambda_0 + \lambda_1) }, q^{\mu_1 + \mu_2 - \lambda_1 }, q^{ \mu_2} ; q)_{\infty}}\\
&\quad \cdot {}_4\phi_3 \Bigl( {q^{\lambda_2}, q^{\mu_0 + \mu_1 + \mu_2 - (\lambda_0 + \lambda_1) }, q^{\mu_1 + \mu_2 - \lambda_1 }, q^{ \mu_2} 
\atop{q^{\mu_0 + \mu_1 + \mu_2 - (\lambda_0 + \lambda_1) + 1}\beta/\alpha , q^{\mu_1 + \mu_2 - (\lambda_0 + \lambda_1) + 1}, q^{\mu_2 - \lambda_1 + 1}  }} ; q, q^{4}/(\alpha x) \Bigr) . 
\end{align*}

\subsection{$q$-hypergeometric series solutions to equation \eqref{eq:g1-o5}} $ $ \label{subsec:qhgs5} \\
\ref{case:1}:
\begin{align*}
g^{[5]}_1(x) 
&= (1 - q)^4 q^{ -\mu_0 (\lambda_0 + \lambda_1 + \lambda_2 + \lambda_3) - (\mu_1 - \lambda_0)(\lambda_1 + \lambda_2 + \lambda_3) - (\mu_2 - \lambda_1)(\lambda_2 + \lambda_3) - (\mu_3 - \lambda_2)\lambda_3 }\\  
&\quad \cdot x^{\mu_0 + \mu_1 + \mu_2  + \mu_3 - (\lambda_0 + \lambda_1 + \lambda_2 + \lambda_3) } \frac{(q, q, q, q ; q)_{\infty}}{(q^{-\lambda_0 + 1 } , q^{-\lambda_1 + 1 } , q^{-\lambda_2 + 1 } , q^{-\lambda_3 + 1 } ; q)_{\infty}} \\
&\quad \cdot \frac{(\bar{b}_1, \bar{b}_2, \bar{b}_3, \bar{b}_4 ; q)_{\infty}}{(\bar{a}_1, \bar{a}_2, \bar{a}_3, \bar{a}_4 ; q)_{\infty}} 
\; {}_5\phi_4 \Bigl( { \alpha /\beta , \bar{a}_1, \bar{a}_2, \bar{a}_3, \bar{a}_4
\atop{\bar{b}_1, \bar{b}_2, \bar{b}_3, \bar{b}_4 }}; q, q^{ -(\lambda_0 + \lambda_1 + \lambda_2 + \lambda_3) } \beta x \Bigr),
\end{align*}
where
\begin{align*}
& \bar{a}_1 = q^{\mu_0 },\quad  \bar{a}_2 = q^{\mu_0 + \mu_1 - \lambda_0 } , \\
& \bar{a}_3 = q^{\mu_0 + \mu_1 + \mu_2 - (\lambda_0 + \lambda_1) } ,\quad  \bar{a}_4 = q^{\mu_0 + \mu_1 + \mu_2 + \mu_3 - (\lambda_0 + \lambda_1 + \lambda_2)}, \\
& \bar{b}_1 = q^{\mu_0 - \lambda_0 + 1} , \quad \bar{b}_2 = q^{\mu_0 + \mu_1 - (\lambda_0 + \lambda_1) + 1 } , \\
& \bar{b}_3 = q^{\mu_0 + \mu_1 + \mu_2 - (\lambda_0 + \lambda_1 + \lambda_2) + 1 } , \quad \bar{b}_4 = q^{\mu_0 + \mu_1 + \mu_2 + \mu_3 - (\lambda_0 + \lambda_1 + \lambda_2 + \lambda_3) + 1 } .
\end{align*}
\ref{case:2} $(r = 0)$ :
\begin{align*}
g^{[5]}_1(x) 
&= (1 - q)^4 q^{-(\lambda_1 + \lambda_2 + \lambda_3)\mu_1 - (\lambda_2 + \lambda_3)(\mu_2 - \lambda_1) - \lambda_3(\mu_3 - \lambda_2)}  
\; \beta^{- \widetilde{\mu} + \lambda_0 } \; x^{\mu_1 + \mu_2 + \mu_3 - (\lambda_1 + \lambda_2 + \lambda_3)} \\
&\quad \cdot \frac{(q, q, q, q  ; q)_{\infty}}{(q\beta/\alpha , q^{-\lambda_1 + 1 } , q^{-\lambda_2 + 1 } , q^{-\lambda_3 + 1 }  ; q)_{\infty}} 
\;\frac{(\bar{b}_1, \bar{b}_2, \bar{b}_3, \bar{b}_4 ; q)_{\infty}}{(\bar{a}_1, \bar{a}_2, \bar{a}_3, \bar{a}_4 ; q)_{\infty}} \\
&\quad \cdot {}_5\phi_4 \Bigl( {q^{\lambda_0} , \bar{a}_1, \bar{a}_2, \bar{a}_3, \bar{a}_4
\atop{\bar{b}_1, \bar{b}_2, \bar{b}_3, \bar{b}_4 }} ; q, q^{ -(\lambda_0 + \lambda_1 + \lambda_2 + \lambda_3) } \beta x \Bigr) , 
\end{align*}
where
\begin{align*}
& \bar{a}_1 = q^{- \mu_0 + \lambda_0 } \alpha/\beta ,\quad  \bar{a}_2 = q^{ \mu_1 } , \quad
\bar{a}_3 = q^{\mu_1  + \mu_2 - \lambda_1}, \quad \bar{a}_4 = q^{\mu_1 + \mu_2 + \mu_3 - (\lambda_1 + \lambda_2)} , \\
& \bar{b}_1 = q^{-\mu_0 + \lambda_0 + 1}, \quad \bar{b}_2 = q^{\mu_1 - \lambda_1 + 1}, \quad
\bar{b}_3 = q^{\mu_1 + \mu_2 - (\lambda_1 + \lambda_2) + 1} , \\ 
& \bar{b}_4 = q^{\mu_1 + \mu_2 + \mu_3 - (\lambda_1 + \lambda_2 + \lambda_3) + 1} .
\end{align*}
\ref{case:2} $(r = 1)$ :
\begin{align*}
g^{[5]}_1(x) 
&= (1 - q)^4 q^{\lambda_0(\mu_1 - (\lambda_0 + \lambda_1)) - \mu_2(\lambda_2 + \lambda_3) -\lambda_3(\mu_3 - \lambda_2)}  \; \beta^{- (\widetilde{\mu} + \mu_1) + \lambda_0 + \lambda_1} \; x^{\mu_2 + \mu_3 - (\lambda_2 + \lambda_3)} \\
&\quad \cdot \frac{(q, q, q, q ; q)_{\infty}}{(q\beta/\alpha , q^{-\lambda_0 + 1 } , q^{-\lambda_2 + 1 } , q^{-\lambda_3 + 1 }  ; q)_{\infty}}  
\;\frac{(\bar{b}_1, \bar{b}_2, \bar{b}_3, \bar{b}_4 ; q)_{\infty}}{(\bar{a}_1, \bar{a}_2, \bar{a}_3, \bar{a}_4 ; q)_{\infty}} \\
&\quad \cdot {}_5\phi_4 \Bigl( {q^{\lambda_1} , \bar{a}_1, \bar{a}_2, \bar{a}_3, \bar{a}_4  
\atop{\bar{b}_1, \bar{b}_2, \bar{b}_3, \bar{b}_4}} ; q, q^{ -(\lambda_0 + \lambda_1 + \lambda_2 + \lambda_3) } \beta x \Bigr) , 
\end{align*}
where
\begin{align*}
& \bar{a}_1 = q^{- (\mu_0 + \mu_1) + \lambda_0 + \lambda_1} \alpha/\beta, \quad \bar{a}_2 = q^{- \mu_1 + \lambda_0 + \lambda_1}, \quad
 \bar{a}_3 = q^{ \mu_2 }, \quad \bar{a}_4 = q^{\mu_2 + \mu_3 - \lambda_2}, \\
& \bar{b}_1 = q^{-(\mu_0 + \mu_1) + \lambda_0 + \lambda_1 + 1}, \quad \bar{b}_2 = q^{- \mu_1 + \lambda_1 + 1}, \quad
\bar{b}_3 = q^{\mu_2 - \lambda_2 + 1}, \quad \bar{b}_4 = q^{\mu_2 + \mu_3 - (\lambda_2 + \lambda_3) + 1}.
\end{align*}
\ref{case:2} $(r = 2)$ :
\begin{align*}
g^{[5]}_1(x) 
&= (1 - q)^4 q^{\lambda_0(\mu_1 - \lambda_0) + (\lambda_0 + \lambda_1)(\mu_2 - (\lambda_1 + \lambda_2))  - \lambda_3\mu_3}  \; \beta^{- (\widetilde{\mu} + \mu_1 + \mu_2) + \lambda_0 + \lambda_1 + \lambda_2} \; x^{\mu_3 - \lambda_3} \\
&\quad \cdot \frac{(q, q, q, q ; q)_{\infty}}{(q\beta/\alpha , q^{-\lambda_0 + 1 } , q^{-\lambda_1 + 1 } , q^{-\lambda_3 + 1 }; q)_{\infty}}  
\;\frac{(\bar{b}_1, \bar{b}_2, \bar{b}_3, \bar{b}_4 ; q)_{\infty}}{(\bar{a}_1, \bar{a}_2, \bar{a}_3, \bar{a}_4 ; q)_{\infty}}\\
&\quad \cdot {}_5\phi_4 \Bigl( {q^{\lambda_2} , \bar{a}_1, \bar{a}_2, \bar{a}_3, \bar{a}_4 
\atop{\bar{b}_1, \bar{b}_2, \bar{b}_3, \bar{b}_4  }} ; q, q^{ -(\lambda_0 + \lambda_1 + \lambda_2 + \lambda_3) } \beta x \Bigr) , 
\end{align*}
where
\begin{align*}
& \bar{a}_1 = q^{-(\mu_0 + \mu_1 + \mu_2) + \lambda_0 + \lambda_1 + \lambda_2}\alpha/\beta , \quad \bar{a}_2 = q^{-(\mu_1 + \mu_2) + \lambda_0 + \lambda_1 + \lambda_2}, \\
& \bar{a}_3 = q^{- \mu_2  + \lambda_1 + \lambda_2} , \quad \bar{a}_4 = q^{ \mu_3 } , \\
& \bar{b}_1 = q^{-(\mu_0 + \mu_1 + \mu_2) + \lambda_0 + \lambda_1 + \lambda_2 + 1}, \quad \bar{b}_2 = q^{-(\mu_1 + \mu_2) + \lambda_1 + \lambda_2 + 1}, \\
& \bar{b}_3 = q^{-\mu_2 + \lambda_2 + 1}, \quad \bar{b}_4 = q^{\mu_3 - \lambda_3 + 1} .
\end{align*}
\ref{case:2} $(r = 3)$ :
\begin{align*}
g^{[5]}_1(x) 
&= (1 - q)^4 q^{\lambda_0(\mu_1 - \lambda_0) + (\lambda_0 + \lambda_1)(\mu_2 - \lambda_1) + (\lambda_0 + \lambda_1 + \lambda_2)(\mu_3 - (\lambda_2 + \lambda_3))} \\
& \quad \cdot \beta^{- (\widetilde{\mu} + \mu_1 + \mu_2 + \mu_3) + \lambda_0 + \lambda_1 + \lambda_2 + \lambda_3} 
\frac{(q, q, q, q ; q)_{\infty}}{(q\beta/\alpha , q^{-\lambda_0 + 1 } , q^{-\lambda_1 + 1 } , q^{-\lambda_2 + 1 } ; q)_{\infty}} \\ 
&\quad \cdot \frac{(\bar{b}_1, \bar{b}_2, \bar{b}_3, \bar{b}_4 ; q)_{\infty}}{(\bar{a}_1, \bar{a}_2, \bar{a}_3, \bar{a}_4 ; q)_{\infty}} 
\; {}_5\phi_4 \Bigl( {q^{\lambda_3} , \bar{a}_1, \bar{a}_2, \bar{a}_3, \bar{a}_4   
\atop{\bar{b}_1, \bar{b}_2, \bar{b}_3, \bar{b}_4  }} ; q, q^{ -(\lambda_0 + \lambda_1 + \lambda_2 + \lambda_3) } \beta x \Bigr) , 
\end{align*}
where
\begin{align*}
& \bar{a}_1 = q^{- (\mu_0 + \mu_1 + \mu_2 + \mu_3) + \lambda_0 + \lambda_1 + \lambda_2 + \lambda_3}\alpha/\beta , \quad
\bar{a}_2 = q^{ -(\mu_1 + \mu_2 + \mu_3) + \lambda_0 + \lambda_1 + \lambda_2 + \lambda_3 }, \\
& \bar{a}_3 = q^{ -(\mu_2 + \mu_3) + \lambda_1 + \lambda_2 + \lambda_3 }, \quad \bar{a}_4 = q^{- \mu_3 + \lambda_2 + \lambda_3}, \\
& \bar{b}_1 = q^{-(\mu_0 + \mu_1 + \mu_2 + \mu_3) + \lambda_0 + \lambda_1 + \lambda_2 + \lambda_3 + 1} , \quad
 \bar{b}_2 = q^{- (\mu_1 + \mu_2 + \mu_3) + \lambda_1 + \lambda_2 + \lambda_3 + 1}, \\
& \bar{b}_3 = q^{-(\mu_2 + \mu_3) + \lambda_2 + \lambda_3 + 1} , \quad \bar{b}_4 = q^{-\mu_3 + \lambda_3 + 1}.
\end{align*}
\ref{case:3}:
\begin{align*}
g^{[5]}_1(x)  
&= (1 - q)^4 q^{-(4\mu_0 + 3\mu_1 + 2\mu_2 + \mu_3) + 4\lambda_0 + 3\lambda_1 + 2\lambda_2 + \lambda_3}  \left( \frac{\alpha}{\beta} \right)^4 x^{ \widetilde{\mu} + \mu_1 + \mu_2 + \mu_3 - (\lambda_0 + \lambda_1 + \lambda_2 + \lambda_3)} \\
&\quad \frac{(q, q, q, q ; q)_{\infty}}{(q^{-\lambda_0+ 1} , q^{-\lambda_1 + 1}, q^{-\lambda_2 + 1}, q^{-\lambda_3 + 1} ; q)_{\infty}}  
\;\frac{(\bar{b}_1, \bar{b}_2, \bar{b}_3, \bar{b}_4 ; q)_{\infty}}{(\bar{a}_1, \bar{a}_2, \bar{a}_3, \bar{a}_4 ; q)_{\infty}} \\
&\quad \cdot  \;  {}_5\phi_4 \Bigl( {\alpha/\beta , \bar{a}_1, \bar{a}_2, \bar{a}_3, \bar{a}_4
\atop{ \bar{b}_1, \bar{b}_2, \bar{b}_3, \bar{b}_4}} ; q, q^{5}/(\alpha x) \Bigr) ,
\end{align*}
where
\begin{align*}
& \bar{a}_1 = q^{-\mu_0 + \lambda_0} \alpha/\beta, \quad \bar{a}_2 = q^{ -(\mu_0  + \mu_1) + \lambda_0 + \lambda_1 }\alpha/\beta, \\
& \bar{a}_3 = q^{ -(\mu_0  + \mu_1 + \mu_2) + \lambda_0 + \lambda_1 + \lambda_2}\alpha/\beta , \quad \bar{a}_4 = q^{ -(\mu_0  + \mu_1 + \mu_2 + \mu_3) + \lambda_0 + \lambda_1 + \lambda_2 + \lambda_3}\alpha/\beta, \\
& \bar{b}_1 = q^{-\mu_0 + 1} \alpha/\beta, \quad \bar{b}_2 = q^{ -(\mu_0 + \mu_1) + \lambda_0 + 1}\alpha/\beta, \\
& \bar{b}_3 = q^{ -(\mu_0  + \mu_1 + \mu_2) + \lambda_0 + \lambda_1 + 1}\alpha/\beta, \quad \bar{b}_4 = q^{ -(\mu_0  + \mu_1 + \mu_2 + \mu_3) + \lambda_0 + \lambda_1 + \lambda_2 + 1}\alpha/\beta .
\end{align*}
\ref{case:4} $(r = 0)$ :
\begin{align*}
g^{[5]}_1(x)  
&= (1 - q)^4 q^{\mu_0 - 3(\mu_1 - (\lambda_0 + \lambda_1)) - 2(\mu_2 - \lambda_2) - (\mu_3 - \lambda_3)} \; \alpha^{-\mu_0} \; x^{\mu_1 + \mu_2 + \mu_3 - (\lambda_0 + \lambda_1 + \lambda_2 + \lambda_3)} \\
&\quad \frac{(q, q, q, q ; q)_{\infty}}{(q\beta/\alpha , q^{-\lambda_1 + 1}, q^{-\lambda_2 + 1}, q^{-\lambda_3 + 1} ; q)_{\infty}}  
\;\frac{(\bar{b}_1, \bar{b}_2, \bar{b}_3, \bar{b}_4 ; q)_{\infty}}{(\bar{a}_1, \bar{a}_2, \bar{a}_3, \bar{a}_4 ; q)_{\infty}} \\
&\quad \cdot  \;  {}_5\phi_4 \Bigl( {q^{\lambda_0} , \bar{a}_1, \bar{a}_2, \bar{a}_3, \bar{a}_4 
\atop{ \bar{b}_1, \bar{b}_2, \bar{b}_3, \bar{b}_4}} ; q, q^{5}/(\alpha x) \Bigr) ,
\end{align*}
where
\begin{align*}
& \bar{a}_1 = q^{\mu_0}, \quad \bar{a}_2 = q^{ -\mu_1 + \lambda_0 + \lambda_1}, \quad
\bar{a}_3 = q^{ -(\mu_1 + \mu_2) + \lambda_0 + \lambda_1 + \lambda_2}, \\
& \bar{a}_4 = q^{ -(\mu_1 + \mu_2 + \mu_3) + \lambda_0 + \lambda_1 + \lambda_2 + \lambda_3}, \\
& \bar{b}_1 = q^{\mu_0 + 1}\beta/\alpha, \quad \bar{b}_2 = q^{-\mu_1 + \lambda_0 + 1}, \quad
\bar{b}_3 = q^{ -(\mu_1 + \mu_2) + \lambda_0 + \lambda_1 + 1} , \\
& \bar{b}_4 = q^{ -(\mu_1 + \mu_2 + \mu_3) + \lambda_0 + \lambda_1 + \lambda_2 + 1}.
\end{align*}
\ref{case:4} $(r = 1)$ :
\begin{align*}
g^{[5]}_1(x)  
&= (1 - q)^4 q^{\mu_0 + 2(\mu_1 - \mu_2) - \mu_3 - \lambda_0 + 2(\lambda_1 + \lambda_2) + \lambda_3}  \; \alpha^{-(\mu_0 + \mu_1) + \lambda_0} \; x^{\mu_2 + \mu_3 - (\lambda_1 + \lambda_2 + \lambda_3)} \\
&\quad \frac{(q, q, q, q ; q)_{\infty}}{(q\beta/\alpha, q^{-\lambda_0 + 1}, q^{-\lambda_2 + 1}, q^{-\lambda_3 + 1} ; q)_{\infty}}  
\;\frac{(\bar{b}_1, \bar{b}_2, \bar{b}_3, \bar{b}_4 ; q)_{\infty}}{(\bar{a}_1, \bar{a}_2, \bar{a}_3, \bar{a}_4 ; q)_{\infty}} \\
&\quad \cdot  \;  {}_5\phi_4 \Bigl( {q^{\lambda_1} , \bar{a}_1, \bar{a}_2, \bar{a}_3, \bar{a}_4
\atop{\bar{b}_1, \bar{b}_2, \bar{b}_3, \bar{b}_4 }} ; q, q^{5}/(\alpha x) \Bigr) ,
\end{align*}
where
\begin{align*}
& \bar{a}_1 = q^{\mu_0 + \mu_1 - \lambda_0}, \quad \bar{a}_2 = q^{\mu_1}, \quad
 \bar{a}_3 = q^{ -\mu_2 + \lambda_1 + \lambda_2}, \quad \bar{a}_4 = q^{ -(\mu_2 + \mu_3) + \lambda_1 + \lambda_2 + \lambda_3}, \\
& \bar{b}_1 = q^{\mu_0 + \mu_1 - \lambda_0 + 1}\beta/\alpha, \quad \bar{b}_2 = q^{\mu_1 - \lambda_0 + 1}, \quad
\bar{b}_3 = q^{ -\mu_2 + \lambda_1 + 1} , \quad \bar{b}_4 = q^{ -(\mu_2 + \mu_3) + \lambda_1 + \lambda_2 + 1}.
\end{align*}
\ref{case:4} $(r = 2)$ :
\begin{align*}
g^{[5]}_1(x) 
&= (1 - q)^4 q^{\mu_0 + 2\mu_1 + 3\mu_2 - \mu_3 - (\lambda_0 + 2\lambda_1) + \lambda_2 + \lambda_3} \; \alpha^{-(\mu_0 + \mu_1 + \mu_2) + \lambda_0 + \lambda_1} \; x^{\mu_3 - (\lambda_2 + \lambda_3)} \\
&\quad \cdot \frac{(q, q, q, q ; q)_{\infty}}{(q\beta/\alpha , q^{-\lambda_0 + 1 }, q^{-\lambda_1 + 1 }, q^{-\lambda_3 + 1 } ; q)_{\infty}} 
\;\frac{(\bar{b}_1, \bar{b}_2, \bar{b}_3, \bar{b}_4 ; q)_{\infty}}{(\bar{a}_1, \bar{a}_2, \bar{a}_3, \bar{a}_4 ; q)_{\infty}} \\
&\quad \cdot {}_5\phi_4 \Bigl( {q^{\lambda_2}, \bar{a}_1, \bar{a}_2, \bar{a}_3, \bar{a}_4 
\atop{\bar{b}_1, \bar{b}_2, \bar{b}_3, \bar{b}_4 }} ; q, q^{5}/(\alpha x) \Bigr) , 
\end{align*}
where
\begin{align*}
& \bar{a}_1 = q^{\mu_0 + \mu_1 + \mu_2 - (\lambda_0 + \lambda_1)}, \quad \bar{a}_2 = q^{\mu_1 + \mu_2 - \lambda_1}, \quad
\bar{a}_3 = q^{\mu_2} , \quad \bar{a}_4 = q^{-\mu_3 + \lambda_2 + \lambda_3}, \\
& \bar{b}_1 = q^{\mu_0 + \mu_1 + \mu_2 - (\lambda_0 + \lambda_1) + 1 }\beta/\alpha, \quad \bar{b}_2 = q^{\mu_1 + \mu_2 - (\lambda_0 + \lambda_1) + 1}, \quad
\bar{b}_3 = q^{\mu_2 - \lambda_1 + 1} , \\
& \bar{b}_4 = q^{-\mu_3 + \lambda_2 + 1}.
\end{align*}
\ref{case:4} $(r = 3)$ :
\begin{align*}
g^{[5]}_1(x) 
&= (1 - q)^4 q^{\mu_0 + 2\mu_1 + 3\mu_2 + 4\mu_3 - (\lambda_0 + 2\lambda_1 + 3\lambda_2)} \; \alpha^{-(\mu_0 + \mu_1 + \mu_2 + \mu_3) + \lambda_0 + \lambda_1 + \lambda_2} \; x^{ -\lambda_3} \\
&\quad \cdot \frac{(q, q, q, q ; q)_{\infty}}{(q\beta/\alpha , q^{-\lambda_0  + 1} , q^{-\lambda_1 + 1}, q^{-\lambda_2 + 1} ; q)_{\infty}}
\;\frac{(\bar{b}_1, \bar{b}_2, \bar{b}_3, \bar{b}_4 ; q)_{\infty}}{(\bar{a}_1, \bar{a}_2, \bar{a}_3, \bar{a}_4 ; q)_{\infty}} \\
&\quad \cdot {}_5\phi_4 \Bigl( {q^{\lambda_3}, \bar{a}_1, \bar{a}_2, \bar{a}_3, \bar{a}_4 
\atop{\bar{b}_1, \bar{b}_2, \bar{b}_3, \bar{b}_4 }} ; q, q^{5}/(\alpha x) \Bigr) , 
\end{align*}
where
\begin{align*}
& \bar{a}_1 = q^{\mu_0 + \mu_1 + \mu_2 + \mu_3 - (\lambda_0 + \lambda_1 + \lambda_2) }, \quad \bar{a}_2 = q^{\mu_1 + \mu_2 + \mu_3 - (\lambda_1 + \lambda_2) }, \\
& \bar{a}_3 = q^{\mu_2 + \mu_3 - \lambda_2}, \quad \bar{a}_4 = q^{ \mu_3}, \\
& \bar{b}_1 = q^{\mu_0 + \mu_1 + \mu_2 + \mu_3 - (\lambda_0 + \lambda_1 + \lambda_2) + 1}\beta/\alpha , \quad \bar{b}_2 = q^{\mu_1 + \mu_2 + \mu_3 - (\lambda_0 + \lambda_1 + \lambda_2) + 1}, \\
& \bar{b}_3 = q^{\mu_2 + \mu_3 - (\lambda_1 + \lambda_2) + 1}, \quad \bar{b}_4 = q^{\mu_3 - \lambda_2 + 1} .
\end{align*}

\subsection{$q$-hypergeometric series solutions to equation \eqref{eq:g1-o6}} $ $ \label{subsec:qhgs6} \\
\ref{case:1}:
\begin{align*}
&g^{[6]}_1(x) = (1 - q)^5 \\
&\quad \cdot q^{ -\mu_0 (\lambda_0 + \lambda_1 + \lambda_2 + \lambda_3 + \lambda_4) - (\mu_1 - \lambda_0)(\lambda_1 + \lambda_2 + \lambda_3 + \lambda_4) - (\mu_2 - \lambda_1)(\lambda_2 + \lambda_3 + \lambda_4) - (\mu_3 - \lambda_2)(\lambda_3 + \lambda_4) - (\mu_4 - \lambda_3)\lambda_4 }\\  
&\quad \cdot x^{\mu_0 + \mu_1 + \mu_2  + \mu_3  + \mu_4 - (\lambda_0 + \lambda_1 + \lambda_2 + \lambda_3 + \lambda_4) } \\
&\quad \cdot \frac{(q, q, q, q, q ; q)_{\infty}}{(q^{-\lambda_0 + 1 } , q^{-\lambda_1 + 1 } , q^{-\lambda_2 + 1 } , q^{-\lambda_3 + 1 } , q^{-\lambda_4 + 1 }; q)_{\infty}} 
\; \frac{(\bar{b}_1, \bar{b}_2, \bar{b}_3, \bar{b}_4, \bar{b}_5 ; q)_{\infty}}{(\bar{a}_1, \bar{a}_2, \bar{a}_3, \bar{a}_4, \bar{a}_5 ; q)_{\infty}} \\
&\quad \cdot  {}_6\phi_5 \Bigl( { \alpha /\beta , \bar{a}_1, \bar{a}_2, \bar{a}_3, \bar{a}_4, \bar{a}_5
\atop{\bar{b}_1, \bar{b}_2, \bar{b}_3, \bar{b}_4, \bar{b}_5 }}; q, q^{ -(\lambda_0 + \lambda_1 + \lambda_2 + \lambda_3 + \lambda_4) } \beta x \Bigr),
\end{align*}
where
\begin{align*}
& \bar{a}_1 = q^{\mu_0 },\quad  \bar{a}_2 = q^{\mu_0 + \mu_1 - \lambda_0 } , \quad \bar{a}_3 = q^{\mu_0 + \mu_1 + \mu_2 - (\lambda_0 + \lambda_1) }, \\
& \bar{a}_4 = q^{\mu_0 + \mu_1 + \mu_2 + \mu_3 - (\lambda_0 + \lambda_1 + \lambda_2)}, \quad  \bar{a}_5 = q^{\mu_0 + \mu_1 + \mu_2 + \mu_3 + \mu_4 - (\lambda_0 + \lambda_1 + \lambda_2 + \lambda_3)}, \\
& \bar{b}_1 = q^{\mu_0 - \lambda_0 + 1} , \quad \bar{b}_2 = q^{\mu_0 + \mu_1 - (\lambda_0 + \lambda_1) + 1 } , \\
& \bar{b}_3 = q^{\mu_0 + \mu_1 + \mu_2 - (\lambda_0 + \lambda_1 + \lambda_2) + 1 } , \quad \bar{b}_4 = q^{\mu_0 + \mu_1 + \mu_2 + \mu_3 - (\lambda_0 + \lambda_1 + \lambda_2 + \lambda_3) + 1 } , \\
& \bar{b}_5 = q^{\mu_0 + \mu_1 + \mu_2 + \mu_3 + \mu_4 - (\lambda_0 + \lambda_1 + \lambda_2 + \lambda_3 + \lambda_4) + 1 } .
\end{align*}
\ref{case:2} $(r = 0)$ :
\begin{align*}
g^{[6]}_1(x) 
&= (1 - q)^5 q^{-(\lambda_1 + \lambda_2 + \lambda_3 + \lambda_4)\mu_1 - (\lambda_2 + \lambda_3 + \lambda_4)(\mu_2 - \lambda_1) - (\lambda_3 + \lambda_4)(\mu_3 - \lambda_2) - \lambda_4(\mu_4 - \lambda_3)} \\ 
&\quad \cdot \beta^{- \widetilde{\mu} + \lambda_0 } \; x^{\mu_1 + \mu_2 + \mu_3 + \mu_4 - (\lambda_1 + \lambda_2 + \lambda_3 + \lambda_4)} \\
&\quad \cdot \frac{(q, q, q, q, q  ; q)_{\infty}}{(q\beta/\alpha , q^{-\lambda_1 + 1 } , q^{-\lambda_2 + 1 } , q^{-\lambda_3 + 1 } , q^{-\lambda_4 + 1 } ; q)_{\infty}} 
\;\frac{(\bar{b}_1, \bar{b}_2, \bar{b}_3, \bar{b}_4, \bar{b}_5 ; q)_{\infty}}{(\bar{a}_1, \bar{a}_2, \bar{a}_3, \bar{a}_4, \bar{a}_5 ; q)_{\infty}} \\
&\quad \cdot {}_6\phi_5 \Bigl( {q^{\lambda_0} , \bar{a}_1, \bar{a}_2, \bar{a}_3, \bar{a}_4, \bar{a}_5
\atop{\bar{b}_1, \bar{b}_2, \bar{b}_3, \bar{b}_4, \bar{b}_5 }} ; q, q^{ -(\lambda_0 + \lambda_1 + \lambda_2 + \lambda_3 + \lambda_4) } \beta x \Bigr) , 
\end{align*}
where 
\begin{align*}
& \bar{a}_1 = q^{- \mu_0 + \lambda_0 } \alpha/\beta ,\quad  \bar{a}_2 = q^{ \mu_1 } , \quad \bar{a}_3 = q^{\mu_1  + \mu_2 - \lambda_1}, \\
& \bar{a}_4 = q^{\mu_1 + \mu_2 + \mu_3 - (\lambda_1 + \lambda_2)} , \quad \bar{a}_5 = q^{\mu_1 + \mu_2 + \mu_3 + \mu_4 - (\lambda_1 + \lambda_2 + \lambda_3)} , \\
& \bar{b}_1 = q^{-\mu_0 + \lambda_0 + 1}, \quad \bar{b}_2 = q^{\mu_1 - \lambda_1 + 1}, \quad \bar{b}_3 = q^{\mu_1 + \mu_2 - (\lambda_1 + \lambda_2) + 1} ,\\
& \bar{b}_4 = q^{\mu_1 + \mu_2 + \mu_3 - (\lambda_1 + \lambda_2 + \lambda_3) + 1} , \quad \bar{b}_5 = q^{\mu_1 + \mu_2 + \mu_3 + \mu_4 - (\lambda_1 + \lambda_2 + \lambda_3 + \lambda_4) + 1} .
\end{align*}
\ref{case:2} $(r = 1)$ :
\begin{align*}
g^{[6]}_1(x) 
&= (1 - q)^5 q^{\lambda_0(\mu_1 - (\lambda_0 + \lambda_1)) - \mu_2(\lambda_2 + \lambda_3 + \lambda_4) - (\lambda_3 + \lambda_4)(\mu_3 - \lambda_2) - \lambda_4(\mu_4 - \lambda_3)}  \\
&\quad \cdot \beta^{- (\widetilde{\mu} + \mu_1) + \lambda_0 + \lambda_1} \; x^{\mu_2 + \mu_3 + \mu_4 - (\lambda_2 + \lambda_3 + \lambda_4)} \\
&\quad \cdot \frac{(q, q, q, q, q ; q)_{\infty}}{(q\beta/\alpha , q^{-\lambda_0 + 1 } , q^{-\lambda_2 + 1 } , q^{-\lambda_3 + 1 } , q^{-\lambda_4 + 1 }  ; q)_{\infty}}  
\;\frac{(\bar{b}_1, \bar{b}_2, \bar{b}_3, \bar{b}_4, \bar{b}_5 ; q)_{\infty}}{(\bar{a}_1, \bar{a}_2, \bar{a}_3, \bar{a}_4, \bar{a}_5 ; q)_{\infty}} \\
&\quad \cdot {}_6\phi_5 \Bigl( {q^{\lambda_1} , \bar{a}_1, \bar{a}_2, \bar{a}_3, \bar{a}_4, \bar{a}_5  
\atop{\bar{b}_1, \bar{b}_2, \bar{b}_3, \bar{b}_4, \bar{b}_5}} ; q, q^{ -(\lambda_0 + \lambda_1 + \lambda_2 + \lambda_3 + \lambda_4) } \beta x \Bigr) , 
\end{align*}
where
\begin{align*}
& \bar{a}_1 = q^{- (\mu_0 + \mu_1) + \lambda_0 + \lambda_1} \alpha/\beta, \quad \bar{a}_2 = q^{- \mu_1 + \lambda_0 + \lambda_1}, \quad \bar{a}_3 = q^{ \mu_2 }, \\
& \bar{a}_4 = q^{\mu_2 + \mu_3 - \lambda_2}, \quad \bar{a}_5 = q^{\mu_2 + \mu_3 + \mu_4 - (\lambda_2 + \lambda_3)}, \\
& \bar{b}_1 = q^{-(\mu_0 + \mu_1) + \lambda_0 + \lambda_1 + 1}, \quad \bar{b}_2 = q^{- \mu_1 + \lambda_1 + 1}, \quad \bar{b}_3 = q^{\mu_2 - \lambda_2 + 1}, \\
& \bar{b}_4 = q^{\mu_2 + \mu_3 - (\lambda_2 + \lambda_3) + 1}, \quad \bar{b}_5 = q^{\mu_2 + \mu_3 + \mu_4 - (\lambda_2 + \lambda_3 + \lambda_4) + 1} .
\end{align*}
\ref{case:2} $(r = 2)$ :
\begin{align*}
g^{[6]}_1(x) 
&= (1 - q)^5 q^{\lambda_0(\mu_1 - \lambda_0) + (\lambda_0 + \lambda_1)(\mu_2 - (\lambda_1 + \lambda_2))  - (\lambda_3 + \lambda_4)\mu_3 -\lambda_4(\mu_4 - \lambda_3)} \\
&\quad \cdot \beta^{- (\widetilde{\mu} + \mu_1 + \mu_2) + \lambda_0 + \lambda_1 + \lambda_2} \; x^{\mu_3 + \mu_4 - (\lambda_3 + \lambda_4)} \\
&\quad \cdot \frac{(q, q, q, q, q ; q)_{\infty}}{(q\beta/\alpha , q^{-\lambda_0 + 1 } , q^{-\lambda_1 + 1 } , q^{-\lambda_3 + 1 }, q^{-\lambda_4 + 1 }; q)_{\infty}}  
\;\frac{(\bar{b}_1, \bar{b}_2, \bar{b}_3, \bar{b}_4, \bar{b}_5 ; q)_{\infty}}{(\bar{a}_1, \bar{a}_2, \bar{a}_3, \bar{a}_4, \bar{a}_5 ; q)_{\infty}}\\
&\quad \cdot {}_6\phi_5 \Bigl( {q^{\lambda_2} , \bar{a}_1, \bar{a}_2, \bar{a}_3, \bar{a}_4, \bar{a}_5 
\atop{\bar{b}_1, \bar{b}_2, \bar{b}_3, \bar{b}_4, \bar{b}_5  }} ; q, q^{ -(\lambda_0 + \lambda_1 + \lambda_2 + \lambda_3 + \lambda_4) } \beta x \Bigr) , 
\end{align*}
where
\begin{align*}
& \bar{a}_1 = q^{-(\mu_0 + \mu_1 + \mu_2) + \lambda_0 + \lambda_1 + \lambda_2}\alpha/\beta , \quad \bar{a}_2 = q^{-(\mu_1 + \mu_2) + \lambda_0 + \lambda_1 + \lambda_2}, \quad \bar{a}_3 = q^{- \mu_2  + \lambda_1 + \lambda_2} , \\
& \bar{a}_4 = q^{ \mu_3 } , \quad \bar{a}_5 = q^{ \mu_3 + \mu_4 - \lambda_3} ,\\
& \bar{b}_1 = q^{-(\mu_0 + \mu_1 + \mu_2) + \lambda_0 + \lambda_1 + \lambda_2 + 1}, \quad \bar{b}_2 = q^{-(\mu_1 + \mu_2) + \lambda_1 + \lambda_2 + 1}, \quad \bar{b}_3 = q^{-\mu_2 + \lambda_2 + 1},\\
& \bar{b}_4 = q^{\mu_3 - \lambda_3 + 1}, \quad \bar{b}_5 = q^{\mu_3 + \mu_4 - (\lambda_3 + \lambda_4) + 1}.
\end{align*}
\ref{case:2} $(r = 3)$ :
\begin{align*}
g^{[6]}_1(x) 
&= (1 - q)^5 q^{\lambda_0(\mu_1 - \lambda_0) + (\lambda_0 + \lambda_1)(\mu_2 - \lambda_1) + (\lambda_0 + \lambda_1 + \lambda_2)(\mu_3 - (\lambda_2 + \lambda_3)) - \lambda_4\mu_4} \\
& \quad \cdot \beta^{- (\widetilde{\mu} + \mu_1 + \mu_2 + \mu_3) + \lambda_0 + \lambda_1 + \lambda_2 + \lambda_3} \; x^{\mu_4 - \lambda_4} \\
&\quad \cdot \frac{(q, q, q, q, q ; q)_{\infty}}{(q\beta/\alpha , q^{-\lambda_0 + 1 } , q^{-\lambda_1 + 1 } , q^{-\lambda_2 + 1 } , q^{-\lambda_4 + 1 } ; q)_{\infty}}  
\;\frac{(\bar{b}_1, \bar{b}_2, \bar{b}_3, \bar{b}_4, \bar{b}_5 ; q)_{\infty}}{(\bar{a}_1, \bar{a}_2, \bar{a}_3, \bar{a}_4, \bar{a}_5 ; q)_{\infty}} \\
&\quad \cdot {}_6\phi_5 \Bigl( {q^{\lambda_3} , \bar{a}_1, \bar{a}_2, \bar{a}_3, \bar{a}_4, \bar{a}_5   
\atop{\bar{b}_1, \bar{b}_2, \bar{b}_3, \bar{b}_4, \bar{b}_5 }} ; q, q^{ -(\lambda_0 + \lambda_1 + \lambda_2 + \lambda_3 + \lambda_4) } \beta x \Bigr) , 
\end{align*}
where
\begin{align*}
& \bar{a}_1 = q^{- (\mu_0 + \mu_1 + \mu_2 + \mu_3) + \lambda_0 + \lambda_1 + \lambda_2 + \lambda_3}\alpha/\beta , \quad
\bar{a}_2 = q^{ -(\mu_1 + \mu_2 + \mu_3) + \lambda_0 + \lambda_1 + \lambda_2 + \lambda_3 }, \\
& \bar{a}_3 = q^{ -(\mu_2 + \mu_3) + \lambda_1 + \lambda_2 + \lambda_3 }, \quad \bar{a}_4 = q^{- \mu_3 + \lambda_2 + \lambda_3}, \quad \bar{a}_5 = q^{\mu_4}, \\
& \bar{b}_1 = q^{-(\mu_0 + \mu_1 + \mu_2 + \mu_3) + \lambda_0 + \lambda_1 + \lambda_2 + \lambda_3 + 1} , \quad
 \bar{b}_2 = q^{- (\mu_1 + \mu_2 + \mu_3) + \lambda_1 + \lambda_2 + \lambda_3 + 1}, \\
& \bar{b}_3 = q^{-(\mu_2 + \mu_3) + \lambda_2 + \lambda_3 + 1} , \quad \bar{b}_4 = q^{-\mu_3 + \lambda_3 + 1}, \quad \bar{b}_5 = q^{\mu_4 - \lambda_4 + 1} .
\end{align*}
\ref{case:2} $(r = 4)$ :
\begin{align*}
g^{[6]}_1(x) 
&= (1 - q)^5 q^{\lambda_0(\mu_1 - \lambda_0) + (\lambda_0 + \lambda_1)(\mu_2 - \lambda_1) + (\lambda_0 + \lambda_1 + \lambda_2)(\mu_3 - \lambda_2) + (\lambda_0 + \lambda_1 + \lambda_2 + \lambda_3)(\mu_4 - (\lambda_3 + \lambda_4))} \\
& \quad \cdot \beta^{- (\widetilde{\mu} + \mu_1 + \mu_2 + \mu_3 + \mu_4) + \lambda_0 + \lambda_1 + \lambda_2 + \lambda_3 + \lambda_4} \\
&\quad \cdot \frac{(q, q, q, q, q ; q)_{\infty}}{(q\beta/\alpha , q^{-\lambda_0 + 1 } , q^{-\lambda_1 + 1 } , q^{-\lambda_2 + 1 } , q^{-\lambda_3 + 1 } ; q)_{\infty}}  
\;\frac{(\bar{b}_1, \bar{b}_2, \bar{b}_3, \bar{b}_4, \bar{b}_5 ; q)_{\infty}}{(\bar{a}_1, \bar{a}_2, \bar{a}_3, \bar{a}_4, \bar{a}_5 ; q)_{\infty}} \\
&\quad \cdot {}_6\phi_5 \Bigl( {q^{\lambda_4} , \bar{a}_1, \bar{a}_2, \bar{a}_3, \bar{a}_4, \bar{a}_5   
\atop{\bar{b}_1, \bar{b}_2, \bar{b}_3, \bar{b}_4, \bar{b}_5 }} ; q, q^{ -(\lambda_0 + \lambda_1 + \lambda_2 + \lambda_3 + \lambda_4) } \beta x \Bigr) , 
\end{align*}
where
\begin{align*}
& \bar{a}_1 = q^{- (\mu_0 + \mu_1 + \mu_2 + \mu_3 + \mu_4) + \lambda_0 + \lambda_1 + \lambda_2 + \lambda_3 + \lambda_4}\alpha/\beta , \quad
\bar{a}_2 = q^{ -(\mu_1 + \mu_2 + \mu_3 + \mu_4) + \lambda_0 + \lambda_1 + \lambda_2 + \lambda_3 + \lambda_4}, \\
& \bar{a}_3 = q^{ -(\mu_2 + \mu_3 + \mu_4) + \lambda_1 + \lambda_2 + \lambda_3 + \lambda_4}, \quad \bar{a}_4 = q^{- (\mu_3 + \mu_4) + \lambda_2 + \lambda_3 + \lambda_4}, \quad \bar{a}_5 = q^{-\mu_4 + \lambda_3 + \lambda_4}, \\
& \bar{b}_1 = q^{-(\mu_0 + \mu_1 + \mu_2 + \mu_3 + \mu_4) + \lambda_0 + \lambda_1 + \lambda_2 + \lambda_3 + \lambda_4 + 1} , \quad
 \bar{b}_2 = q^{- (\mu_1 + \mu_2 + \mu_3 + \mu_4) + \lambda_1 + \lambda_2 + \lambda_3 + \lambda_4 + 1}, \\
& \bar{b}_3 = q^{-(\mu_2 + \mu_3 + \mu_4) + \lambda_2 + \lambda_3 + \lambda_4 + 1} , \quad \bar{b}_4 = q^{-(\mu_3 + \mu_4) + \lambda_3 + \lambda_4 + 1}, \quad \bar{b}_5 = q^{-\mu_4 + \lambda_4 + 1} .
\end{align*}
\ref{case:3}:
\begin{align*}
g^{[6]}_1(x)  
&= (1 - q)^5 q^{-(5\mu_0 + 4\mu_1 + 3\mu_2 + 2\mu_3 + \mu_4) + 5\lambda_0 + 4\lambda_1 + 3\lambda_2 + 2\lambda_3 + \lambda_4}  \left( \frac{\alpha}{\beta} \right)^5 \\
&\quad \cdot x^{ \widetilde{\mu} + \mu_1 + \mu_2 + \mu_3 + \mu_4 - (\lambda_0 + \lambda_1 + \lambda_2 + \lambda_3 + \lambda_4)} \\
&\quad \frac{(q, q, q, q, q ; q)_{\infty}}{(q^{-\lambda_0+ 1} , q^{-\lambda_1 + 1}, q^{-\lambda_2 + 1}, q^{-\lambda_3 + 1}, q^{-\lambda_4 + 1} ; q)_{\infty}}  
\;\frac{(\bar{b}_1, \bar{b}_2, \bar{b}_3, \bar{b}_4, \bar{b}_5 ; q)_{\infty}}{(\bar{a}_1, \bar{a}_2, \bar{a}_3, \bar{a}_4, \bar{a}_5 ; q)_{\infty}} \\
&\quad \cdot  \;  {}_6\phi_5 \Bigl( {\alpha/\beta , \bar{a}_1, \bar{a}_2, \bar{a}_3, \bar{a}_4, \bar{a}_5
\atop{ \bar{b}_1, \bar{b}_2, \bar{b}_3, \bar{b}_4, \bar{b}_5 }} ; q, q^{6}/(\alpha x) \Bigr) ,
\end{align*}
where
\begin{align*}
& \bar{a}_1 = q^{-\mu_0 + \lambda_0} \alpha/\beta, \quad \bar{a}_2 = q^{ -(\mu_0  + \mu_1) + \lambda_0 + \lambda_1 }\alpha/\beta, \\
& \bar{a}_3 = q^{ -(\mu_0  + \mu_1 + \mu_2) + \lambda_0 + \lambda_1 + \lambda_2}\alpha/\beta , \quad \bar{a}_4 = q^{ -(\mu_0  + \mu_1 + \mu_2 + \mu_3) + \lambda_0 + \lambda_1 + \lambda_2 + \lambda_3}\alpha/\beta, \\
& \bar{a}_5 = q^{ -(\mu_0  + \mu_1 + \mu_2 + \mu_3 + \mu_4) + \lambda_0 + \lambda_1 + \lambda_2 + \lambda_3 + \lambda_4}\alpha/\beta, \\
& \bar{b}_1 = q^{-\mu_0 + 1} \alpha/\beta, \quad \bar{b}_2 = q^{ -(\mu_0 + \mu_1) + \lambda_0 + 1}\alpha/\beta, \\
& \bar{b}_3 = q^{ -(\mu_0  + \mu_1 + \mu_2) + \lambda_0 + \lambda_1 + 1}\alpha/\beta, \quad \bar{b}_4 = q^{ -(\mu_0  + \mu_1 + \mu_2 + \mu_3) + \lambda_0 + \lambda_1 + \lambda_2 + 1}\alpha/\beta, \\
& \bar{b}_5 = q^{ -(\mu_0  + \mu_1 + \mu_2 + \mu_3 + \mu_4) + \lambda_0 + \lambda_1 + \lambda_2 + \lambda_3 + 1}\alpha/\beta .
\end{align*}
\ref{case:4} $(r = 0)$ :
\begin{align*}
g^{[6]}_1(x)  
&= (1 - q)^5 q^{\mu_0 - 4(\mu_1 - (\lambda_0 + \lambda_1)) - 3(\mu_2 - \lambda_2) - 2(\mu_3 - \lambda_3) - (\mu_4 - \lambda_4)} \; \alpha^{-\mu_0} \\
&\quad \cdot x^{\mu_1 + \mu_2 + \mu_3 + \mu_4 - (\lambda_0 + \lambda_1 + \lambda_2 + \lambda_3 + \lambda_4)} \\
&\quad \frac{(q, q, q, q, q ; q)_{\infty}}{(q\beta/\alpha , q^{-\lambda_1 + 1}, q^{-\lambda_2 + 1}, q^{-\lambda_3 + 1}, q^{-\lambda_4 + 1} ; q)_{\infty}}  
\;\frac{(\bar{b}_1, \bar{b}_2, \bar{b}_3, \bar{b}_4, \bar{b}_5 ; q)_{\infty}}{(\bar{a}_1, \bar{a}_2, \bar{a}_3, \bar{a}_4, \bar{a}_5 ; q)_{\infty}} \\
&\quad \cdot  \;  {}_6\phi_5 \Bigl( {q^{\lambda_0} , \bar{a}_1, \bar{a}_2, \bar{a}_3, \bar{a}_4, \bar{a}_5 
\atop{ \bar{b}_1, \bar{b}_2, \bar{b}_3, \bar{b}_4, \bar{b}_5 }} ; q, q^{6}/(\alpha x) \Bigr) ,
\end{align*}
where
\begin{align*}
& \bar{a}_1 = q^{\mu_0}, \quad \bar{a}_2 = q^{ -\mu_1 + \lambda_0 + \lambda_1}, \quad \bar{a}_3 = q^{ -(\mu_1 + \mu_2) + \lambda_0 + \lambda_1 + \lambda_2}, \\
& \bar{a}_4 = q^{ -(\mu_1 + \mu_2 + \mu_3) + \lambda_0 + \lambda_1 + \lambda_2 + \lambda_3}, \quad \bar{a}_5 = q^{ -(\mu_1 + \mu_2 + \mu_3 + \mu_4) + \lambda_0 + \lambda_1 + \lambda_2 + \lambda_3 + \lambda_4}, \\
& \bar{b}_1 = q^{\mu_0 + 1}\beta/\alpha, \quad \bar{b}_2 = q^{-\mu_1 + \lambda_0 + 1}, \quad \bar{b}_3 = q^{ -(\mu_1 + \mu_2) + \lambda_0 + \lambda_1 + 1} , \\
& \bar{b}_4 = q^{ -(\mu_1 + \mu_2 + \mu_3) + \lambda_0 + \lambda_1 + \lambda_2 + 1}, \quad \bar{b}_5 = q^{ -(\mu_1 + \mu_2 + \mu_3 + \mu_4) + \lambda_0 + \lambda_1 + \lambda_2 + \lambda_3 + 1} .
\end{align*}
\ref{case:4} $(r = 1)$ :
\begin{align*}
g^{[6]}_1(x)  
&= (1 - q)^5 q^{\mu_0 + 2\mu_1 - 3\mu_2 - 2\mu_3 - \mu_4 - \lambda_0 + 3(\lambda_1 + \lambda_2) + 2\lambda_3 + \lambda_4} \; \alpha^{-(\mu_0 + \mu_1) + \lambda_0} \\
&\quad \cdot x^{\mu_2 + \mu_3 + \mu_4 - (\lambda_1 + \lambda_2 + \lambda_3 + \lambda_4)} \\
&\quad \frac{(q, q, q, q, q ; q)_{\infty}}{(q\beta/\alpha, q^{-\lambda_0 + 1}, q^{-\lambda_2 + 1}, q^{-\lambda_3 + 1}, q^{-\lambda_4 + 1} ; q)_{\infty}}  
\;\frac{(\bar{b}_1, \bar{b}_2, \bar{b}_3, \bar{b}_4, \bar{b}_5 ; q)_{\infty}}{(\bar{a}_1, \bar{a}_2, \bar{a}_3, \bar{a}_4, \bar{a}_5 ; q)_{\infty}} \\
&\quad \cdot  \;  {}_6\phi_5 \Bigl( {q^{\lambda_1} , \bar{a}_1, \bar{a}_2, \bar{a}_3, \bar{a}_4, \bar{a}_5
\atop{\bar{b}_1, \bar{b}_2, \bar{b}_3, \bar{b}_4, \bar{b}_5 }} ; q, q^{6}/(\alpha x) \Bigr) ,
\end{align*}
where
\begin{align*}
& \bar{a}_1 = q^{\mu_0 + \mu_1 - \lambda_0}, \quad \bar{a}_2 = q^{\mu_1}, \quad \bar{a}_3 = q^{ -\mu_2 + \lambda_1 + \lambda_2}, \\
& \bar{a}_4 = q^{ -(\mu_2 + \mu_3) + \lambda_1 + \lambda_2 + \lambda_3}, \quad \bar{a}_5 = q^{ -(\mu_2 + \mu_3 + \mu_4) + \lambda_1 + \lambda_2 + \lambda_3 + \lambda_4}, \\
& \bar{b}_1 = q^{\mu_0 + \mu_1 - \lambda_0 + 1}\beta/\alpha, \quad \bar{b}_2 = q^{\mu_1 - \lambda_0 + 1}, \quad \bar{b}_3 = q^{ -\mu_2 + \lambda_1 + 1} , \\
& \bar{b}_4 = q^{ -(\mu_2 + \mu_3) + \lambda_1 + \lambda_2 + 1}, \quad \bar{b}_5 = q^{ -(\mu_2 + \mu_3 + \mu_4) + \lambda_1 + \lambda_2 + \lambda_3 + 1}.
\end{align*}
\ref{case:4} $(r = 2)$ :
\begin{align*}
g^{[6]}_1(x) 
&= (1 - q)^5 q^{\mu_0 + 2\mu_1 + 3\mu_2 - 2\mu_3 - \mu_4 - \lambda_0 - 2\lambda_1 + 2(\lambda_2 + \lambda_3) + \lambda_4} \; \alpha^{-(\mu_0 + \mu_1 + \mu_2) + \lambda_0 + \lambda_1} \\
&\quad \cdot x^{\mu_3 + \mu_4 - (\lambda_2 + \lambda_3 + \lambda_4)} \\
&\quad \cdot \frac{(q, q, q, q, q ; q)_{\infty}}{(q\beta/\alpha , q^{-\lambda_0 + 1 }, q^{-\lambda_1 + 1 }, q^{-\lambda_3 + 1 }, q^{-\lambda_4 + 1 } ; q)_{\infty}} 
\;\frac{(\bar{b}_1, \bar{b}_2, \bar{b}_3, \bar{b}_4, \bar{b}_5 ; q)_{\infty}}{(\bar{a}_1, \bar{a}_2, \bar{a}_3, \bar{a}_4, \bar{a}_5 ; q)_{\infty}} \\
&\quad \cdot {}_6\phi_5 \Bigl( {q^{\lambda_2}, \bar{a}_1, \bar{a}_2, \bar{a}_3, \bar{a}_4, \bar{a}_5 
\atop{\bar{b}_1, \bar{b}_2, \bar{b}_3, \bar{b}_4, \bar{b}_5 }} ; q, q^{6}/(\alpha x) \Bigr) , 
\end{align*}
where
\begin{align*}
& \bar{a}_1 = q^{\mu_0 + \mu_1 + \mu_2 - (\lambda_0 + \lambda_1)}, \quad \bar{a}_2 = q^{\mu_1 + \mu_2 - \lambda_1}, \quad  \bar{a}_3 = q^{\mu_2} , \\
& \bar{a}_4 = q^{-\mu_3 + \lambda_2 + \lambda_3}, \quad \bar{a}_5 = q^{-(\mu_3 + \mu_4) + \lambda_2 + \lambda_3 + \lambda_4}, \\
& \bar{b}_1 = q^{\mu_0 + \mu_1 + \mu_2 - (\lambda_0 + \lambda_1) + 1 }\beta/\alpha, \quad \bar{b}_2 = q^{\mu_1 + \mu_2 - (\lambda_0 + \lambda_1) + 1}, \\
& \bar{b}_3 = q^{\mu_2 - \lambda_1 + 1} , \quad \bar{b}_4 = q^{-\mu_3 + \lambda_2 + 1}, \quad \bar{b}_5 = q^{-(\mu_3 + \mu_4) + \lambda_2 + \lambda_3 + 1} .
\end{align*}
\ref{case:4} $(r = 3)$ :
\begin{align*}
g^{[6]}_1(x) 
&= (1 - q)^5 q^{\mu_0 + 2\mu_1 + 3\mu_2 + 4\mu_3 - \mu_4 - \lambda_0 - 2\lambda_1 - 3\lambda_2 + \lambda_3 + \lambda_4} \; \alpha^{-(\mu_0 + \mu_1 + \mu_2 + \mu_3) + \lambda_0 + \lambda_1 + \lambda_2} \\
&\quad \cdot x^{\mu_4 -(\lambda_3 + \lambda_4)} \\
&\quad \cdot \frac{(q, q, q, q, q ; q)_{\infty}}{(q\beta/\alpha , q^{-\lambda_0  + 1} , q^{-\lambda_1 + 1}, q^{-\lambda_2 + 1}, q^{-\lambda_4 + 1} ; q)_{\infty}}
\;\frac{(\bar{b}_1, \bar{b}_2, \bar{b}_3, \bar{b}_4, \bar{b}_5 ; q)_{\infty}}{(\bar{a}_1, \bar{a}_2, \bar{a}_3, \bar{a}_4, \bar{a}_5 ; q)_{\infty}} \\
&\quad \cdot {}_6\phi_5 \Bigl( {q^{\lambda_3}, \bar{a}_1, \bar{a}_2, \bar{a}_3, \bar{a}_4, \bar{a}_5 
\atop{\bar{b}_1, \bar{b}_2, \bar{b}_3, \bar{b}_4, \bar{b}_5 }} ; q, q^{6}/(\alpha x) \Bigr) , 
\end{align*}
where
\begin{align*}
& \bar{a}_1 = q^{\mu_0 + \mu_1 + \mu_2 + \mu_3 - (\lambda_0 + \lambda_1 + \lambda_2) }, \quad \bar{a}_2 = q^{\mu_1 + \mu_2 + \mu_3 - (\lambda_1 + \lambda_2) }, \\
& \bar{a}_3 = q^{\mu_2 + \mu_3 - \lambda_2}, \quad \bar{a}_4 = q^{ \mu_3}, \quad \bar{a}_5 = q^{-\mu_4 + \lambda_3 + \lambda_4}, \\
& \bar{b}_1 = q^{\mu_0 + \mu_1 + \mu_2 + \mu_3 - (\lambda_0 + \lambda_1 + \lambda_2) + 1}\beta/\alpha , \quad \bar{b}_2 = q^{\mu_1 + \mu_2 + \mu_3 - (\lambda_0 + \lambda_1 + \lambda_2) + 1}, \\
& \bar{b}_3 = q^{\mu_2 + \mu_3 - (\lambda_1 + \lambda_2) + 1}, \quad \bar{b}_4 = q^{\mu_3 - \lambda_2 + 1}, \\
& \bar{b}_5 = q^{-\mu_4 + \lambda_3 + 1}. 
\end{align*}
\ref{case:4} $(r = 4)$ :
\begin{align*}
g^{[6]}_1(x) 
&= (1 - q)^5 q^{\mu_0 + 2\mu_1 + 3\mu_2 + 4\mu_3 + 5\mu_4 - (\lambda_0 + 2\lambda_1 + 3\lambda_2 + 4\lambda_3)} \; \alpha^{-(\mu_0 + \mu_1 + \mu_2 + \mu_3 + \mu_4) + \lambda_0 + \lambda_1 + \lambda_2 + \lambda_3} \\
&\quad \cdot x^{ -\lambda_4} \; \frac{(q, q, q, q, q ; q)_{\infty}}{(q\beta/\alpha , q^{-\lambda_0  + 1} , q^{-\lambda_1 + 1}, q^{-\lambda_2 + 1}, q^{-\lambda_3 + 1} ; q)_{\infty}}
\;\frac{(\bar{b}_1, \bar{b}_2, \bar{b}_3, \bar{b}_4, \bar{b}_5 ; q)_{\infty}}{(\bar{a}_1, \bar{a}_2, \bar{a}_3, \bar{a}_4, \bar{a}_5 ; q)_{\infty}} \\
&\quad \cdot {}_6\phi_5 \Bigl( {q^{\lambda_4}, \bar{a}_1, \bar{a}_2, \bar{a}_3, \bar{a}_4, \bar{a}_5 
\atop{\bar{b}_1, \bar{b}_2, \bar{b}_3, \bar{b}_4, \bar{b}_5 }} ; q, q^{6}/(\alpha x) \Bigr) , 
\end{align*}
where
\begin{align*}
& \bar{a}_1 = q^{\mu_0 + \mu_1 + \mu_2 + \mu_3 + \mu_4 - (\lambda_0 + \lambda_1 + \lambda_2 + \lambda_3) }, \quad \bar{a}_2 = q^{\mu_1 + \mu_2 + \mu_3 + \mu_4 - (\lambda_1 + \lambda_2 + \lambda_3) }, \\
& \bar{a}_3 = q^{\mu_2 + \mu_3 + \mu_4 - (\lambda_2 + \lambda_3)}, \quad \bar{a}_4 = q^{\mu_3 + \mu_4 - \lambda_3}, \quad \bar{a}_5 = q^{\mu_4 }, \\
& \bar{b}_1 = q^{\mu_0 + \mu_1 + \mu_2 + \mu_3 + \mu_4 - (\lambda_0 + \lambda_1 + \lambda_2 + \lambda_3) + 1}\beta/\alpha , \quad \bar{b}_2 = q^{\mu_1 + \mu_2 + \mu_3 + \mu_4 - (\lambda_0 + \lambda_1 + \lambda_2 + \lambda_3) + 1}, \\
& \bar{b}_3 = q^{\mu_2 + \mu_3 + \mu_4 - (\lambda_1 + \lambda_2 + \lambda_3) + 1}, \quad \bar{b}_4 = q^{\mu_3 + \mu_4 - (\lambda_2 + \lambda_3) + 1}, \quad \bar{b}_5 = q^{\mu_4 - \lambda_3 + 1}. 
\end{align*}


\begin{thebibliography}{9999}
\bibitem{Ar}
Y.~Arai, Solutions to $q$-hypergeometric equations associated with $q$-middle convolution, {\sf arXiv:2403.02662v2}.
\bibitem{AT}
Y.~Arai, K.~Takemura, On $q$-middle convolution and $q$-hypergeometric equations, {\textit{SIGMA}} \textbf{19} (2023), 037, 40 pages.
\bibitem{AT2}
Y.~Arai, K.~Takemura, Reformulation of $q$-middle convolution and applications, {\textit{SIGMA}} \textbf{22} (2026), 040, 43 pages.
\bibitem{DR1}
M.~Dettweiler, S.~Reiter, An algorithm of Katz and its application to the inverse Galois problem. Algorithmic methods in Galois theory, {\textit{J. Symbolic Comput.}} \textbf{30} (2000), 761--798.
\bibitem{DR2}
M.~Dettweiler, S.~Reiter, Middle convolution of Fuchsian systems and the construction of rigid differential systems, {\textit{J. Algebra}} \textbf{318} (2007), 1--24.
\bibitem{GR}
G.~Gasper, M.~Rahman, {\textit{Basic Hypergeometric Series}}, Second edition, Encyclopedia of Mathematics and its Applications, \textbf{96}, Cambridge university press, 2004.
\bibitem{Katz}
N.~M.~Katz, Rigid local systems, Princeton University Press, 1996.
\bibitem{Ohy}
Y.~Ohyama, Connection formula of basic hypergeometric series ${}_r\phi_{r-1}(0;b;q,x)$, {\textit{J. Math. Tokushima Univ.}} \textbf{51} (2017), 29--36.
\bibitem{Ro}
J.~Roques, Generalized basic hypergeometric equations, {\textit{Invent. Math.}} \textbf{184} (2011), 499--528.
\bibitem{SY} 
H.~Sakai, M.~Yamaguchi, Spectral types of linear $q$-difference equations and $q$-analog of middle convolution, {\textit{Int. Math. Res. Not.}} \textbf{2017} (2017), 1975--2013.
\end{thebibliography}
\end{document}